\documentclass{article}

\RequirePackage{amsthm,amsmath,amsfonts,amssymb}
\RequirePackage[authoryear]{natbib}
\RequirePackage{graphicx}

\usepackage[usenames,dvipsnames]{xcolor}
\definecolor{mygreen}{rgb}{.02,.60,.30}

\usepackage{hyperref}
\hypersetup{pdfpagemode=FullScreen,  
	colorlinks=true,
	citecolor=Bittersweet,
	linkcolor=Bittersweet,
	urlcolor=Bittersweet
}

\usepackage{listings}
\numberwithin{equation}{section}
\theoremstyle{plain}

\newtheorem{thm}{Theorem}[section]
\newtheorem{lem}{Lemma}[section]
\newtheorem{prop}{Proposition}[section]
\newtheorem{co}{Corollary}[section]
\theoremstyle{remark}

\newcommand{\Epred}{E_n^{\mathrm{pred}}}

\begin{document}

\title{Stopping on the last success with unknown odds: asymptotic minimax optimality of the plug-in rule}


			\author{Davy Paindaveine\footnote{Davy.Paindaveine@ulb.be}}
\date{Universit\'{e} libre de Bruxelles}

\maketitle

\begin{abstract}
We study the last-success problem for sequential Bernoulli trials in the homogeneous setting where \(X_1,\ldots,X_n\) are i.i.d. \(\mathrm{Bernoulli}(p)\), with unknown \(p\in(0,1)\). For known \(p\), Bruss' sum-the-odds theorem gives an optimal threshold rule with win probability \(V_n(p)\); for unknown \(p\), the odds driving this threshold must be learned online from the same sequence on which one is trying to stop. We analyze the resulting statistical decision problem over all \(p\)-blind rules, and write \(W_n(p)\) for the win probability of the natural plug-in odds rule. Our main result is an exact asymptotic minimax theorem: for any \(p_0\in(0,\tfrac12)\), the limit of \(\sqrt n\,\inf_\pi\sup_{p\in[p_0,1)}\{V_n(p)-W_n^\pi(p)\}\), where the infimum is over all possibly randomized \(p\)-blind rules, is \(C_\star=\tfrac12\sup_{u>0}u\Phi(-u)=0.08498\ldots\), with \(\Phi\) denoting the standard normal distribution function. The same constant is attained by the plug-in rule, which is therefore asymptotically minimax optimal. The result is local in nature: at each transition point \(p=1/k\), where the oracle threshold jumps, the deficit has an exact local minimax constant proportional to \(\gamma_k=(1-\tfrac1k)^{k-2}\{k^{-1}(1-k^{-1})\}^{1/2}\), and the global least favourable point is \(k=2\). Thus the root-\(n\) barrier is caused not by estimating \(p\) itself, but by the discontinuity of the oracle action. We also quantify the price of sample splitting: estimating \(p\) on an initial fraction \(a\) of the horizon and then freezing the estimate is rate-optimal but inflates the sharp constant by \(1/\sqrt a\). Finally, in sparse regimes \(p=p_n\to0\) with \(np_n\to\infty\), the plug-in rule is asymptotically oracle-optimal, and the critical window \(p\asymp1/n\) is a genuine barrier: no \(p\)-blind rule can converge uniformly to the oracle win probability over all \(p\in(0,1)\).
\end{abstract}

%
%



\section{Introduction}
\label{sec:intro}

Optimal stopping problems lie at the interface of probability and sequential
decision theory: decisions must be taken online, on the basis of the
observations revealed so far, with no access to future outcomes; we refer to
\cite{PeskirShiryaev2006} for a broad treatment. A canonical
illustration is the \emph{parking problem} of \cite{MacQueenMiller1960}; see
\citet[Section~2.5]{FergusonBook} for a textbook treatment. A driver travels
along a one-way street towards her destination, each parking place being free
with probability~$p$, independently across places. Places are inspected one at a
time and, at each free place, she must decide irrevocably whether to park or to
drive on, a place that has been passed being lost forever. In the original
formulation, the driver minimizes the expected distance between the place she
takes and her destination, and an optimal rule is a threshold rule: drive on
until a prescribed number of places from the destination, then take the first
free place; see \cite{Tamaki1988} for the variant in which U-turns are allowed.
Here we consider the natural variant in which no reward function is
specified: the driver simply wishes to park at the \emph{last free
place}---the best outcome achievable on the realized configuration. Writing~$X_t$
for the indicator that place~$t$ is free, this is exactly the problem of stopping
on the last success in a sequence of Bernoulli trials.

This \emph{last-success problem}, in which one observes independent Bernoulli
trials~$X_1,\ldots,X_n$ with possibly distinct success probabilities
$p_1,\ldots,p_n$ and wishes to stop exactly on the final success, is one of the
most classical instances of optimal stopping. When these success probabilities
are known, it admits an elegant and complete solution through Bruss'
\emph{sum-the-odds} theorem \citep[Theorem~1]{Bruss2000}: summing the odds
$p_t/(1-p_t)$ backwards from~$t=n$ until the total first reaches one determines a
threshold~$s_n$, and stopping at the first success (if any) occurring at or
after~$s_n$ is optimal.
The sum-the-odds theorem, which also provides a closed-form expression for the
corresponding optimal win probability, has become a cornerstone of the
last-success literature. A large body of work---of which we cite only a few representative
references---develops extensions and refinements of the model and of the
resulting optimal threshold rules, including sharper bounds and variants of the
theorem \citep{Bruss2003,Ferguson2011,GrauRibas2020NoteLastSuccess},
Markov-dependent trials \citep{HsiauYang2002}, multiple-choice formulations in
which one is allowed up to $\ell$ stopping chances
\citep{AnoKakinumaMiyoshi2010}, versions in which one aims to stop on any of the
last $\ell$ successes \citep{Tamaki2010}, problems in which exactly $k$ among
the last $\ell$ successes should be selected \citep{MatsuiAno2017}, and stopping
on the $\ell$th last success \citep{BrussPaindaveine2000}. Further variations
include the group-interview secretary variant of
\cite{HsiauYang2000NaturalVariation}, trapping the ultimate success
\citep{GnedinDerbazi2021TrappingUltimateSuccess}, and random observation times
\citep{GnedinDerbazi2025RandomObsTimes}.
A common feature of all of this work is that the success probabilities are
treated as \emph{known} to the decision maker.

In most applications, however, the success probabilities are not available, so
that the oracle rule cannot be implemented. This turns the last-success problem
into a genuinely \emph{statistical} decision problem: the odds that drive the
optimal threshold must be estimated from the very same sequence on which one is
trying to stop, and there is no separate training phase. In his survey of the odds
theorem, \citet[Section~2.1]{Dendievel2012} describes this unknown-odds case as
an open problem, and ``an important one regarding applications, for it is often
closer to reality than the model in which we assume that we know the
parameters''. It has nonetheless received surprisingly little attention. A notable contribution is
\cite{BrussLouchard2009}, which studies an odds-type algorithm based on
sequential updating and plug-in/empirical-odds ideas. More recent work
investigates variants in which the decision maker receives auxiliary
information, for instance $m$ preliminary samples from each Bernoulli
distribution \citep{YoshinagaKawase2024}; in the special case $m=1$, this
setting connects, via a reduction that treats the no-success scenario as a win,
to the adversarial-order single-sample secretary problem of
\cite{NutiVondrak2023}, whose upper bound shows the resulting guarantee to be
best possible. They also show that no policy can guarantee a winning
probability of exactly~$1/e$---the value that Bruss' rule guarantees when the
odds are known---from any finite number of samples. These
works nonetheless leave the core statistical questions open: the
sample-augmented model of \cite{YoshinagaKawase2024} relies on side information
absent from the standard online observation model, while
\cite{BrussLouchard2009}, though it develops and analyzes sequential-updating
rules, provides neither sharp finite-horizon worst-case comparisons to the
oracle nor an identification of the regimes in which oracle approximation is
possible---let alone a decision-theoretic account of whether any oracle-free
rule can be optimal.

\subsection{Setting and the plug-in rule}
\label{sec:setting}

The present paper takes up precisely this decision-theoretic viewpoint. We
consider the last-success problem in the minimal-information model in which only
the sequential Bernoulli outcomes are observed. In this model, it is of course
impossible to estimate the full collection of success probabilities without
structural assumptions, so we work in the homogeneous setting in which
$X_1,\ldots,X_n$ are i.i.d.\ $\mathrm{Bernoulli}(p)$ for some unknown
$p\in(0,1)$ (we write $\mathbb P_p$ for the corresponding probability measure).
This is the canonical such assumption---it is the one under which the parking
problem was described above---and arguably the case to settle
first.\footnote{More generally, \cite{BrussLouchard2009} considers
$p_t=f_tp\in[0,1]$ for $t=1,\ldots,n$, with known coefficients~$f_t$ and a single
unknown~$p$. We expect the rate results of Section~\ref{sec:oracle_bounds} to
extend to that framework under mild regularity assumptions on~$(f_t)$. The sharp
constants of Section~\ref{sec:constants} are more delicate: in the homogeneous
case, the oracle threshold depends on~$p$ only through~$\lceil1/p\rceil$, and it
is this discrete structure that generates the constants, so that these would have
to be recomputed for a general sequence~$(f_t)$.} If $p$ were known, the sum-the-odds theorem states
that an optimal rule in this homogeneous setting is the threshold rule stopping at the first success time
$t\ge s_n(p)$ (if any), where
\begin{equation}
\label{eq:sn_homo}
s_n(p)
:=
\max\{1,\ n-m(p)+1\},
\qquad
m(p):=\bigg\lceil \frac1p\bigg\rceil-1,
\end{equation}
with associated oracle win probability
\begin{equation}
\label{WinOracle}
V_n(p)
=
\mathbb P_p\bigl[\textstyle\sum_{t=s_n(p)}^{n}X_t=1\bigr]
=
(n-s_n(p)+1)\,p\,(1-p)^{\,n-s_n(p)}.
\end{equation}
An oracle-free rule must be \emph{$p$-blind}: it may depend on the observed data
(and possibly on additional internal randomization) but not on $p$. For such a
rule $\pi$, with associated stopping time $\tau_n^\pi$, we write
\begin{equation}
\label{eq:Wpi_def}
W_n^\pi(p)
:=
\mathbb P_p\bigl[\tau_n^\pi\le n,\ X_{\tau_n^\pi}=1,\
X_{\tau_n^\pi+1}=\cdots=X_n=0\bigr]
\end{equation}
for its win probability under $\mathbb P_p$, i.e.\ the probability that $\pi$
stops exactly on the last success. The canonical $p$-blind rule is the
\emph{plug-in} (odds) rule $\hat\pi$, which replaces $p$ online by the running
empirical estimate $\hat p_t:=S_t/t$, where $S_t:=\sum_{i=1}^{t}X_i$, and applies
the oracle decision with $p=\hat p_t$ at each time $t$. It is easy to check that it corresponds to the
stopping time
\begin{equation}
\label{StoppingTimePlugin}
\hat\tau_n
:=
\inf\Bigl\{t\in\{1,\dots,n\}:\ X_t=1
\ \text{ and }\
\bigl(\hat p_t<\tfrac{1}{n-t+1}\ \text{ or }\ t=n\bigr)\Bigr\},
\end{equation}
with $\inf\varnothing:=+\infty$; we abbreviate its win probability as
$W_n(p):=W_n^{\hat\pi}(p)$. Our goal is to compare $W_n(p)$ with the oracle
benchmark $V_n(p)$ and, more broadly, to quantify the intrinsic limits of
oracle-freeness over the class of \emph{all} $p$-blind rules.

Before proceeding, we show a basic yet important structural feature of the plug-in rule (throughout, the plug-in rule refers to the rule based on the stopping time~$\hat\tau_n$  in~(\ref{StoppingTimePlugin})). On~$\{\hat\tau_n\le n-1\}$, we have 
\[
X_{\hat\tau_n}=1
\quad (\textrm{so } S_{\hat\tau_n}\ge1)
\quad
\textrm{ and } 
\quad
\frac{1}{n-\hat\tau_n+1}>\hat p_{\hat\tau_n}=\frac{S_{\hat\tau_n}}{\hat\tau_n}\ge \frac{1}{\hat\tau_n}
,
\]
which, since~$\hat\tau_n$ takes integer values, 
implies that
\begin{equation}
\label{noearlystop}
\hat\tau_n\ \ge\ \Big\lceil\frac{n}{2}\Big\rceil+1
\qquad\text{almost surely.}
\end{equation}
This “second-half” constraint will play a key role in our analysis: in particular, it guarantees that when the plug-in rule stops,
it does so based on a relatively stable estimate of~$p$.

Figure~\ref{Fig1} illustrates the resulting performance: its left panel plots the
plug-in win probability~$W_n(p)$ together with the oracle benchmark~$V_n(p)$
of~(\ref{WinOracle}) for~$n=5,10,30$, and its right panel the deficit
$V_n(p)-W_n(p)$. The deficit is small over a wide range of~$p$, even at the
moderate horizons considered here. The kinks at~$p=1/k$ reflect the
non-differentiability of~$V_n(p)$ at these points: as~$p$ decreases
through~$1/k$, the number~$m(p)$ of observations on which the oracle rule acts
jumps from~$k-1$ to~$k$. The plug-in win probability, by contrast, is smooth---it
is in fact a polynomial in~$p$, as follows from the dynamic-programming recursion
of Appendix~\ref{secExactFiniteHorizon}, through which it is
computed exactly here.

\begin{figure}[h!]
\hspace*{-9mm}
\includegraphics[width=1.075\textwidth]{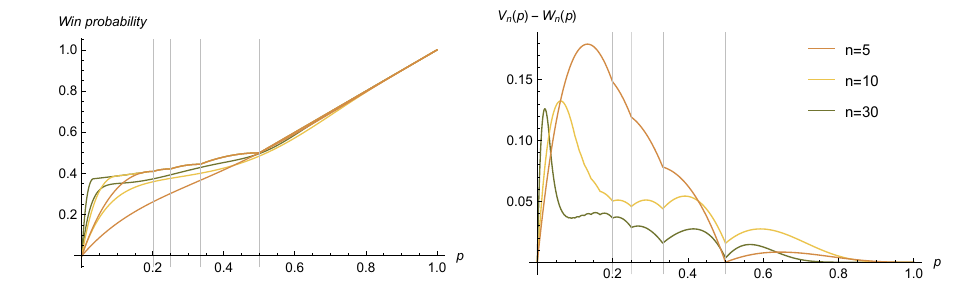}
\caption{%
(Left panel:) plug-in win probability $W_n(p)$ and oracle win probability
$V_n(p)$ as functions of $p\in(0,1)$, for $n\in\{5,10,30\}$ (for each $n$, the
upper curve is $V_n(p)$, since $V_n(p)\ge W_n(p)$). (Right panel:) the deficit
$V_n(p)-W_n(p)$ for $n\in\{5,10,30\}$. Vertical gray lines mark the
non-differentiability points of $V_5(p)$, i.e.\ $p=1/k$ with $k=2,3,4,5$.}
\label{Fig1}
\end{figure}

\subsection{Contributions}
The central contribution of this paper is that the plug-in rule is
\emph{asymptotically minimax optimal}, with an exact constant. Fix
$p_0\in(0,\tfrac12)$,\footnote{The restriction $p_0<\tfrac12$ is necessary: for
$p\ge\tfrac12$ the oracle stops only at the terminal step~$n$ and is matched by a
trivial \mbox{$p$-blind} rule, so no non-trivial lower bound can hold
on~$[p_0,1)$ for $p_0\ge\tfrac12$.} let $\Phi$ denote the standard normal
distribution function, and set
\begin{equation}
\label{eq:intro_Cstar}
C_\star
:=
\frac12\sup_{u>0}u \Phi(-u)
=
0.08498\ldots
.
\end{equation}
We prove that
\begin{equation}
\label{eq:intro_sharp}
\lim_{n\to\infty}\sqrt n\,
\sup_{p\in[p_0,1)}\bigl(V_n(p)-W_n(p)\bigr)
\ =\
C_\star
\ =\
\lim_{n\to\infty}\sqrt n\,
\inf_\pi\sup_{p\in[p_0,1)}\bigl(V_n(p)-W_n^\pi(p)\bigr)
,
\end{equation}
the infimum being taken over all (possibly randomized) \mbox{$p$-blind} rules;
the left-hand equality is Theorem~\ref{thm:plugin_sharp} and the right-hand one
is Theorem~\ref{thm:minimax_sharp}. In words, the worst-case oracle deficit of
the plug-in rule and the minimax risk of the problem share the \emph{same} exact
asymptotic constant: no oracle-free rule, however sophisticated and even allowing
randomization, can improve on the simple empirical-odds prescription---not even
by a constant factor.

Reaching~(\ref{eq:intro_sharp}) requires first settling the rate, which is the
object of Section~\ref{sec:oracle_bounds}. We show there that
\begin{equation}
\label{eq:intro_minimax}
\inf_\pi\ \sup_{p\in[p_0,1)}\bigl(V_n(p)-W_n^\pi(p)\bigr)
\ \asymp\
\frac{1}{\sqrt n},
\end{equation}
the upper bound being attained by the plug-in rule
(Theorem~\ref{thm:finite_oracle_ineq_p0}) and the matching lower bound holding
for \emph{every} \mbox{$p$-blind} rule (Theorem~\ref{thmLowerBound}); in
particular, the root-$n$ rate is an intrinsic feature of the unknown-$p$
last-success problem rather than an artifact of the plug-in choice. The two
levels of description are complementary rather than redundant: the
finite-horizon, pointwise bounds behind~(\ref{eq:intro_minimax}) are the ones
that apply at a fixed~$n$, that cover every $p_0\in(0,1)$, and that feed the
later sparse-regime analysis, whereas~(\ref{eq:intro_sharp}) describes the worst
case in the limit.

We are not aware of a previous exact asymptotic minimax analysis for a
sequential optimal-stopping problem of this type, and
achieving it raises a difficulty that is specific to stopping problems. In
standard estimation, a concentration inequality $|\hat p-p|\le\varepsilon$
translates, through smoothness of the target functional, into an
$O(\varepsilon)$ risk bound. Here, by contrast, the quantity of interest---a
win or a loss---is not a smooth functional of $\hat p_t$ but a \emph{global}
event determined by the entire trajectory of the stopping rule: the rule wins
only if its stopping time lands exactly on the last success. Moreover, the oracle
win probability $V_n(p)$ is only piecewise smooth, with kinks at the transition points
$p=1/k$ where the optimal threshold~$s_n(p)$ jumps. Near such a point, a small
error in $\hat p_t$ can push the plug-in rule across the threshold
$1/(n-t+1)$ and make it act on the ``wrong'' horizon. The crux of the analysis
is therefore to quantify precisely how a Hoeffding-type deviation bound on
$\hat p_t$ propagates to the win probability of the induced stopping rule. We
show that the deficit $V_n(p)-W_n(p)$ decays exponentially for every fixed $p$,
but that the worst case over $[p_0,1)$ is governed by parameters lying at
distance of order $1/\sqrt n$ from a transition point $1/k$---close enough that
no $p$-blind rule can reliably tell which side of the threshold it is
on.\footnote{\label{fn:transition}It may seem surprising that the least favourable parameters are not
the transition points themselves, since these are precisely where a
\mbox{$p$-blind} rule is most likely to misjudge on which side of the threshold
$p$ lies. The reason is that the cost of misjudging vanishes there: at $p=1/k$,
the oracle is indifferent between stopping on the first success among the last
$k-1$ observations and doing so among the last $k$, both thresholds achieving the
oracle win probability~$V_n(p)$, so that either choice is harmless. We comment
further on this after Theorem~\ref{thm:finite_oracle_ineq_p0}.} This
is exactly where the root-$n$ barrier originates, and the matching lower bound
turns this indistinguishability into a quantitative loss via a two-point
(Le~Cam--Pinsker) argument at $p=\tfrac12\pm\tfrac{h}{\sqrt n}$. In fact, the
analysis is local at each transition point separately: Proposition~\ref{prop:local_at_1k}
shows that, on the local scale $p=\tfrac1k+u\sigma_k/\sqrt n$---where
$\sigma_k$ denotes the standard deviation of a Bernoulli variable with success
probability~$\tfrac1k$---the plug-in rule has limiting risk
$\gamma_k|u|\Phi(-|u|)$, for an explicit constant~$\gamma_k$ given in
Section~\ref{sec:constants}, and that the corresponding local minimax constant is
$\gamma_k\sup_{u>0}u\Phi(-u)=2\gamma_k C_\star$. The global constant is
the maximum of these over~$k$, attained at~$k=2$; the exact constant is thus not
a phenomenon attached to $p=\tfrac12$, but the largest term in a complete local
theory of the oracle discontinuities.

We complement these core results along three further axes.

\begin{itemize}

\item[(a)] \emph{The price of sample splitting.}
Section~\ref{sec:splitting} studies the natural alternative that estimates~$p$ on
an initial fraction~$a$ of the horizon, then applies the sum-the-odds rule with
this frozen estimate over the remaining times. Such rules admit a closed-form
win probability, remain minimax rate-optimal, and satisfy the exact analogue
of~(\ref{eq:intro_minimax}) with~$n$ replaced by~$an$
(Theorem~\ref{thm:oracle_split}). Comparing them with the plug-in rule, however,
is a matter of exact constants and is therefore beyond the reach of such rate
statements: it is the sharp asymptotics of Theorem~\ref{thm:split_sharp}, which
produce the very same constant~$C_\star$ on the~$1/\sqrt{an}$ scale, that make
the worst-case deficit of~$\pi_a$ larger by the exact factor~$1/\sqrt a$
(Corollary~\ref{co:ratio}). Sequential updating is thus strictly preferable, and
the value of the discarded information is quantified exactly.
\vspace{1mm}

\item[(b)] \emph{Sparse regimes and maximal uniform convergence.}
In the sparse regime where $p=p_n\to0$ with $np_n\to\infty$, we prove that the
plug-in rule is asymptotically oracle-optimal, in the sense that
$V_n(p_n)-W_n(p_n)\to0$, together with an explicit rate for the deficit
(Theorem~\ref{thm:plugin-eff-sparse}); here the crude Hoeffding control used in
the non-sparse regime is replaced by a variance-sensitive martingale deviation
inequality. Combining the sparse and non-sparse bounds yields a uniform
convergence statement on $(0,\tilde p_n]\cup[p_n,1)$ whenever $n\tilde p_n\to0$
(Theorem~\ref{thm:uniform-etan}), which we show cannot be extended to the
critical window $p\asymp1/n$.
\vspace{1mm}

\item[(c)] \emph{A global barrier.}
We prove that \emph{no} sequence of oracle-free (possibly randomized) rules can
converge uniformly to the oracle win probability over all $p\in(0,1)$
(Theorem~\ref{thm:no_uniform_abs}), identifying a genuine and rule-independent
barrier to global uniform oracle approximation. Together with~(b), this shows
that the plug-in rule already achieves the largest uniform convergence any oracle-free
rule can achieve, so that its failure to be uniform over the whole of~$(0,1)$
reflects the hardness of the problem rather than a defect of the rule. Combined with the minimax optimality
of Section~\ref{sec:constants}, this makes the plug-in rule a thoroughly
satisfactory answer to the unknown-$p$ last-success problem: it is optimal where
optimality is possible, and where it fails, so does everything else.
\vspace{1mm}

\end{itemize}

For the sake of completeness, Appendix~\ref{secExactFiniteHorizon} also provides an exact
finite-horizon analysis of the plug-in rule.
Theorem~\ref{thm:exact_Wn_plugin} gives a
representation of~$W_n(p)$ through a dynamic-programming recursion, yielding
the~$O(n^2)$ evaluation scheme on which Figure~\ref{Fig1} rests. The
unknown-$p$ formulation moreover exhibits finite-horizon decision-theoretic
obstructions of independent interest---a strict separation~$W_n(p)<V_n(p)$ for
all~$n\ge6$ and~$p\in(0,1)$, and the nonexistence of a uniformly optimal
\mbox{$p$-blind} rule, even allowing randomization.

\subsection{Organization}

Section~\ref{sec:oracle_bounds} settles the rate: an oracle inequality for the plug-in rule and a matching minimax lower bound show that, away from
sparsity, the best achievable worst-case deficit is of exact order~$1/\sqrt n$. Section~\ref{sec:constants} sharpens both
into exact constants and proves that the plug-in rule is asymptotically minimax
optimal. Section~\ref{sec:splitting} quantifies the cost of the sample-splitting
alternative. Section~\ref{sec:sparse} treats the sparse regime, the maximal
uniform-convergence result, the $p\asymp1/n$ barrier, and the global
impossibility. Section~\ref{sec:conclu} wraps up and provides perspectives for future research.
The appendices contain the complementary results and all the proofs.
Appendix~\ref{secExactFiniteHorizon} develops the exact finite-horizon analysis
of the plug-in rule, and Appendix~\ref{secAuxiliaryResults} gathers the auxiliary
results together with the finite-horizon decision-theoretic obstructions.
Appendices~\ref{appProofsRates}--\ref{secProofsSection6} then collect the proofs
of Sections~\ref{sec:oracle_bounds}--\ref{sec:sparse}, in that order.


\section{Minimax rate optimality away from sparsity}
\label{sec:oracle_bounds}

At a fixed horizon, no
\mbox{$p$-blind} rule can be uniformly optimal: the dominance partial order on
such rules---where~$\pi$ dominates~$\pi'$ whenever $W_n^\pi(p)\ge W_n^{\pi'}(p)$
for all~$p\in(0,1)$---has no greatest element for any~$n\ge2$, even if one allows
randomization (Theorem~\ref{thm:no_uniform_opt}). The natural objective is
therefore to control the oracle gap~$V_n(p)-W_n^\pi(p)$ uniformly over
$p$ in a meaningful range, rather than to seek a rule that is exactly optimal at
every~$p$ at once.
We first consider the non-sparse regime $p\in[p_0,1)$, with $p_0>0$ fixed, and
discuss what can be achieved in terms of rate. The following result establishes a
sharp oracle inequality for the plug-in rule there, with worst-case rate of
order $1/\sqrt{n}$. Its proof is where the deviation
bounds on~$\hat p_t$ are propagated to the win probability---the step
identified in Section~\ref{sec:intro} as the crux of the analysis: it proceeds by
introducing a \emph{predictable switch time} at which the rule's threshold
crosses~$p$, and by comparing the conditional and unconditional win probabilities
on either side of it. 

\begin{thm}
\label{thm:finite_oracle_ineq_p0}
Fix $p_0\in(0,1)$. Let
$
M_0:=m(p_0)=\lceil \frac{1}{p_0}\rceil-1$
and 
$$
\mathcal B_{p_0}
:=
\bigg\{\frac{1}{k}: k=2,\dots,M_0+1\bigg\}
.
$$
Let $\Delta_p:=\mathrm{dist}(p,\mathcal B_{p_0})=\min\{|p-q|:q\in \mathcal B_{p_0}\}$. Then, (i) there exist positive constants $C_1(p_0),C_2(p_0)$, and $c(p_0)$ such that for all $n\ge 2$ and all $p\in[p_0,1)$,
\begin{equation}
\label{eq:p0_pointwise_distance}
V_n(p)-W_n(p)
\le 
C_1(p_0)
\Delta_p
e^{-c(p_0)n\Delta_p^2}
 +
C_2(p_0)
e^{-c(p_0)n}
.
\end{equation}
(ii) For~$p_0>\frac{1}{2}$, there exist positive constants $C(p_0),c(p_0)$ such that for all $n\ge 2$,
\begin{equation}
\label{eq:p0_uniform_exp_term}
\sup_{p\in[p_0,1)}\bigl(V_n(p)-W_n(p)\bigr)
\le 
C(p_0)e^{-c(p_0)n}
,
\end{equation}
whereas for~$p_0\leq \frac{1}{2}$,  there exists a positive constant~$C(p_0)$ such that for all $n\ge 2$,
\begin{equation}
\label{eq:p0_uniform_two_term}
\sup_{p\in[p_0,1)}\bigl(V_n(p)-W_n(p)\bigr)
\le 
\frac{C(p_0)}{\sqrt{n}}
.
\end{equation}
\end{thm}


In Figure~\ref{Fig2}, the left panel illustrates how the deficit $V_n(p)-W_n(p)$ depends on the horizon when $p$ is fixed. The curves become eventually close to linear in~$n$ on the logarithmic $y$-axis used there, which is consistent with an exponential decay of the deficit for fixed~$p$ (with a rate that can vary substantially with~$p$). The right panel shows the dependence on~$n$ of the worst-case deficit over~$p\in[p_0,1)$ for $p_0\in\{0.2,0.3,0.4\}$  and suggests a root-$n$ scaling, as the graphs of~$n\mapsto  \sqrt{n}\,\sup_{p\in[p_0,1)}(V_n(p)-W_n(p))$ appear to stabilize as~$n$ grows.
This motivates the $1/\sqrt{n}$ oracle upper bound~(\ref{eq:p0_uniform_two_term}) and the matching minimax lower bounds proved in Theorem~\ref{thmLowerBound} below. Overall, the figure highlights a marked gap between pointwise and uniform behavior: although the deficit decays much faster for any fixed~$p$, the worst-case deficit on~$[p_0,1)$ is governed by a genuinely hardest region that enforces the root-$n$ rate.

The pointwise bound~(\ref{eq:p0_pointwise_distance}) locates that region, and completes the discussion opened in Footnote~\ref{fn:transition} of the Introduction. Its two factors pull in opposite directions. The factor~$\Delta_p$ is the cost of misjudging on which side of a boundary the parameter lies: as~$p$ approaches a point of~$\mathcal B_{p_0}$, the two adjacent thresholds become equally good---at~$p=1/k$ both are oracle-optimal---so that an incorrect decision costs nothing. The factor~$e^{-c(p_0)n\Delta_p^2}$, on the other hand, is the probability of such a misjudgement, and it decays as~$p$ moves away from~$\mathcal B_{p_0}$ and the empirical odds separate the two sides. Parameters much closer to~$\mathcal B_{p_0}$ than~$1/\sqrt n$ may therefore be misclassified, but harmlessly so, whereas parameters much farther are classified correctly with overwhelming probability; the deficit is largest where the two effects balance. Maximizing~$\Delta\mapsto\Delta e^{-c(p_0)n\Delta^2}$ over~$\Delta\ge0$ locates this balance at~$\Delta=1/\sqrt{2c(p_0)n}$ and gives the maximal value~$1/\sqrt{2ec(p_0)n}$: this is where the~$1/\sqrt n$ rate of~(\ref{eq:p0_uniform_two_term}) comes from.

\begin{figure}[h!]
\hspace*{-2mm}\includegraphics[width=1.02\textwidth]{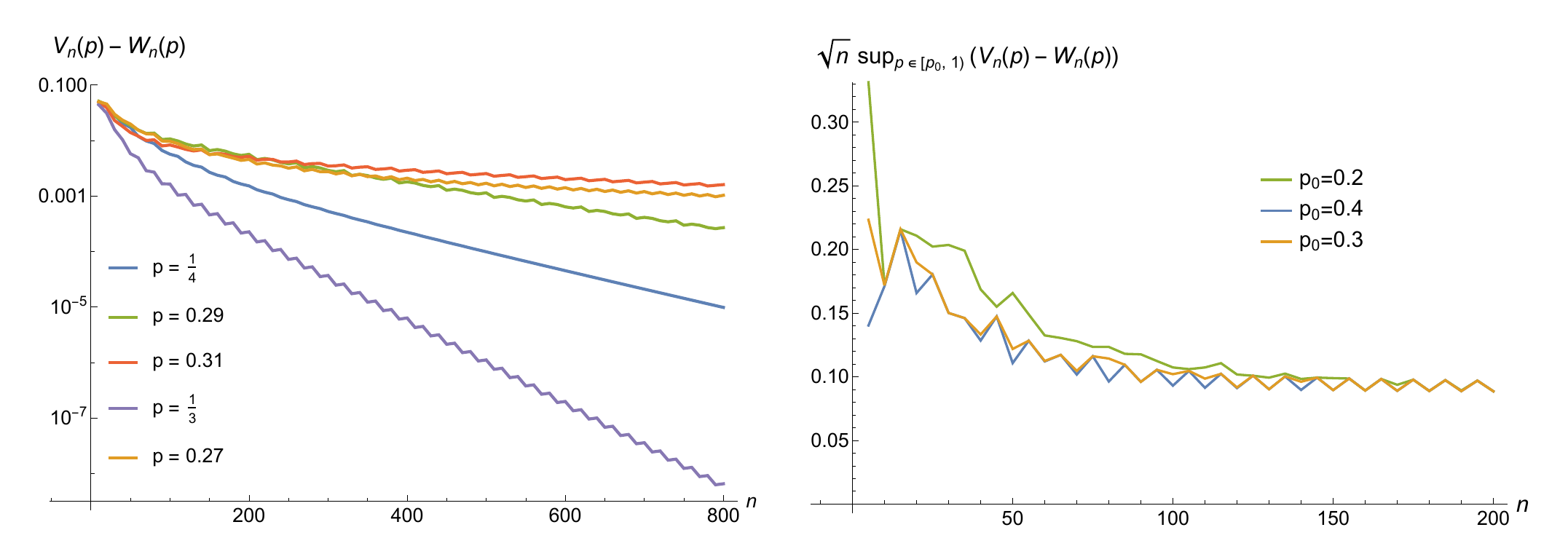}
%
\caption{
(Left panel:) deficit $V_n(p)-W_n(p)$ as a function of $n$, evaluated at the horizons $n=10,20,\ldots,800$, for several fixed values of~$p$ between the boundary points~$p=1/3$ and~$p=1/4$ (logarithmic $y$-axis). (Right panel:) $\sqrt{n}$-scaled worst-case deficit as a function of $n$, for $p_0\in\{0.2,0.3,0.4\}$.
}
\label{Fig2}
\end{figure}


Theorem~\ref{thm:finite_oracle_ineq_p0} shows that, under the mild condition~$p_0\leq \frac{1}{2}$, the worst-case deficit of the plug-in rule in the non-sparse regime converges to zero at rate~$1/\sqrt{n}$. For~$p_0< \frac{1}{2}$, we complement this result with a matching minimax lower bound showing that no (possibly randomized) \mbox{$p$-blind} rule can exhibit a faster rate.

\begin{thm}
\label{thmLowerBound}
For any $p_0\in(0,\tfrac12)$, there exists a positive constant $C(p_0)$ such that, for all~$n$ large enough,
\begin{equation}
\label{worb}
\inf_\pi \sup_{p\in[p_0,1)}\bigl(V_n(p)-W_n^\pi(p)\bigr)\ge \frac{C(p_0)}{\sqrt n},
\end{equation}
where the infimum is over all (possibly randomized) \mbox{$p$-blind} rules and where $W_n^\pi(p)$ denotes the win probability of~$\pi$.
\end{thm}

Note that for~$p\in[\frac{1}{2},1)$, we have~$V_n(p)=p=W_n^{\pi^{\rm last}}(p)$, where~$\pi^{\rm last}$ is the~\mbox{$p$-blind} rule that never stops before time~$n$ and stops on time~$n$ if~$X_n=1$. Consequently, for any~$p_0\geq \frac{1}{2}$ and any~$n$, 
\[
\inf_\pi \sup_{p\in[p_0,1)}\bigl(V_n(p)-W_n^\pi(p)\bigr)
=
0
,
\]
so that no non-trivial lower bound exists for~$p_0\geq \frac{1}{2}$.

Two features of Theorem~\ref{thm:finite_oracle_ineq_p0} will matter repeatedly below and are not
captured by any worst-case rate. First, it is a genuine \emph{finite-horizon} statement: it holds for
every~$n\ge2$, with constants depending on~$p_0$ only---in particular, neither
on~$n$ nor on~$p$---rather than being a statement about a limit. Second, it is \emph{pointwise} in~$p$: at any
fixed~$p$, the bound~(\ref{eq:p0_pointwise_distance}) shows that the deficit decays exponentially
fast in~$n$. Both features are used in Section~\ref{sec:sparse}.



\section{Sharp constants and asymptotic minimax optimality}
\label{sec:constants}

Section~\ref{sec:oracle_bounds} provides matching bounds of order~$1/\sqrt n$ for the worst-case oracle deficit over~$[p_0,1)$: an upper bound for the plug-in rule and, in the minimax sense, a lower bound for the whole class of \mbox{$p$-blind} rules. Rates, however, may fail to discriminate between procedures: any rule matching~(\ref{eq:p0_uniform_two_term}) up to a constant is rate-optimal, and Theorem~\ref{thmLowerBound} leaves open the possibility that some cleverer rule improves substantially upon~$\hat\pi$. This section removes that indeterminacy by computing the constants exactly. We show that the worst-case deficit of the plug-in rule admits an exact asymptotic constant~$C_\star$ (Theorem~\ref{thm:plugin_sharp}), and that~$C_\star$ is also the exact minimax constant (Theorem~\ref{thm:minimax_sharp}); the plug-in rule is therefore asymptotically minimax optimal, and not merely rate-optimal. Throughout, we write
\begin{equation}
\label{eq:gdef}
g(\ell):=\ell p\,(1-p)^{\ell-1}
\end{equation}
for the win probability of the threshold rule that stops at the first success among the last~$\ell$ Bernoulli trials, so that the sum-the-odds theorem states that~$V_n(p)=g(m(p))=\max_{\ell\ge1}g(\ell)$ whenever~$n\ge m(p)$.

\begin{thm}
\label{thm:plugin_sharp}
Fix~$p_0\in(0,\frac12)$ and let
\begin{equation}
\label{eq:Cstar}
C_\star
:=
\frac{1}{2}\sup_{u>0}
u\Phi(-u)
=
0.08498\ldots
,
\end{equation}
where~$\Phi$ denotes the standard normal distribution function; the supremum is attained at the unique positive root~$u_\star=0.7517\ldots$ of~$\Phi(-u)=u\varphi(u)$, with~$\varphi:=\Phi'$. Then,
\begin{equation}
\label{eq:plugin_sharp}
\lim_{n\to\infty}
\ \sqrt n
\sup_{p\in[p_0,1)}\bigl(V_n(p)-W_n(p)\bigr)
\ =\
C_\star
,
\end{equation}
and, for any~$[p_0,1)$-valued sequence~$(p_n)$, one has~$\sqrt n\bigl(V_n(p_n)-W_n(p_n)\bigr)\to C_\star$ if and only if~$\sqrt n\,|p_n-\tfrac12|\to\tfrac{u_\star}{2}$.
\end{thm}

We outline here the mechanism of the proof, which will reappear in Section~\ref{sec:splitting}. Fix~$p\in[p_0,1)$, write~$m=m(p)$, and set
$$
\Delta_-:=p-\frac{1}{m+1}
\quad
\textrm{ and,  for }m\geq 2, 
\quad
\Delta_+:=\frac1m-p
,
$$
so that~$\Delta_p=\min\{\Delta_-,\Delta_+\}$ when~$m\ge2$, whereas~$\Delta_p=\Delta_-=p-\frac12$ when~$m=1$, the value~$\frac1m=1$ not belonging to~$\mathcal B_{p_0}$. The plug-in rule stops on a success at time~$t(<n)$ if and only if~$\hat p_t<1/(n-t+1)$, whereas the oracle does so if and only if~$p<1/(n-t+1)$; see~(\ref{StoppingTimePlugin}). These two prescriptions can therefore disagree only at times~$t(<n)$ for which~$p$ is close to~$1/(n-t+1)$, and since~$p$ lies in the cell~$[\frac{1}{m+1},\frac1m)$, there are exactly two such times:
$$
t_-:=n-m
\
(\textrm{where the plug-in rule may stop although the oracle would not})
,
$$
$$
t_+:=n-m+1
\
(\textrm{where the plug-in rule may fail to stop although the oracle would})
.
$$
(For~$m=1$, one has~$t_+=n$, where the terminal clause in~(\ref{StoppingTimePlugin}) makes the plug-in rule stop on a success, so that only~$t_-$ is critical.)
At every other time, the discrepancy between~$p$ and~$1/(n-t+1)$ is bounded below by a constant depending only on~$p_0$, so that the corresponding errors have exponentially small probability. Conditioning on~$X_{t_\pm}$ and using~(\ref{eq:gdef}), the losses attached to these two events are
\begin{equation}
\label{eq:g_gap_minus}
g(m)-g(m+1)
=
(m+1)p(1-p)^{m-1}
\Delta_-
=:K_-(p)
\Delta_-
\end{equation}
and
\begin{equation}
\label{eq:g_gap_plus}
g(m)-g(m-1)
=
mp(1-p)^{m-2}\Delta_+
=:K_+(p) \Delta_+
.
\end{equation}
The probability of an incorrect decision, in turn, satisfies
$$
\mathbb P_p\bigl[\pm(\hat p_{t_\pm}-p)\ge\Delta_\pm\bigr]
=
\Phi\Bigl(-\frac{\Delta_\pm\sqrt n}{\sqrt{p(1-p)}}\Bigr)
+
O\Big(\frac{1}{\sqrt{n}}\Big)
,
$$
the error term being provided by the Berry--Esseen theorem \citep[see, e.g.,][Chapter~5]{Petrov1995}. The central limit theorem would give~$\sqrt{t_\pm}$ rather
than~$\sqrt n$ in the argument of~$\Phi$; the two may be interchanged here
because~$t_\pm=n-O(1)$, the critical times lying within~$M_0$ of the horizon. This is the precise sense in which concentration of~$\hat p_t$ propagates to the win probability. 

Near a boundary, the deficit is therefore given, to leading order, by
$$
K_\pm(p)\,\Delta_\pm\,
\Phi\Bigl(-\frac{\Delta_\pm\sqrt{n}}{\sqrt{p(1-p)}}\Bigr)
,
$$
which has to be maximized over~$\Delta_\pm$ and over the boundary at which the worst case occurs. Writing~$\Delta_\pm=u\sqrt{p(1-p)}/\sqrt n$ turns this deficit into
$$
\frac{K_\pm(p)\sqrt{p(1-p)}}{\sqrt n}\ 
u
\Phi(-u)
,
$$
so that the first maximization amounts to maximizing~$u\Phi(-u)$ over~$u>0$: it produces the factor~$\sup_{u>0}u\Phi(-u)=2C_\star$ and confirms that the least favourable parameters lie at distance~$\Delta_\pm\asymp1/\sqrt n$ from the boundary. There remains to maximize the prefactor~$K_\pm(p)\sqrt{p(1-p)}$ over the boundaries~$\frac1k\in\mathcal B_{p_0}$. Since~$\Delta_\pm\to0$ there, this prefactor is evaluated at~$p=\frac1k$, where it equals
\begin{equation}
\label{eq:gammak}
\gamma_k
:=
\Bigl(1-\frac1k\Bigr)^{k-2}\sqrt{\frac1k\Bigl(1-\frac1k\Bigr)}
,
\qquad
\textrm{ with}
\quad
\max_{k\ge2}\gamma_k=\gamma_2=\frac12
.
\end{equation}
Note that the same value~$\gamma_k$ is obtained whether~$\frac1k$ is approached from above or from below: in the first case~$m(p)=k-1$ and the relevant coefficient is~$K_-(\frac1k)=(1-\frac1k)^{k-2}$, in the second~$m(p)=k$ and it is~$K_+(\frac1k)=(1-\frac1k)^{k-2}$. The maximization over~$k$ in~(\ref{eq:gammak}) is elementary. Collecting the two maximizations, the constant in~(\ref{eq:plugin_sharp}) is
$$
\biggl(\, 
\max_{k\ge2}\gamma_k
\biggr)
\biggl(\, 
\sup_{u>0}u\Phi(-u)
\biggr)
=
\frac12
\times
2C_\star
=
C_\star
, 
$$
and the worst case occurs near~$p=\frac12$, where the Bernoulli variance is largest and the oracle hesitates between using the last one and the last two observations.

The analysis behind Theorem~\ref{thm:plugin_sharp} is in fact local at each transition point separately, and this local form is worth isolating: it holds at \emph{every} point of~$\mathcal B_{p_0}$, and it is what will yield the minimax constant. For~$\frac1k\in\mathcal B_{p_0}$ and~$u\in\mathbb R$, write
\begin{equation}
\label{eq:localseq}
\qquad
p_{n,k}(u):=\frac1k+\frac{u\sigma_k}{\sqrt n}
,
\qquad 
\textrm{with }
\
\sigma_k:=\sqrt{\tfrac1k\bigl(1-\tfrac1k\bigr)}
,
\end{equation}
so that~$u$ measures the distance to~$\frac1k$ in units of the local standard deviation, and recall~$\gamma_k$ from~(\ref{eq:gammak}).

\begin{prop}
\label{prop:local_at_1k}
Fix~$p_0\in(0,1)$ and~$\frac1k\in\mathcal B_{p_0}$. Then,
\begin{itemize}
\item[(i)] for every~$u\in\mathbb R$,
\begin{equation}
\label{eq:local_plugin_risk}
\lim_{n\to\infty}\sqrt n\,\bigl(V_n\bigl(p_{n,k}(u)\bigr)-W_n\bigl(p_{n,k}(u)\bigr)\bigr)
=
\gamma_k |u| \Phi(-|u|)
;
\end{equation}
\item[(ii)] for every~$c>0$ and every sequence~$(\pi_n)$ of (possibly randomized) \mbox{$p$-blind} rules,
\begin{eqnarray}
\lefteqn{
\hspace{-30mm}
\lim_{c\to\infty}\ \liminf_{n\to\infty}\ \sup_{|u|\le c}\ \sqrt n\,\bigl(V_n\bigl(p_{n,k}(u)\bigr)-W_n^{\pi_n}\bigl(p_{n,k}(u)\bigr)\bigr)
}
\nonumber
\\[2mm]
& &
\hspace{-25mm}
 \ge
\gamma_k\sup_{u>0}u\Phi(-u)
=
2\gamma_kC_\star
,
\label{eq:local_minimax_1k}
\end{eqnarray}
a bound that the plug-in rule attains, by~(\ref{eq:local_plugin_risk}).
\end{itemize}
\end{prop}

Proposition~\ref{prop:local_at_1k} is the complete local picture. At each oracle discontinuity~$\frac1k$, the plug-in rule has the limiting local risk function~$u\mapsto\gamma_k|u|\Phi(-|u|)$, which is symmetric in~$u$, vanishes at~$u=0$ and as~$|u|\to\infty$, and peaks at~$|u|=u_\star$---the same~$u_\star$ at every~$k$, since the profile depends on the boundary only through the multiplicative factor~$\gamma_k$. That factor is thus the local difficulty of the~$k$th discontinuity, and~(\ref{eq:local_minimax_1k}) says that no \mbox{$p$-blind} rule can reduce it. This also identifies the least favourable sequences of Theorem~\ref{thm:plugin_sharp} as the maximizers~$u=\pm u_\star$ of the local profile at the winning boundary: since~$p_{n,2}(u)=\frac12+\frac{u}{2\sqrt n}$, the condition~$|u|=u_\star$ is exactly~$\sqrt n\,|p_n-\frac12|\to\frac{u_\star}{2}$.

While Part~(i) of Proposition~\ref{prop:local_at_1k} describes the deficit only along the local sequences~(\ref{eq:localseq}), and therefore falls short of Theorem~\ref{thm:plugin_sharp}, whose supremum over~$[p_0,1)$ calls for a control that is uniform in~$p$, Part~(ii) readily yields the following matching lower bound (by taking~$k=2$, for which~$2\gamma_2C_\star=C_\star$;
see Appendix~\ref{appProofsConstants} for details).

\begin{thm}
\label{thm:minimax_sharp}
Fix~$p_0\in(0,\frac12)$. Then, with~$C_\star$ as in~(\ref{eq:Cstar}),
\begin{equation}
\label{eq:minimax_sharp}
\lim_{n\to\infty}
\ \sqrt n\,
\inf_\pi\sup_{p\in[p_0,1)}\bigl(V_n(p)-W_n^\pi(p)\bigr)
\ =\
C_\star
,
\end{equation}
where the infimum is over all (possibly randomized) \mbox{$p$-blind} rules. Consequently, the plug-in rule is asymptotically minimax optimal, with the exact constant.
\end{thm}

Theorem~\ref{thm:minimax_sharp} sharpens Theorem~\ref{thmLowerBound} from a rate statement into an exact constant: no oracle-free rule---however sophisticated, and even allowing randomi\-zation---can asymptotically outperform the simple empirical-odds prescription in the worst case over~$[p_0,1)$. The obstruction is moreover entirely local: it is already present in an arbitrarily small neighbourhood of~$\frac12$, so that not even a rule tailored to that single neighbourhood could do better.

Since Theorem~\ref{thm:minimax_sharp} follows from Proposition~\ref{prop:local_at_1k}(ii), it is natural to sketch the argument behind that part, which is where the constant is produced and which explains \emph{why} it takes the value it does. We describe it at~$k=2$, where the notation is lightest; the general case only replaces the decision time~$n-1$ by~$n-k+1$ and the factor~$\frac12$ by~$\gamma_k$.
Near the least favourable parameter~$p=\frac12$, all decisions but one are asymptotically clear-cut, and any rule is characterized by a single binary choice: whether to stop upon observing a success at time~$n-1$. Stopping wins with probability~$1-p$, continuing with probability~$p$, so the correct action is to stop if and only if~$p<\frac12$; taking the incorrect action costs~$|(1-p)-p|=|1-2p|$ in win probability, and the situation itself arises with probability~$p\to\frac12$. Along the local sequences~$p_n^\pm=\frac12\pm\frac{h}{2\sqrt n}$, the associated statistical experiments converge to the Gaussian shift experiment~$\mathcal N(\pm h,1)$, and the rule's choice becomes a test~$\varphi$ between them. Writing~$\beta(h):=\mathbb E_h[\varphi]$, the two local deficits are~$\frac{h}{2\sqrt n}\beta(h)$ and~$\frac{h}{2\sqrt n}(1-\beta(-h))$ up to negligible terms, so that Le~Cam's two-point bound gives

\begin{eqnarray*}
2\max\bigl\{\textrm{deficit}(p_n^+),\textrm{deficit}(p_n^-)\bigr\}
& \gtrsim &
\frac{h}{2\sqrt n}\bigl[\beta(h)+1-\beta(-h)\bigr]
\\[2mm]
& \ge &
\frac{h}{2\sqrt n}\Bigl[1-\mathrm{TV}\bigl(\mathcal N(h,1),\mathcal N(-h,1)\bigr)\Bigr]
\\[2mm]
& = &
\frac{h\,\Phi(-h)}{\sqrt n}
,
\end{eqnarray*}
and optimizing over~$h$ produces~$\gamma_2\sup_{h>0}h\Phi(-h)=2\gamma_2C_\star=C_\star$, which is~(\ref{eq:local_minimax_1k}) at~$k=2$. The bound is attained by the likelihood-ratio test cutting at the midpoint~$0$ of the two hypotheses, which is precisely what the plug-in rule implements, since it stops at~$n-1$ on a success if and only if~$\hat p_{n-1}<\frac12$. The two occurrences of the quantity~$\sup_{u>0}u\Phi(-u)$---in Theorem~\ref{thm:plugin_sharp} through the maximization of~$\Delta\Phi(-\Delta\sqrt n/\sigma)$, and here through a Gaussian testing bound---are therefore two faces of the same phenomenon.

We stress that Theorems~\ref{thm:plugin_sharp}--\ref{thm:minimax_sharp} do not render
Theorems~\ref{thm:finite_oracle_ineq_p0}--\ref{thmLowerBound} superfluous; the two pairs answer
different questions, and the results of this section in fact \emph{rest} on those of
Section~\ref{sec:oracle_bounds}. The present theorems are limit statements: they identify the
constant that governs the worst case as~$n\to\infty$, but they are silent at any fixed horizon,
they describe only the supremum over~$p$, and they require~$p_0<\frac12$. Theorem~\ref{thm:finite_oracle_ineq_p0},
by contrast, holds for every~$n\ge2$ and every~$p_0\in(0,1)$---including~$p_0>\frac12$, where the
deficit is exponentially small and~$C_\star$ plays no role---and its pointwise form quantifies the
deficit at each individual~$p$, which is what makes it usable as an input elsewhere: it is through
Theorem~\ref{thm:finite_oracle_ineq_p0} that the uniform convergence result of
Section~\ref{sec:sparse} and the~$p\asymp1/n$ barrier are obtained. Moreover, the proof of
Theorem~\ref{thm:plugin_sharp} relies on the exponential controls underlying
Theorem~\ref{thm:finite_oracle_ineq_p0} to discard all decision times but~$t_-$ and~$t_+$, and
Theorem~\ref{thmLowerBound}---a two-point argument requiring no local asymptotics, and valid on the
whole range where a non-trivial bound exists---remains the tool of choice whenever rate optimality
suffices, as for the sample-splitting rules of Section~\ref{sec:splitting}. The sharp constants of
this section should therefore be read as a refinement of Section~\ref{sec:oracle_bounds} along one
particular axis, namely the worst case over~$[p_0,1)$ with~$p_0<\frac12$ as~$n\to\infty$, and not as
a replacement for it.



\section{Sample-splitting rules: the price of freezing the estimate}
\label{sec:splitting}

Sections~\ref{sec:oracle_bounds}--\ref{sec:constants} have settled the behaviour of the plug-in rule: for~$p_0\in(0,\frac12)$, it is minimax rate-optimal on~$[p_0,1)$, and in fact asymptotically minimax optimal with the exact constant~$C_\star$. One may wonder how much of this owes to the specific sequential design of~$\hat\pi$. Indeed, $\hat\pi$ refreshes~$\hat p_t$ at every time~$t$, so that the quantity driving the stopping decision is itself a function of the trajectory on which the decision is taken; this entanglement between estimation and stopping is what made the proofs of Theorems~\ref{thm:finite_oracle_ineq_p0} and~\ref{thm:plugin_sharp} delicate, forcing us respectively to pass to a predictable version of~$\hat p_t$ and to show that only the two decision times~$t_\pm$ matter. It is therefore natural to examine the design in which the two operations are deliberately decoupled: spend an initial fraction of the horizon on estimation only, freeze the resulting estimate, and then run the oracle algorithm with that frozen value. This is the optimal-stopping analogue of sample splitting, arguably the first rule a statistician would write down, and hence a natural candidate to compete with~$\hat\pi$; freezing the estimate moreover removes the entanglement just described, which makes the resulting rule far easier to analyse. It does not compete, however: we show that it is rate-optimal but never constant-optimal, and Theorem~\ref{thm:minimax_sharp} allows us to quantify its minimax deficiency exactly.

Formally, fix~$a\in(0,1)$ and let
$$
\hat p^{(a)}:=\frac{S_{t_a}}{t_a}
,
\quad
\textrm{ with }
\ 
t_a:=\lfloor an\rfloor
.
$$
The rule~$\pi_a$ spends the times~$\{1,\ldots,t_a\}$ on estimation only and, from time~$t_a+1$ onward, applies the sum-the-odds algorithm with the \emph{frozen} estimate~$\hat p^{(a)}$; that is, $\pi_a$ is associated with the stopping time
\begin{equation}
\label{StoppingTimeSplit}
\tau_n^{(a)}
:=
\inf
\biggl\{
t\in\{t_a+1,\dots,n\}:\ X_t=1
\ \text{ and }\
\Bigl(
\hat p^{(a)}<\frac{1}{n-t+1}
\ \text{ or }\
t=n
\Bigr)
\biggr\}
,
\end{equation}
with $\inf\varnothing:=+\infty$. As for~$\hat\pi$, the terminal clause makes~$\pi_a$ stop at time~$n$ on a success if it has not stopped earlier. Clearly, $\pi_a$ is~\mbox{$p$-blind}. We write~$W_n^{\pi_a}(p)$ for its win probability, as defined in~(\ref{eq:Wpi_def}).

The decisive structural feature of~$\pi_a$ is that~$\hat p^{(a)}$ is~$\sigma(X_1,\ldots,X_{t_a})$-measurable, hence independent of the observations~$X_{t_a+1},\ldots,X_n$ on which the rule actually operates. Conditionally on~$\hat p^{(a)}$, the rule is therefore an \emph{oracle-type threshold rule with a deterministic threshold}, run on i.i.d.\ data independent of that threshold. Writing
\begin{equation}
\label{eq:La_def}
L_j
:=
\max\biggl\{1,\min\Bigl\{n-t_a,\ \Bigl\lceil\frac{t_a}{j}\Bigr\rceil-1\Bigr\}\biggr\}
,
\qquad
j=0,1,\ldots,t_a
,
\end{equation}
with the convention~$\lceil t_a/0\rceil:=+\infty$ (so that~$L_0=n-t_a$), the rule~$\pi_a$ will stop, when~$S_{t_a}=j$, at the first success (if any) in the terminal block~$\{n-L_j+1,\ldots,n\}$: the sum-the-odds prescription based on~$\hat p^{(a)}=j/t_a$ opens the window at time~$n-\lceil t_a/j\rceil+2$, the truncation at~$n-t_a$ enforces that no stopping occurs before~$t_a+1$, and the terminal clause guarantees a window of length at least one. This yields the compact representation
\begin{equation}
\label{eq:Wa_compact}
W_n^{\pi_a}(p)
=
\mathbb E_p[g(L_{S_{t_a}})]
,
\end{equation}
with~$g$ as in~(\ref{eq:gdef}). Expanding the expectation against the~${\rm Bin}(t_a,p)$ probability mass function of~$S_{t_a}$ turns~(\ref{eq:Wa_compact}) into a closed-form expression, evaluable in~$O(n)$ arithmetic operations and showing that~$p\mapsto W_n^{\pi_a}(p)$ is a polynomial; see Theorem~\ref{thm:exact_Wa}. This is in marked contrast with the dynamic-programming recursion required for~$\hat\pi$ (Theorem~\ref{thm:exact_Wn_plugin}).

We now turn to the oracle bound. Since~$V_n(p)=g(m(p))=\max_{\ell\ge1}g(\ell)$ whenever~$n\ge m(p)$---which, for~$p\in[p_0,1)$, holds as soon as~$n>M_0$, since~$m(p)\le M_0$ there---(\ref{eq:Wa_compact}) expresses the oracle deficit, for such~$n$, as
\begin{equation}
\label{eq:deficit_as_expectation}
V_n(p)-W_n^{\pi_a}(p)
=
\mathbb E_p[g(m(p))-g(L_{S_{t_a}})]
\ \ (\ge0)
,
\end{equation}
so that the deficit is entirely governed by the event that the frozen estimate falls in a different cell~$[\frac{1}{k+1},\frac1k)$ than~$p$ does. This is the mechanism announced in the introduction, in its simplest form: a concentration statement on~$\hat p^{(a)}$ propagates to the win probability through the loss~$g(m(p))-g(\ell)$ incurred by an incorrect cell. We have the following result.

\begin{thm}
\label{thm:oracle_split}
Fix~$p_0\in(0,1)$ and~$a\in(0,1)$, and let~$M_0$, $\mathcal B_{p_0}$ and~$\Delta_p$ be as in Theorem~\ref{thm:finite_oracle_ineq_p0}. Then, (i) there exist positive constants~$C_1(p_0)$, $C_2(p_0,a)$ and~$c(p_0,a)$ such that, for all~$n\ge2$ and all~$p\in[p_0,1)$,
\begin{equation}
\label{eq:split_pointwise}
V_n(p)-W_n^{\pi_a}(p)
\le
C_1(p_0)\,\Delta_p\,e^{-2\lfloor an\rfloor\Delta_p^2}
+
C_2(p_0,a)\,e^{-c(p_0,a)n}
.
\end{equation}
(ii) Consequently, for~$p_0\le\frac12$ there exists a positive constant~$C(p_0)$ such that, for all~$n$ large enough,
\begin{equation}
\label{eq:split_uniform}
\sup_{p\in[p_0,1)}\bigl(V_n(p)-W_n^{\pi_a}(p)\bigr)
\le
\frac{C(p_0)}{\sqrt{an}}
.
\end{equation}
\end{thm}

Theorem~\ref{thm:oracle_split} has the same structure as Theorem~\ref{thm:finite_oracle_ineq_p0}, with one difference that is the whole point: the horizon~$n$ multiplying~$\Delta_p^2$ in the first exponent, and with it the resulting rate, has been replaced by the \emph{effective sample size}~$\lfloor an\rfloor$ on which the frozen estimate is built. The explicit constant in this exponent matters here: had it been left unspecified, $a$ could have been absorbed into it, whereas as stated it propagates to the~\mbox{$a$-free} constant in~(\ref{eq:split_uniform}). Since~$a$ is fixed, $\pi_a$ still attains the~$1/\sqrt n$ rate, so that, by Theorem~\ref{thmLowerBound}, it is minimax rate-optimal on~$[p_0,1)$ for every~$a\in(0,1)$. At the resolution of Section~\ref{sec:oracle_bounds}, then, sample splitting is indistinguishable from sequential updating. It is only at the resolution of Section~\ref{sec:constants} that the two designs separate, and the following result---the analogue for~$\pi_a$ of Theorem~\ref{thm:plugin_sharp}---shows that the inflation by~$1/\sqrt a$ in~(\ref{eq:split_uniform}) is genuine rather than an artifact of the proof.

\begin{thm}
\label{thm:split_sharp}
Fix~$p_0\in(0,\frac12)$ and~$a\in(0,1)$, and let~$C_\star$ and~$u_\star$ be as in Theorem~\ref{thm:plugin_sharp}. Then,
\begin{equation}
\label{eq:split_sharp}
\lim_{n\to\infty}
\ \sqrt{an}
\sup_{p\in[p_0,1)}\bigl(V_n(p)-W_n^{\pi_a}(p)\bigr)
\ =\
C_\star
,
\end{equation}
and, for any~$[p_0,1)$-valued sequence~$(p_n)$, one has~$\sqrt{an}\bigl(V_n(p_n)-W_n^{\pi_a}(p_n)\bigr)\to C_\star$ if and only if~$\sqrt{an}\,\bigl|p_n-\tfrac12\bigr|\to\tfrac{u_\star}{2}$.
\end{thm}

Theorem~\ref{thm:split_sharp} is the exact analogue, for~$\pi_a$, of Theorem~\ref{thm:plugin_sharp}: the very same constant~$C_\star$ appears, and the very same parameters---those lying on either side of~$\frac12$, at distance~$u_\star/2$ from it once rescaled---are least favourable; only the effective sample size differs, being~$\lfloor an\rfloor$ instead of~$n$, which is precisely what makes that distance larger for~$\pi_a$. The proof, given in Appendix~\ref{secProofsSampleSplit}, follows the scheme outlined after Theorem~\ref{thm:plugin_sharp}, but is markedly simpler: by~(\ref{eq:deficit_as_expectation}) the deficit of~$\pi_a$ is a \emph{single} binomial expectation, so that a one-dimensional normal approximation suffices, whereas for~$\hat\pi$ one had first  to show that only the two decision times~$t_\pm$ matter. In both cases, Hoeffding's inequality---which suffices for the rates in Theorems~\ref{thm:finite_oracle_ineq_p0} and~\ref{thm:oracle_split}---is too lossy in the polynomial factor, and the parameter space is split at distance~$\Lambda/\sqrt n$ from~$\mathcal B_{p_0}$: beyond that distance a Hoeffding bound already contributes~$\le\varepsilon/\sqrt n$ for~$\Lambda=\Lambda(\varepsilon)$ large, while within it a Berry--Esseen bound applies, its~$O(1/\sqrt n)$ error being multiplied by a loss of order~$1/\sqrt n$ and hence negligible.

Combining Theorems~\ref{thm:plugin_sharp} and~\ref{thm:split_sharp} settles the comparison between the two designs.

\begin{co}
\label{co:ratio}
Fix~$p_0\in(0,\frac12)$ and~$a\in(0,1)$. Then,
\begin{equation}
\label{eq:constant_comparison}
\lim_{n\to\infty}
\frac{\sup_{p\in[p_0,1)}\bigl(V_n(p)-W_n^{\pi_a}(p)\bigr)}
      {\sup_{p\in[p_0,1)}\bigl(V_n(p)-W_n(p)\bigr)}
=
\frac{1}{\sqrt a}
\ >\ 1
.
\end{equation}
In particular, the worst-case deficit of every sample-splitting rule~$\pi_a$, $a\in(0,1)$, over~$[p_0,1)$ is asymptotically larger than that of the plug-in rule, by the factor~$1/\sqrt a$. Equivalently, in view of Theorem~\ref{thm:minimax_sharp}, no~$\pi_a$ with~$a<1$ is asymptotically minimax optimal, its asymptotic minimax deficiency being~$1/\sqrt a$.
\end{co}

Table~\ref{tab:constants} illustrates numerically Theorems~\ref{thm:plugin_sharp} and~\ref{thm:split_sharp}, as well as Corollary~\ref{co:ratio}.

\begin{table}[h!]
\hspace{-10mm}
\begin{tabular}{r|c|ccc|ccc}
& $\sqrt n\,\sup_{p\in[p_0,1)}(V_n-W_n)$ & \multicolumn{3}{c|}{$\sqrt{an}\,\sup_{p\in[p_0,1)}(V_n-W_n^{\pi_a})$} & \multicolumn{3}{c}{ratio to $\hat\pi$} \\
$n$ & $\hat\pi$ & $a=\frac14$ & $a=\frac12$ & $a=\frac34$ & $a=\frac14$ & $a=\frac12$ & $a=\frac34$ \\
\hline
$400$  & $0.0875$ & $0.1140$ & $0.0881$ & $0.0876$ & $2.606$ & $1.424$ & $1.156$ \\
$800$  & $0.0867$ & $0.0881$ & $0.0872$ & $0.0869$ & $2.032$ & $1.422$ & $1.156$ \\
$1600$ & $0.0862$ & $0.0873$ & $0.0866$ & $0.0863$ & $2.024$ & $1.421$ & $1.156$ \\
$3200$ & $0.0859$ & $0.0866$ & $0.0862$ & $0.0859$ & $2.018$ & $1.419$ & $1.156$ \\
\hline
$1/\sqrt a$ & --- & --- & --- & --- & $2.000$ & $1.414$ & $1.155$
\end{tabular}
\caption{Worst-case oracle deficits over~$p\in[p_0,1)$, with~$p_0=0.2$, scaled by the square root of the corresponding effective sample size ($n$ for the plug-in rule~$\hat\pi$, $an$ for the sample-splitting rule~$\pi_a$). All scaled deficits approach the common value~$C_\star=0.08498\ldots$ of~(\ref{eq:Cstar}) (Theorems~\ref{thm:plugin_sharp} and~\ref{thm:split_sharp}), and the last three columns, which report the ratio of the worst-case deficit of~$\pi_a$ to that of~$\hat\pi$, converge to~$1/\sqrt a$ (Corollary~\ref{co:ratio}).}
\label{tab:constants}
\end{table}

Three comments are in order. First, sample splitting is never advantageous here: by Corollary~\ref{co:ratio}, the worst-case deficit of~$\pi_a$ exceeds that of~$\hat\pi$ by the factor~$1/\sqrt a$, irrespective of~$p_0$, and the natural sequential design is thereby vindicated. Second, the loss vanishes as~$a\to 1$, but not uniformly in~$n$: the decision window~$\{t_a+1,\ldots,n\}$ must contain the oracle window, which requires~$(1-a)n\gtrsim M_0$, so that~$a$ may be taken close to~$1$ only for large~$n$; no single~$a<1$ closes the gap. Third, the rules~$\hat\pi$ and~$\pi_a$ fail in exactly the same way---at the same least favourable parameters, through the same boundary decisions, with the same constant~$C_\star$ of Section~\ref{sec:constants}---and differ only through the number of observations available when those decisions are taken. The comparison therefore isolates, in a single scalar, the value of the information that a sequential rule keeps accumulating while it waits; and since Theorem~\ref{thm:minimax_sharp} identifies~$C_\star$ as the minimax constant, that scalar is exactly the asymptotic minimax deficiency of sample splitting.

Figure~\ref{FigSplit} shows the same comparison pointwise rather than in the worst case. At~$n=400$, the curves are ordered by~$a$ over most of~$[p_0,1)$, with that of~$\hat\pi$ lowest, and all of them dip sharply at the boundary points~$\frac1k$, where~$g(k)=g(k-1)$ makes the two candidate values of~$m(p)$ equally good, so that misclassifying~$m(p)$ costs almost nothing. Two features temper the picture. First, the ordering is one of suprema, not a pointwise one: near the boundary points, where all deficits are close to zero, the curves may cross (consistently with Theorem~\ref{thm:no_uniform_opt}, $\hat\pi$ does not uniformly dominate~$\pi_a$: there are parameter values at which~$\pi_a$ wins with the larger probability). Second, it is asymptotic, and is not yet in force at small horizons: at~$n=10$, the curve of~$\pi_{3/4}$ lies below that of~$\hat\pi$ on a substantial part of the range and its worst-case deficit over~$[p_0,1)$ is even marginally the smaller of the two. It is only once~$\lfloor an\rfloor$ is large that the frozen estimate becomes the binding constraint and the~$1/\sqrt a$ ordering of Corollary~\ref{co:ratio} emerges.

\begin{figure}[h!]
\makebox[\textwidth][c]{\includegraphics[width=1.05\textwidth]{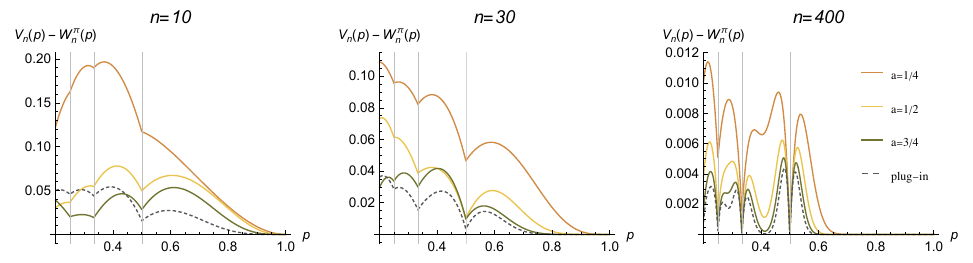}}
\caption{%
Oracle deficits~$V_n(p)-W_n^{\pi_a}(p)$ over~$p\in[p_0,1)$, with~$p_0=0.2$, for the
sample-splitting rules with~$a=1/4$, $1/2$, $3/4$ and horizons~$n=10$ (left), $30$ (center)
and~$400$ (right) (note the different vertical scales across panels); the dashed curve is the deficit~$V_n(p)-W_n(p)$ of the plug-in
rule~$\hat\pi$. The vertical grey lines mark the boundary points~$1/k$
of~$\mathcal B_{p_0}$. In the
right panel, the suprema of the four curves are in the ratios predicted by
Corollary~\ref{co:ratio}, in agreement with the first row of
Table~\ref{tab:constants}.}
\label{FigSplit}
\end{figure}


\section{The sparse regime and the $p\asymp 1/n$ barrier}
\label{sec:sparse}

So far, the parameter set has been of the form~$[p_0,1)$, with~$p_0>0$ fixed. We now dispense with this restriction and determine what can be achieved uniformly over the whole of~$(0,1)$. The essential step is the sparse regime, in which the success probability vanishes with the horizon while the expected number of successes still diverges, that is, $p=p_n\to0$ with~$np_n\to\infty$. There, the oracle win probability~$V_n(p_n)$ is known to converge to~$1/e$, and the question is whether the plug-in rule attains this benchmark although it learns~$p_n$ from the data alone. We answer positively, by establishing an explicit rate of decay for the deficit~$V_n(p_n)-W_n(p_n)$. Combining this with the finite-horizon oracle bounds of Section~\ref{sec:oracle_bounds} then yields a broad uniform convergence statement, which we finally show to be maximal: no~\mbox{$p$-blind} rule can satisfy a broader one, the window~$p\asymp1/n$ forming a genuine barrier.

\subsection{The sparse regime}
\label{sec:sparsecase}

Consider the asymptotic scenario associated with a sequence~$(p_n)$ in~$(0,1)$ such that~$p_n\to 0$ and~$n p_n\to\infty$.
For~$n\geq m_n:=\lceil \frac{1}{p_n}\rceil-1$, the win probability of the $p_n$-oracle rule is
$
V_n(p_n)
=
m_n p_n (1-p_n)^{m_n-1}
$.
It can then be shown\footnote{For the sake of completeness, we prove this in Appendix~\ref{secProofsSection6}.} that there exists a positive constant~$C$ such that, for all~$n$ with~$n\geq m_n$ and~$p_n\leq 1/2$, we have
\begin{equation}
\label{eq:Vn-rate}
\Big|V_n(p_n)-\frac{1}{e}\Big|
\le
Cp_n
,
\end{equation}
so that in particular, $V_n(p_n)\to 1/e$ in the sparse regime.
The following result entails in particular that the plug-in rule is asymptotically optimal in this regime.

\begin{thm}
\label{thm:plugin-eff-sparse}
Let~$(p_n)$ be a sequence in~$(0,1)$ such that~$p_n\to 0$ and~$np_n\to\infty$. Then, there exist positive constants~$C_1,C_2$ such that
$$
V_n(p_n)-W_n(p_n)
\le
C_1 \sqrt{\frac{\log(n p_n)}{n p_n}}
+
C_2 p_n
$$
for all~$n$ large enough.
In particular, $W_n(p_n)\to 1/e$. 
\end{thm}


While a Hoeffding-based uniform control of $\hat p_t-p_t$ is sufficient in the non-sparse regime considered in Section~\ref{sec:oracle_bounds}, it becomes too crude when $p_n\to0$: in the sparse setting, the relevant fluctuations are governed by the (small) variance scale $t p_n$, and we therefore rely on a variance-sensitive martingale deviation inequality, namely Freedman's inequality (see, \emph{e.g.}, \citealp{Freedman1975TailMartingales}, \citealp{Tropp2011FreedmanMatrix}, or \citealp{Howard2021TimeUniformChernoff}). More precisely, the proofs of this section rest on a uniform law of large numbers for~$\hat p_t/p_n$ over time windows~$\{t\ge t_n\}$ with~$n/t_n=O(1)$, stated and proved as Lemma~\ref{lem:uniform-lln-horizon} in Appendix~\ref{secProofsSection6}.


\subsection{A maximal uniform convergence result}
\label{sec:maximal}

Theorems~\ref{thm:finite_oracle_ineq_p0}--\ref{thm:plugin-eff-sparse} allow us to establish the following uniform convergence result.

\begin{thm}
\label{thm:uniform-etan}
Let~$(p_n)$ be a sequence in~$(0,1)$ such that~$p_n\to 0$ and~$n p_n\to\infty$.
Let~$(\tilde{p}_n)$ be a sequence in~$(0,1)$ such that~$n \tilde{p}_n\to 0$.  Then,
$$
\lim_{n\to\infty}
\sup_{p\in(0,\tilde{p}_n]\cup[p_n,1)} 
\big(V_n(p)-W_n(p)\big) 
= 0.
$$
\end{thm}


The uniform convergence result in Theorem~\ref{thm:uniform-etan} is maximal, since convergence does not hold in the regime~$p\asymp 1/n$. 
To show this, let~$p_n:=c/n$ with~$c\in(0,1)$. Consider the event
$$
G_n
:=
\Big\{
\textstyle\sum_{i=1}^n X_i=1
\ \text{and the unique success occurs in } \{1,\dots,\lceil \tfrac{n}{2} \rceil\}
\Big\}
.
$$
Since~$p_n<1/n$, the $p_n$-oracle threshold in~(\ref{eq:sn_homo}) is $s_n(p_n)=1$, so that this oracle rule stops at the first success (if any). In particular, this rule wins if and only if~$S_n=1$. In contrast, the plug-in rule loses on~$G_n$ because it never stops earlier than~$\lceil \frac{n}{2} \rceil+1$ (see~(\ref{noearlystop})), so that
$$
W_n(p_n)
\le
\mathbb P_{p_n}[S_n=1, G_n^{c}]
+
\mathbb P_{p_n}[S_n\ge 2]
.
$$
Consequently,
\begin{eqnarray*}
V_n(p_n)-W_n(p_n)
\!\!&\!\!\ge \!\! &\!\!
\mathbb P_{p_n}[S_n=1]
-
\mathbb P_{p_n}[S_n=1, G_n^{c}]
-
\mathbb P_{p_n}[S_n\ge 2]
\\[2mm]
\!\!&\!\!= \!\! &\!\!
\mathbb P_{p_n}[G_n]
-
\mathbb P_{p_n}[S_n\ge 2]
.
\end{eqnarray*}
Since
$$
\mathbb P_{p_n}[G_n]
=
\lceil \tfrac{n}{2}\rceil p_n(1-p_n)^{n-1}
\to 
\frac{1}{2}ce^{-c}
$$
and
$$
\mathbb P_{p_n}[S_n\ge 2]
=
1-(1-p_n)^n-np_n(1-p_n)^{n-1}
\to 
1-(1+c)e^{-c}
$$
this yields
$$
\liminf_{n\to\infty}
\big(V_n(p_n)-W_n(p_n)\big)
 \ge
\Bigl(1+\frac{3c}{2}\Bigr)e^{-c}-1.
$$
If~$c$ is sufficiently small to make the right-hand side positive, we then have
$$
\liminf_{n\to\infty} 
\sup_{p\in(0,1)}
\big(V_n(p)-W_n(p)\big) 
\ge
\liminf_{n\to\infty}
\big(V_n(p_n)-W_n(p_n)\big)
> 0
,
$$
which proves that the uniform convergence in Theorem~\ref{thm:uniform-etan} cannot be extended to~$(0,1)$.  This is clearly supported by the plot of the deficit~$V_n(p)-W_n(p)$ in Figure~\ref{Fig3}.

\begin{figure}[h!]
\includegraphics[width=.58\textwidth]{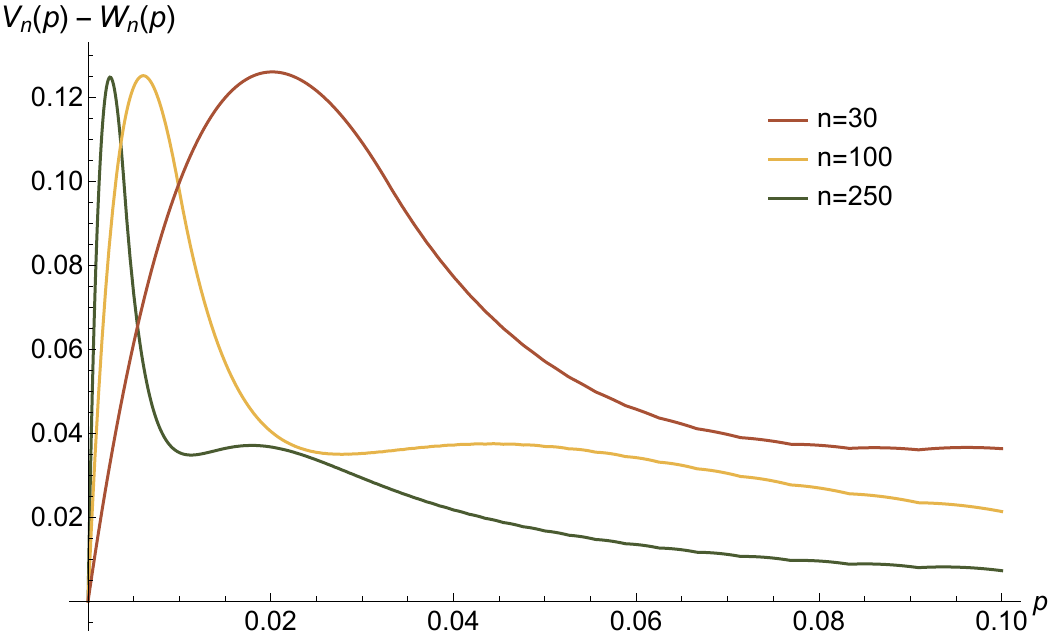}
\caption{
The deficit $V_n(p)-W_n(p)$ for $n\in\{30,100,250\}$.
}
\label{Fig3}
\end{figure}




However, it should not be seen as a negative property of the plug-in rule that convergence does not hold uniformly in~$p\in(0,1)$. As the following result shows, no rule can achieve this.

\begin{thm}
\label{thm:no_uniform_abs}
There does not exist a sequence of (possibly randomized)  rules $(\pi_n)$ such that
\begin{equation}
\label{eq:assumption_gap}
\lim_{n\to\infty}
\sup_{p\in(0,1)}
\big(V_n(p)-W_n^{\pi_n}(p)\big)
= 0
,
\end{equation}
where $W_n^{\pi_n}(p)$ is the win probability of $\pi_n$
under success probability~$p$.
\end{thm}

Theorem~\ref{thm:no_uniform_abs} is the substantive form of the impossibility. That uniform optimality should fail at a \emph{fixed} horizon is hardly surprising---Theorem~\ref{thm:no_uniform_opt} confirms it, but with finitely many observations one would not expect a single rule to be optimal at every~$p$ simultaneously. The meaningful question is whether uniformity can be recovered in the limit, as the horizon grows and~$p$ becomes ever easier to estimate. Theorem~\ref{thm:no_uniform_abs} answers that it cannot, and it does so for \emph{all} rules, not merely the \mbox{$p$-blind} ones: the obstruction is therefore not a price paid for oracle-freeness, but a property of the last-success problem itself.

Together with Theorem~\ref{thm:uniform-etan}, this settles the sparse direction, and it is worth separating the three regimes that arise. If~$np_n\to\infty$, successes accumulate fast enough for~$\hat p_t$ to identify the correct cell, and the plug-in rule matches the oracle asymptotically. If~$np_n\to0$, it is the \emph{oracle} that degenerates: $s_n(p_n)=1$ for~$n$ large and~$V_n(p_n)=np_n(1-p_n)^{n-1}\to0$, hence~$W_n(p_n)\to0$ as well; the deficit does vanish there, but only because both terms do. The window~$p\asymp1/n$ is the one place where neither escape is available: for~$p_n=c/n$ one has~$V_n(p_n)\to ce^{-c}>0$, so that the oracle value is non-degenerate, while the number of successes stays~$O(1)$, so that no estimator of~$p$ can be consistent. This is why the range in Theorem~\ref{thm:uniform-etan} is~$(0,\tilde p_n]\cup[p_n,1)$ rather than a half-line: the excluded window is not an artifact of the proof, and by Theorem~\ref{thm:no_uniform_abs} no rule can remove it.


\section{Wrap up and perspectives for future research}
\label{sec:conclu}

We investigated optimal stopping for the homogeneous last-success problem with unknown success probability~$p$, from the point of view of statistical decision theory. Beyond the main results described below, we identified regimes in which oracle-freeness is achievable: the plug-in rule matches the oracle win probability in absolute error
uniformly over $p\in[p_0,1)$ for any~$p_0>0$, achieves the optimal $1/e$ limit in sparse regimes with $p=p_n\to0$ and
$np_n\to\infty$, and attains a maximal uniform convergence statement that cannot be extended through the hardest neighborhood $p\asymp 1/n$.

Our main results concern the non-sparse regime, and they are of two kinds, operating at two different resolutions. At the level of rates, Theorems~\ref{thm:finite_oracle_ineq_p0}--\ref{thmLowerBound} provide a finite-horizon oracle inequality, valid for every~$n\ge2$ and every~$p_0\in(0,1)$ and pointwise in~$p$, together with a matching minimax lower bound; these are the statements that apply at a fixed horizon, that quantify the exponential decay of the deficit at each individual~$p$, and that serve as inputs to the sparse-regime analysis. At the level of constants, Theorems~\ref{thm:plugin_sharp}--\ref{thm:minimax_sharp} show that, for~$p_0<\frac12$, the worst-case deficit of the plug-in rule and the minimax risk over all (possibly randomized) \mbox{$p$-blind} rules both equal~$C_\star/\sqrt n$ asymptotically, with~$C_\star=\frac12\sup_{u>0}u\Phi(-u)$; the plug-in rule is thus asymptotically minimax optimal, with the exact constant. Proposition~\ref{prop:local_at_1k} shows that this is not a phenomenon attached to the single parameter value~$\frac12$: every discontinuity~$\frac1k$ of the oracle carries its own exact local minimax constant~$2\gamma_kC_\star$, again attained by the plug-in rule, and~$C_\star$ is merely the largest of them. What the worst case reports as one number is therefore the maximum of a whole family of local difficulties indexed by the oracle discontinuities, $\gamma_k$ measuring the cost of misjudging the~$k$th of them; the exact constant is the output of a complete local asymptotic minimax theory rather than an isolated computation at~$p=\frac12$. Finally, we showed that the natural sample-splitting alternative, which freezes the estimate after a fraction~$a$ of the horizon, remains rate-optimal but pays exactly~$C_\star/\sqrt{an}$, so that sequential updating is strictly preferable and the value of the discarded information is quantified by the factor~$1/\sqrt a$.

Several directions for future research appear natural. On the decision-theoretic front, the nonexistence of a greatest element at fixed horizon~$n$ motivates studying alternative principles for selecting \mbox{$p$-blind} rules, such as minimax regret or Bayes optimality, and characterizing rules that are optimal under these criteria. On the modeling front, it would be of interest to move beyond the homogeneous setting,
for instance to piecewise-constant or slowly varying success probabilities, where one may hope to retain a tractable
threshold structure while allowing for nonstationarity. Finally, on the formulation front, one could tackle the case where the horizon~$n$ is not fixed but is itself random, in the spirit of \citet{HillKrengel1991Minimax}. The combined uncertainty about the horizon and the success probability would make the resulting stopping problem substantially more complex, but also of even higher practical relevance.

\vspace{2mm}
\begin{appendix}
\section{Exact finite-horizon analysis of the plug-in rule}
\label{secExactFiniteHorizon}

This section provides an exact, computable expression for the win probability
$W_n(p)$ of the plug-in rule~\eqref{StoppingTimePlugin}. It makes the oracle
deficit $V_n(p)-W_n(p)$ explicit at any finite horizon, and it is what underlies
the numerical illustrations of the paper: Figures~\ref{Fig1}--\ref{Fig3}
and, jointly with the closed-form expression of Appendix~\ref{secProofsSampleSplit},
Table~\ref{tab:constants}.

We now derive the win probability of the plug-in rule.
First note that, 
for~$t\leq n-1$, 
the condition~$\hat{p}_t<1/(n-t+1)$ in~(\ref{StoppingTimePlugin}) is equivalent to
\begin{equation}
\label{bt}
S_t\le b_t
:=
\Bigl\lceil\frac{t}{\,n-t+1\,}\Bigr\rceil -1
.
\end{equation}
We incorporate the terminal clause~$t=n$ by setting~$b_n:=n$.
Based on the state probabilities
$$
u_{t,k}
:=
\mathbb P_p[
\hat{\tau}_n>t \text{ and } S_t=k]
,
\qquad
t=0,1,\dots,n,
\quad
 k=0,1,\dots,t
 ,
$$
the probability that the plug-in rule stops at time~$t$ is 
\begin{eqnarray*}
\lefteqn{
\ell_t(p)
:=
\mathbb P_p[\hat{\tau}_n=t]
=
\mathbb P_p[\hat{\tau}_n>t-1,X_t=1, \textrm{ and } S_t\leq b_t]
}
\\[1mm]
& & 
\hspace{15mm}
=
p\, \mathbb P_p[\hat{\tau}_n>t-1 \textrm{ and } S_{t-1}\leq b_t-1]
=
p
\sum_{k=1}^{b_t} 
u_{t-1,k-1}
.
\end{eqnarray*}
In view of~(\ref{noearlystop}), the win probability of this rule is then
$$
W_n(p)
=
\sum_{t=\lceil\frac{n}{2}\rceil+1}^n (1-p)^{n-t} \mathbb P_p[\hat{\tau}_n=t]
=
p
\sum_{t=\lceil\frac{n}{2}\rceil+1}^n (1-p)^{n-t}\sum_{k=1}^{b_t} u_{t-1,k-1}
,
$$
since, on~$\{\hat{\tau}_n=t\}$, the rule wins if and only if~$X_{t+1}=\cdots=X_n=0$, which occurs with probability $(1-p)^{n-t}$ and is independent of $\{\hat{\tau}_n=t\}$.

Now, for $t=1,\dots,n$ and~$k>0$, conditioning on $X_t$ yields\footnote{Throughout, $\mathbb I[A]$ will stand for the indicator of the condition (or set)~$A$.}
\begin{eqnarray*}
\hspace{-12mm}
u_{t,k}
\!\!&\!\! = \!\! &\!\!
p
\,
\mathbb P_p[\hat{\tau}_n>t \text{ and } S_t=k | X_t=1]
+
(1-p)
\mathbb P_p[\hat{\tau}_n>t \text{ and } S_t=k | X_t=0]
\\[2mm]
\!\!&\!\! = \!\! &\!\!
p
\,
\mathbb P_p[\hat{\tau}_n>t-1 \text{ and } S_{t-1}=k-1 | X_t=1]
\mathbb{I}[k>b_t]
\\[2mm]
& & 
\hspace{20mm}
+
(1-p)
\mathbb P_p[\hat{\tau}_n>t-1 \text{ and } S_{t-1}=k | X_t=0]
\\[1mm]
\!\!&\!\! = \!\! &\!\!
pu_{t-1,k-1}
\mathbb{I}[k>b_t]
+
(1-p)u_{t-1,k}
,
\end{eqnarray*}
whereas, for $t=1,\dots,n$ and~$k=0$, the fact that $\{S_t=0\}\subseteq \{\hat{\tau}_n>t\}$ provides
$$
u_{t,0}
=
\mathbb P_p[
\hat{\tau}_n>t \text{ and } S_t=0]
=
\mathbb P_p[S_t=0]
=
(1-p)u_{t-1,0}
.
$$
The state probabilities~$u_{t,k}$ can thus be obtained via the recursion
\begin{equation}
\begin{split}
& 
u_{t,k}
=
pu_{t-1,k-1}
\mathbb{I}[k>b_t]
+
(1-p)u_{t-1,k}
,
\qquad
t=1,\ldots,n, \quad k>0,
\\
& 
u_{t,0}=(1-p)u_{t-1,0}
,
\qquad
t=1,\ldots,n,
\end{split}
\label{eq:u_recursion}
\end{equation}
initialized at $u_{0,k}=\mathbb{I}[k=0]$ (note that $u_{t,0}=\mathbb P_p[S_t=0]=(1-p)^t$, but we keep the one-step update above to maintain a uniform dynamic-programming recursion over~$k=0,\dots,t$).

We have proved the following result.


\begin{thm}
\label{thm:exact_Wn_plugin}
For any~$n\ge 2$ and $p\in(0,1)$, the win probability of the plug-in rule is 
\begin{equation}
\label{eq:Wn_exact}
W_n(p)
=
p
\sum_{t=\lceil\frac{n}{2}\rceil+1}^n (1-p)^{n-t}
\sum_{k=1}^{b_t} u_{t-1,k-1}
,
\end{equation}
where the quantities~$u_{t,k}$ can be computed via the recursion~\eqref{eq:u_recursion}
(the corresponding stopping probabilities are then
$
\mathbb P_p[\hat{\tau}_n=t]
=
p
\sum_{k=1}^{b_t} u_{t-1,k-1}$,
for $
t=1,\ldots,n
$).
\end{thm}

The win probability~$W_n(p)$ in~(\ref{eq:Wn_exact}) is exact for any~$n$ and
$p\in(0,1)$ and can be evaluated in~$O(n^2)$ arithmetic operations; in particular,
Theorem~\ref{thm:exact_Wn_plugin} shows that~$p\mapsto W_n(p)$ is a polynomial.
It is what makes Figure~\ref{Fig1} computable, and it explains
two features visible there. The kinks at
$p=1/k$ in the right panel come from the non-differentiability of $V_n(p)$ at the
threshold transition points, whereas $W_n(p)$ is smooth, being a polynomial by
Theorem~\ref{thm:exact_Wn_plugin}. The two functions also differ in monotonicity:
$V_n(p)$ is strictly increasing in $p$ for every $n$, while $W_n(p)$ may fail to
be nondecreasing---though the deviations from monotonicity are
minute.\footnote{The smallest $n$ for which monotonicity fails is $n=60$; see
Appendix~\ref{secMonotonocity}.}

As Figure~\ref{Fig1} suggests, not knowing $p$ typically
entails a positive cost, $V_n(p)-W_n(p)>0$. In fact this holds for \emph{every} $p\in(0,1)$ as soon
as $n\ge6$: the plug-in rule never coincides with the oracle, a strict
separation that we establish, together with a description of the small horizons
$n\in\{2,3,4,5\}$ where equality can occur, in
Appendix~\ref{secFiniteSampleComparison}. The unknown-$p$ formulation moreover admits no
uniformly best rule: for every fixed $n\ge2$, the natural dominance partial
order on $p$-blind rules (where $\pi$ dominates $\pi'$ if
$W_n^\pi(p)\ge W_n^{\pi'}(p)$ for all $p$) has no greatest element, even if one
allows randomization. These finite-horizon obstructions---although not
surprising, and not needed for the quantitative theory below---clarify why one
must adopt an asymptotic and quantitative criterion rather than seek
finite-horizon uniform optimality---the route followed in the main text, where
the analysis proceeds through the minimax results of Sections~2--3. Both
obstructions are proved in Appendix~\ref{secFiniteSampleBarrier}.



\section{Auxiliary results}
\label{secAuxiliaryResults}

This appendix collects the finite-horizon obstructions announced in
Appendix~\ref{secExactFiniteHorizon}. They are not needed for the quantitative
theory of Sections~\ref{sec:oracle_bounds}--\ref{sec:sparse}, but they clarify
the structure of the unknown-$p$ formulation: the plug-in rule never attains the
oracle win probability once $n\ge6$, and no \mbox{$p$-blind} rule---randomized or
not---can be uniformly optimal at a fixed horizon.

\subsection{Finite sample comparison}
\label{secFiniteSampleComparison}

We turn to the strict separation between the plug-in rule and the oracle.
\begin{prop}
\label{prop:no_p_equality_n_ge_6}
(i) For $n\in\{2,3\}$, one has $W_n(p)=V_n(p)$ if and only if $p\in[\tfrac12,1)$.
\\
(ii) For $n\in\{4,5\}$, one has $W_n(p)=V_n(p)$ if and only if $p=\tfrac12$.
\\
(iii) For $n\ge 6$, there is no~$p\in(0,1)$ such that~$W_n(p)=V_n(p)$.
\end{prop}

\begin{proof}[Proof of Proposition~\ref{prop:no_p_equality_n_ge_6}]
(i)--(ii) Note that the oracle win probability in~(\ref{WinOracle}) can be written in the familiar piecewise form
$$
V_n(p)=
\begin{cases}
\ np(1-p)^{n-1} & \textrm{if } 0<p<\tfrac1n\\
\ (n-1)\,p(1-p)^{n-2} & \textrm{if } \tfrac1n\le p<\tfrac1{n-1}\\
\ \hspace{8mm}\vdots & \hspace{8mm}\vdots\\
\ 2p(1-p) & \textrm{if }  \tfrac13\le p<\tfrac12\\
\ p & \textrm{if } \tfrac12\le p<1.
\end{cases}
$$
Also, by specializing Theorem~\ref{thm:exact_Wn_plugin} to the corresponding values of~$n$, one obtains that
$$
W_2(p)=p,
\qquad
W_3(p)=p,
\qquad
W_4(p)=p(1-p)^3+p-p^2(1-p)^2,
$$
and
$$
W_5(p)=p(1-p)^4+p-p^2(1-p)^3.
$$
We verify (i)--(ii) by a case analysis for~$n=2,3,4,5$.


\smallskip
\noindent\emph{Case $n=2$.}
For $0<p<\tfrac12$,
$
V_2(p)-W_2(p)=2p(1-p)-p=p(1-2p)>0,
$
while for $\tfrac12\le p<1$ we have $V_2(p)-W_2(p)=p-p=0$.


\smallskip
\noindent\emph{Case $n=3$.}
If $p\ge\tfrac12$, then $V_3(p)-W_3(p)=p-p=0$.
If $p\in[\tfrac13,\tfrac12)$, then
\[
V_3(p)-W_3(p)=2p(1-p)-p=p(1-2p)>0.
\]
If $p\in(0,\tfrac13)$, then
\[
V_3(p)-W_3(p)=3p(1-p)^2-p=p\bigl(3(1-p)^2-1\bigr)>0,
\]
since $p<\tfrac13$ implies $1-p>\tfrac23$, hence $3(1-p)^2>4/3$.


\smallskip
\noindent\emph{Case $n=4$.}
A direct algebraic simplification gives:
\[
\begin{aligned}
p\in[\tfrac12,1):\quad &V_4(p)-W_4(p)=p(1-p)^2(2p-1)\ge 0, \ \text{with equality only if }p=\tfrac12,\\
p\in[\tfrac13,\tfrac12):\quad &V_4(p)-W_4(p)=p^2(2-p)(1-2p)>0,\\
p\in[\tfrac14,\tfrac13):\quad &V_4(p)-W_4(p)
=p\{2(1-p)^2(1+p)-1\},\\
p\in(0,\tfrac14):\quad &V_4(p)-W_4(p)
=p\{(1-p)^2(3-2p)-1\}
.
\end{aligned}
\]
We consider the last two cases.

\begin{itemize}
\item
If $p\in[\tfrac14,\tfrac13)$, then $(1-p)^2\ge 4/9$ and $1+p\ge 5/4$, hence
$
2(1-p)^2(1+p) \ge 
10/9>1$,
which shows that~$V_4(p)-W_4(p)>0$. 
\vspace{2mm}

\item
If $p\in(0,\tfrac14)$, then $(1-p)^2>9/16$ and $3-2p>5/2$, hence
$
(1-p)^2(3-2p) >
45/32>1$,
so that~$V_4(p)-W_4(p)>0$. 
\end{itemize}

\noindent
Therefore, $V_4(p)>W_4(p)$ for all $p\neq\tfrac12$, and $V_4(\tfrac12)=W_4(\tfrac12)=\tfrac12$.


\smallskip
\noindent\emph{Case $n=5$.}
Again, simplifying $V_5(p)-W_5(p)$ on each interval yields:
\[
\begin{aligned}
p\in[\tfrac12,1):\quad &V_5(p)-W_5(p)=p(1-p)^3(2p-1)\ge 0,\ \text{with equality only if }p=\tfrac12,\\
p\in[\tfrac13,\tfrac12):\quad &V_5(p)-W_5(p)=p^2(1-2p)(p^2-3p+3)>0,\\
p\in[\tfrac14,\tfrac13):\quad &V_5(p)-W_5(p)
= p\{(1-p)^2 (2+3p-2p^2) -1\}
,\\
p\in[\tfrac15,\tfrac14):\quad &V_5(p)-W_5(p)
= p\{(1-p)^3(3+2p)-1 \}
,\\
p\in(0,\tfrac15):\quad &V_5(p)-W_5(p)
= p\{(1-p)^3(4-3p)-1 \}
.
\end{aligned}
\]
We consider the last three cases.

\begin{itemize}
\item
If $p\in[\tfrac14,\tfrac13)$, then $(1-p)^2\ge 4/9$ and
$2+3p-2p^2\geq 21/8$, hence~$(1-p)^2 (2+3p-2p^2)\ge 7/6$, which yields~$V_5(p)-W_5(p)>0$.
\vspace{2mm}

\item
If $p\in[\tfrac15,\tfrac14)$, then $(1-p)^3\ge 27/64$ and $3+2p\ge 17/5$, hence $(1-p)^3(3+2p) \ge 459/320$, so that $V_5(p)-W_5(p)>0$.
\vspace{2mm}

\item
If $p\in(0,\tfrac15)$, then $(1-p)^3>64/125$ and $4-3p>17/5$, hence $(1-p)^3(4-3p) > 1088/625$, which implies again that $V_5(p)-W_5(p)>0$.
\end{itemize}

\noindent
Thus, $V_5(p)>W_5(p)$ for all $p\neq\tfrac12$, and $V_5(\tfrac12)=W_5(\tfrac12)=\tfrac12$.
\vspace{3mm}

(iii) 
Let
$
t_\star:=\lceil\frac{n}{2}\rceil+1
$ and note that, since $n\ge 6$, we have $4\leq t_\star\le n-2$. Recalling~(\ref{noearlystop}), the plug-in rule cannot stop before time~$t_\star$. Moreover, since
$$
1<\frac{t_\star}{n-t_\star+1}\le 2
,
$$
we have~$b_{t_\star}=1$, so at time $t_\star$ the plug-in rule stops on a success if and only if $S_{t_\star}\le 1$, \emph{i.e.}, if and only if~$S_{t_\star-1}=0$; see~(\ref{bt}).
Consider then the events
$$
E_{\mathrm{stop}}
:=\{S_{t_\star-1}=0, X_{t_\star}=1\}
\quad
\textrm{ and }
\quad
E_{\mathrm{cont}}
:=\{S_{t_\star-1}=1, X_{t_\star}=1\}
.
$$
On~$E_{\mathrm{stop}}$, we have $S_{t_\star}=1\le b_{t_\star}=1$, so the plug-in rule stops at $t_\star$, whereas on~$E_{\mathrm{cont}}$, we have $S_{t_\star}=2>b_{t_\star}=1$, so the plug-in rule does not stop at $t_\star$. Thus, at the same time~$t_\star$ and on the same observation $X_{t_\star}=1$, the plug-in rule sometimes stops and sometimes continues, depending on the past (since~$\mathbb P_p[E_{\mathrm{stop}}]=(1-p)^{t_\star-1}p$ and $\mathbb P_p[E_{\mathrm{cont}}]=(t_\star-1)p^2(1-p)^{t_\star-2}$,
both events have positive probability under~$\mathbb P_p$ for any $p\in(0,1)$).
\vspace{2mm}

Now fix $p\in(0,1)$ and consider the homogeneous known-$p$ problem. By the sum-the-odds theorem, there exists a \emph{threshold} rule---that is, a rule that stops on the first success on or after some deterministic time~$s$---that is optimal. More precisely,  $s$ is the quantity~$s_n(p)$ in~(\ref{eq:sn_homo}),
and the resulting optimal win probability is~$V_n(p)$ in~(\ref{WinOracle}). In the boundary case when~$p=1/(m+1)$ for some positive integer~$m$, there are exactly two oracle-optimal thresholds, based on~$s=n-m+1$ and~$\tilde s=s-1=n-m$
(indeed, the threshold rule using~$\tilde s$ wins if and only if there is exactly one success in~$\{n-m,n-m+1,\ldots,n\}$, which occurs with probability
$$
(m+1)p(1-p)^m
=
\Big( \frac{m}{m+1} \Big)^m
=
mp(1-p)^{m-1}
=
V_n(p)
,
$$
and it is easy to check that all other thresholds provide a strictly lower win probability).
Crucially, this information is enough to pin down the optimal action after observing~\mbox{$X_t=1$}, even though the sum-the-odds theorem does not state that any optimal rule must be a threshold rule. Indeed, assume that one is at some time $t\in\{1,\dots,n-1\}$ and has not stopped yet, and that one observes $X_t=1$. If one stops at~$t$, then the conditional win probability is~$(1-p)^{n-t}$. If one continues, then 
the maximal conditional win probability from time~$t+1$ onward is precisely the oracle win probability for a horizon of length $n-t$, namely~$V_{n-t}(p)$.
Therefore, the sign of the strict comparison
$$
(1-p)^{n-t} \ \text{ versus }\ V_{n-t}(p)
$$
determines whether optimality forces ``stop'' or ``continue'' at time~$t$ upon observing $X_t=1$.

Now, because the sum-the-odds theorem characterizes the oracle-optimal thresholds as above (and in the boundary case yields
exactly two adjacent optimal thresholds), there is \emph{at most one} time index at which the two actions (stop/continue upon observing $X_t=1$) can be tied,
namely $t=s-1$ in the boundary case $p=1/(m+1)$. At all other times, the comparison is strict, and hence
every optimal rule (threshold or not) must take a \emph{deterministic} action upon observing $X_t=1$.
Consequently, if the plug-in rule were optimal at $p$, then at time $t_\star$ it would have to take a
deterministic action upon observing~$X_{t_\star}=1$, except possibly in the single boundary situation
where~$t_\star$ coincides with that unique “tie time’’ $s-1$. This allows us to conclude the proof by considering two cases.
\vspace{2mm}

\emph{Case~(a): $t_\star$ is not the tie time for $p$.}
Then, as explained above, optimality deterministically forces either to stop or not to stop on $X_{t_\star}=1$. However, we have seen that the plug-in rule stops on~$E_{\mathrm{stop}}$ and continues on $E_{\mathrm{cont}}$, and both events have positive probability. Therefore the plug-in rule cannot be optimal, and~$W_n(p)<V_n(p)$.
\vspace{2mm}

\emph{Case~(b): $t_\star$ is the tie time for $p$.}
Then, the discussion above implies that~$p$ must be equal to the unique parameter value
$$
p_\star:=\frac{1}{n-t_\star+1}
$$
for which $t_\star=s_n(p_\star)-1$, and the two oracle-optimal thresholds are $s=t_\star$ and $s=t_\star+1$. In particular, at time $s=t_\star+1$ there is no tie: the oracle-optimal action upon observing~\mbox{$X_{t_\star+1}=1$} is uniquely determined and consists in stopping. 
We now show that the plug-in rule fails to take this unique optimal action at time $s=t_\star+1$ with positive
probability, hence cannot be optimal at~$p_\star$.
First note that
$$
b_{t_\star+1}=\biggl\lceil\frac{t_\star+1}{n-(t_\star+1)+1}\biggr\rceil-1
\le 2
.
$$
Consider the event
$
F:=\{S_{t_\star-1}=1,\ X_{t_\star}=1,\ X_{t_\star+1}=1\}$.
On $F$, we have $X_{t_\star+1}=1$ and
$
S_{t_\star+1}=S_{t_\star-1}+X_{t_\star}+X_{t_\star+1}=3>b_{t_\star+1}$,
so the plug-in stopping condition $S_{t_\star+1}\le b_{t_\star+1}$ fails and the plug-in rule does not
stop at time $t_\star+1$ despite $X_{t_\star+1}=1$. Since~$\mathbb P_{p_\star}[F]=(t_\star-1)p_\star^3(1-p_\star)^{t_\star-2}>0$, the plug-in rule violates the (strict) optimal action at time $s=t_\star+1$ on a set of positive probability. Therefore, it is not optimal and~$W_n(p_\star)<V_n(p_\star)$.
\vspace{2mm}

Combining the two cases shows that for every $p\in(0,1)$ one has $W_n(p)<V_n(p)$, which establishes the result.
\end{proof}

\subsection{Monotonicity}
\label{secMonotonocity}
In this section, we provide a computer-assisted yet fully rigorous verification that the win probability~$W_n(p)$ of the plug-in rule is nondecreasing in~$p\in (0,1)$ for all $n\le 59$, whereas monotonicity fails for
$n=60$. 
The argument relies on the following result.

\begin{lem}
\label{lem:Wn_polynomial}
For any~$n\ge 2$,  $W_n(p)$ is a
polynomial in $p$ with integer coefficients, and so is its derivative~$W_n'(p)$.
\end{lem}

\begin{proof}
Fix $n\ge 2$. Note that Theorem~\ref{thm:exact_Wn_plugin} implies that
$$
W_n(p)
=
p
\sum_{t=\lceil\frac{n}{2}\rceil+1}^n 
\sum_{j=0}^{n-t} 
\sum_{k=1}^{b_t}
\textstyle{n-t \choose j} 
(-1)^j p^j
 u_{t-1,k-1}(p)
,
$$
where the quantities~$u_{t,k}(p)$, $t=0,1,\ldots,n$, $k=0,1,\ldots,t$, satisfy the recursion
\begin{align*}
u_{t,k}(p) &= p\,u_{t-1,k-1}(p)\,\mathbb I[k>b_t] + (1-p)u_{t-1,k}(p)
,
\qquad t=1,\dots,n,\ k\ge 1, \\
u_{t,0}(p) &= (1-p)u_{t-1,0}(p), \qquad t=1,\dots,n,
\end{align*}
initialized at $u_{0,k}(p)=\mathbb{I}[k=0]$. Since~$\mathbb I[k>b_t]\in\{0,1\}$ does not depend on~$p$, an induction argument directly yields that each~$u_{t,k}(p)$ is a polynomial in~$p$ with integer coefficients. It follows that~$W_n(p)$, hence also~$W_n'(p)$, is a polynomial in~$p$ with integer coefficients.
\end{proof}

\begin{prop}
\label{prop:first_failure_monotonicity}
For all $n\le 59$, the function $p\mapsto W_n(p)$ is nondecreasing on $(0,1)$. For~$n=60$, this function is  not
monotone on~$(0,1)$.
\end{prop}

\begin{proof}
By Lemma~\ref{lem:Wn_polynomial}, the map $p\mapsto W_n(p)$ is $C^1$ on $(0,1)$. Hence it fails to be nondecreasing on $(0,1)$ if and only if
\begin{equation}
\label{FirstOrderSentence}
\exists\,p\in(0,1)\ \text{such that}\ W_n'(p)<0.
\end{equation}
Since $W_n'(p)$ is a polynomial with integer (hence rational) coefficients, deciding the first-order sentence~(\ref{FirstOrderSentence}) is an exact decision problem in real algebraic geometry and can be resolved by quantifier elimination over the reals.

We performed an exact symbolic verification for $n\in\{2,3,\dots,60\}$ using real quantifier elimination
on the formula $(0<p<1)\ \wedge\ (W_n'(p)<0)$. The outcome is:
(i) for every $n\le 59$, the formula is unsatisfiable, hence $W_n'(p)\ge 0$ for all~$p\in(0,1)$, so that~$p\mapsto W_n(p)$ is nondecreasing on~$(0,1)$; (ii) for $n=60$, the formula is satisfiable, and the computation returns an explicit nonempty semi-algebraic set of values of~$p$ (in fact, an open interval~$\mathcal{I}=(\varphi,\psi)$ with algebraic endpoints) on which $W_{60}'(p)<0$. Thus, $p\mapsto W_{60}(p)$ is not nondecreasing on~$(0,1)$.  Since~$p^n\leq W_n(p)\leq np$ (the lower-bound results from the fact that the plug-in rule wins on $\{X_1=1,\ldots,X_n=1\}$, whereas the upper-bound was established in the proof of Theorem~\ref{thm:uniform-etan}), we have
$$
\lim_{p\stackrel{>}{\to} 0} 
W_{n}(p)=0
\qquad
\textrm{ and }
\qquad
\lim_{p\stackrel{<}{\to} 1} 
 W_{n}(p)=1
$$
for all~$n\geq 2$, which implies that~$p\mapsto W_{60}(p)$ is not nonincreasing on~$(0,1)$. Therefore,  $p\mapsto W_{60}(p)$ is not monotone on~$(0,1)$.
\end{proof}

For reproducibility purposes, we provide  the following \textsc{Mathematica} code that constructs $W_n(p)$ exactly as a polynomial in $p$ with integer coefficients from Theorem~\ref{thm:exact_Wn_plugin}, differentiates it symbolically, and then uses \texttt{Reduce[..., Reals]} to decide whether the derivative~$W_n'(p)$ is negative for some $p\in(0,1)$.

\begin{lstlisting}
(* Define auxiliary quantities *)
b[t_, n_] := If[t < n, Ceiling[t/(n - t + 1)] - 1, n];

(* Obtain the exact expression for W_n(p)from Theorem 2.1 *)
Wp[n_] := Module[{uPrevious, uCurrent, ell, t, k, bt, W, hnLocal},
   hnLocal = Ceiling[n/2] + 1;
   uPrevious = ConstantArray[0, n + 1];
   uPrevious[[1]] = 1;
   ell = ConstantArray[0, n + 1];
   For[t = 1, t <= n, t++, bt = b[t, n];
    uCurrent = ConstantArray[0, n + 1];
    For[k = 0, k <= t, k++, 
     uCurrent[[k + 1]] = (1 - p) uPrevious[[k + 1]] + 
        If[k >= 1 && k > bt, p uPrevious[[k]], 0];];
    ell[[t + 1]] = If[bt >= 1, p Sum[uPrevious[[k]], {k, 1, bt}], 0];
    uPrevious = uCurrent;];
   W = Sum[(1 - p)^(n - t) ell[[t + 1]], {t, hnLocal, n}];
   Expand[W]];

(* Obtain the exact expression for the derivative of W_n(p) *)
WpPrime[n_Integer] := Expand[D[Wp[n], p]];

(* Decide existence of p in (0,1) with D_n(p)<0 for all n in {2,...,\
Nmax} *)
allFailures[Nmax_] := 
 Module[{n, cond, res = {}}, 
  For[n = 2, n <= Nmax, n++, 
   cond = Reduce[0 < p < 1 && WpPrime[n] < 0, p, Reals];
   If[cond =!= False, AppendTo[res, {n, cond}]];];
  res]
  
allFailures[60] 
\end{lstlisting}

The code returns~$n=60$ as the only value of~$n\in\{2,3,\ldots,60\}$ for which the derivative becomes negative on~$(0,1)$, and indicates that the domain on which it is negative is~$\mathcal{I}=(\varphi,\psi)$, where the algebraic endpoints are (up to four decimal digits)~$\varphi=0.0537$ and~$\psi=0.0602$. Figure~\ref{Fig4} illustrates the lack of monotonicity of~$p\mapsto W_{60}(p)$ and shows the plot of the monotone function~$p\mapsto W_{59}(p)$ for the sake of comparison.

\begin{figure}[h!]
\hspace*{-4.5mm}
\includegraphics[width=1.035\textwidth]{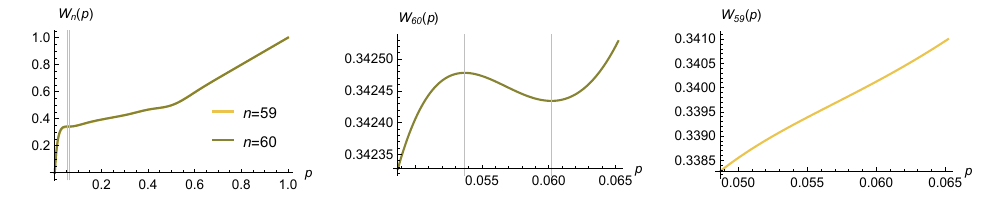}
\caption{(Left panel:) $W_n(p)$ as a function of $p\in(0,1)$ for $n=59$ and $n=60$ (the curves are visually indistinguishable at this scale); the vertical lines indicate~$\varphi$ and~$\psi$, the algebraic endpoints of the interval~$\mathcal{I}$ on which~$W_{60}(p)$ is monotone decreasing. (Middle panel:) zoom of $W_{60}(p)$ on a region containing~$\mathcal{I}$. (Right panel:) the same zoom for $W_{59}(p)$.
}
\label{Fig4}
\end{figure}

We stress that the proof of Proposition~\ref{prop:first_failure_monotonicity} above is ``computer-assisted'' only in the sense that a certified exact algebraic procedure (quantifier elimination) is invoked to decide the sign of an integer polynomial on an interval, but that the procedure is exact (no floating-point arithmetic is involved: \texttt{p} is symbolic and \texttt{Expand}, \texttt{D}, and \texttt{Reduce} are executed in exact arithmetic).

\subsection{Finite sample barrier}
\label{secFiniteSampleBarrier}

Consider now the homogeneous stopping problem with fixed horizon~$n\ge 2$, and
recall the win probability~$W_n^\pi(p)$ of a rule~$\pi$ defined
in~(\ref{eq:Wpi_def}). For rules~$\pi$, $\pi'$ and~$\pi^\star$, we say that
$\pi$ \emph{dominates} $\pi'$ if
$$
W_n^\pi(p)\ge W_n^{\pi'}(p)\quad\text{for all }p\in(0,1),
$$
and that~$\pi^\star$ is an $n$-optimal rule if~$\pi^\star$ dominates every other
rule (the uniformity in~$p$ in these definitions encodes the unknown-$p$ nature
of the stopping problem).
For any fixed~$p_\star\in(0,1)$, one may use the $p_\star$-oracle rule as a \mbox{$p$-blind} rule (this rule will be optimal if~$p=p_\star$, but of course it is expected to perform poorly if~$p$ is far from~$p_\star$). Comparing then against the fixed oracle rule at, \emph{e.g.},~$p_\star=\frac{1}{4}$, a direct corollary of Proposition~\ref{prop:no_p_equality_n_ge_6} is that there is no~$n\geq 2$ for which the plug-in rule is $n$-optimal. As the following result shows, however, this is not a deficiency of the plug-in rule; instead, it reflects the intrinsic difficulty of the unknown-$p$ stopping problem.


\begin{thm}
\label{thm:no_uniform_opt}
There is no~$n\ge 2$ for which an $n$-optimal rule exists, and this is the case even if one allows for randomized rules.
\end{thm}

\begin{proof}
Fix~$n\geq 2$, and assume \emph{ad absurdum} that $\pi^\star$ is a (possibly randomized) $n$-optimal rule. If $\pi^\star$ is randomized, we realize its internal randomization by an auxiliary random variable $U$, defined on the same probability space, independent of $(X_1,\ldots,X_n)$ and with a distribution that does not depend on $p$. We then write the (possibly randomized) stopping time associated with $\pi^\star$ as
$
\tau_n^\star=\tau_n^\star(X_1,\ldots,X_n;U)$,
and probabilities involving the randomization are taken with respect to $U$ (conditionally on the observed $(X_1,\ldots,X_n)$).

For each~$t\in\{1,\dots,n\}$, denote as~$A_t$ the event that there is exactly one success, occurring at time~$t$:
$
A_t:=\{X_1=\cdots=X_{t-1}=0,\, X_t=1,\, X_{t+1}=\cdots=X_n=0\}
.
$
Then, $\mathbb{P}_p[A_t]=p(1-p)^{n-1}$ for any $t$.
On $A_t$, the last success is at time $t$, so $\pi^\star$ wins on $A_t$ if and only if it stops at time $t$ when it sees the success at $t$. Let then
$$
\alpha_t
:=
\mathbb{P}[\text{$\pi^\star\!$ stops at time $t$}  |  A_t]
=
\mathbb{P}[\tau^\star_n=t  |  A_t]
\in[0,1]
,
$$
where the probability~$\mathbb{P}$ is over the internal randomization variable $U$ (equivalently, under the joint law of $(X_1,\ldots,X_n,U)$, conditional on $A_t$; since $A_t$ depends only on $X_1,\ldots,X_n$, conditioning on $A_t$ does not affect the law of $U$). Since~$\pi^\star$ cannot win when~$\sum_{t=1}^n X_t=0$, the total probability formula provides
\begin{eqnarray}
W_n^{\pi^\star}(p)
&=&
\sum_{t=1}^n 
\alpha_t 
\mathbb{P}_p[A_t]
+
\mathbb{P}_p
\big[\text{$\pi^\star\!$ wins}|
\textstyle\sum_{t=1}^n X_t\ge 2
\big]
\mathbb{P}_p
\big[
\textstyle\sum_{t=1}^n X_t\ge 2
\big]
\nonumber
\\[0mm]
&\le & 
p(1-p)^{n-1}
\sum_{t=1}^n \alpha_t
+
2^n 
p^2
,
\label{eq:W_upper_with_p2}
\end{eqnarray}
since~$\sum_{t=1}^n X_t\sim {\rm  Bin}(n,p)$ yields
$
\mathbb{P}_p
\big[
\textstyle\sum_{t=1}^n X_t\ge 2
\big]
=
p^2 \sum_{k=2}^n \binom{n}{k}
\leq
2^n 
p^2
$.

Now, let $\pi^{\mathrm{first}}$ be the rule that stops at the first time $t$ such that $X_t=1$ (if any). Of course,  this rule wins if and only if there is exactly one success,
so 
\begin{equation}
\label{eq:W_first_exact}
W_n^{\pi^{\mathrm{first}}}(p)
=
\sum_{t=1}^n \mathbb{P}_p[A_t]
= 
np(1-p)^{n-1}
.
\end{equation}
Since $\pi^\star$ dominates $\pi^{\mathrm{first}}$, we have~$W_n^{\pi^\star}(p)\ge W_n^{\pi^{\mathrm{first}}}(p)$ for all~$p\in(0,1)$, so~\eqref{eq:W_upper_with_p2}--\eqref{eq:W_first_exact} yield
$$
p(1-p)^{n-1}
\sum_{t=1}^n \alpha_t
+
2^n p^2
 \ge 
 np(1-p)^{n-1}
\quad\text{for all }
p\in(0,1)
 .
$$
Dividing by~$p$ and letting $p\to 0$ gives
$
\sum_{t=1}^n \alpha_t  \ge n.
$
Since each $\alpha_t\le 1$, we must then have that~$\alpha_t=1$ for all~$t$. In particular,
$\alpha_1=\mathbb{P}[\tau^\star_n=1  |  A_1]=1$. Since~$\tau^\star_n$ is a (possibly randomized) stopping time, the event $\{\tau^\star_n=1\}$ is $\sigma(X_1,U)$-measurable; moreover, $A_1=\{X_1=1\}\cap\{X_2=\cdots=X_n=0\}$ and $\{X_2=\cdots=X_n=0\}$ is independent of~$\sigma(X_1,U)$. Therefore,
$
\mathbb{P}[\tau^\star_n=1  |  A_1]
=
\mathbb{P}[\tau^\star_n=1  |  X_1=1]
$,
so that $\mathbb{P}[\tau^\star_n=1  |  X_1=1]=1$, \emph{i.e.}, $\pi^\star$ almost surely stops at time~$1$ whenever $X_1=1$. Thus, for any~$p\in(0,1)$, we have 
\begin{eqnarray*}
W_n^{\pi^\star}(p)
&=& 
\mathbb{P}_p[\text{$\pi^\star$ wins and }X_1=1]
+
\mathbb{P}_p[\text{$\pi^\star$ wins and }X_1=0]
\\[2mm]
&=&
\mathbb{P}_p[X_1=1, X_2=\cdots=X_n=0]
+
\mathbb{P}_p[\text{$\pi^\star$ wins and }X_1=0]
\\[2mm]
&\leq &
p(1-p)^{n-1}+1-p
.\end{eqnarray*}

Let $\pi^{\mathrm{last}}$ be the rule that never stops before time~$n$ and stops on time~$n$ if~$X_n=1$. Since its win probability is
$
W_n^{\pi^{\mathrm{last}}}(p)
=
\mathbb{P}_p[X_n=1]=p
$
and since $\pi^\star$ dominates $\pi^{\mathrm{last}}$, we must have
$$
p(1-p)^{n-1}+1-p
\geq
p
\quad\text{for all }
p\in(0,1)
.
$$
Since this fails for large~$p$, we obtain a contradiction. Thus, no $n$-optimal rule exists.
\end{proof}


\section{Proofs for Section~2}
\label{appProofsRates}
Both the proof of Theorem~\ref{thm:finite_oracle_ineq_p0}---deferred, on account of its length, to Appendix~\ref{secProofsUpperBound}---and the proofs of Section~\ref{sec:constants} below rest on the following exponential controls.

\begin{lem}
\label{lem:Etails}
Fix $p_0\in(0,1)$ and~$\delta_0>0$. With~$h_n:=\lceil \tfrac n2\rceil+1$, let
$$
E_n^{\delta_0}(p)
:=
\bigg\{\max_{h_n\leq t\leq n}|\hat p_t-p|\le \delta_0\bigg\}
,
$$
and, with~$r_t:=1/(n-t+1)$ and~$M_0
=\lceil 1/p_0\rceil-1$, let
$$
H_t
:=
\Big\{\big|\hat p_{t-1}-r_t\big|\leq \frac{1}{t}\Big\}
\quad
\textrm{ for }
t\in T_n:=\{\max\{1,n-M_0\},\dots,n-1\}
.
$$
Then, there exist positive constants $C_1(\delta_0),c_1(\delta_0),C_2(p_0)$, and $c_2(p_0)$ such that
$$
\mathbb P_p[(E_n^{\delta_0}(p))^c]
\le
C_1(\delta_0)e^{-c_1(\delta_0)n}
\quad
\textrm{ and }
\quad
\mathbb P_p[H_t]
\le
C_2(p_0) e^{-c_2(p_0)n(p-r_t)^2}
$$
for all $n\ge 2$, all $p\in[p_0,1)$, and all~$t\in T_n$.
\end{lem}

\begin{proof}
By Hoeffding's inequality, for every $t\in\{1,\dots,n\}$ and $p\in(0,1)$,
$$
\mathbb P_p\big[|\hat p_t-p|>\delta_0\big]\le 2e^{-2t\delta_0^2}.
$$
A union bound over $t\in\{h_n,\dots,n\}$ yields
$$
\mathbb P_p[(E_n^{\delta_0}(p))^c]
\le
2\sum_{t=h_n}^{\infty} e^{-2t\delta_0^2}
=
\frac{2e^{-2h_n\delta_0^2}}{1-e^{-2\delta_0^2}}
\le
C_1(\delta_0)e^{-c_1(\delta_0)n},
$$
for positive constants $C_1(\delta_0),c_1(\delta_0)$. This establishes the result for~$\mathbb P_p[(E_n^{\delta_0}(p))^c]$.

We then turn to~$\mathbb P_p[H_t]$. Assume that~$n\ge 2(M_0+1)$ and fix $t\in T_n$. On~$H_t$, we have
\[
|\hat p_{t-1}-p|
\ge |p-r_t|-|\hat p_{t-1}-r_t|
\geq 
|p-r_t|-\frac{1}{t},
\]
hence
\[
H_t
\subseteq
\Big\{|\hat p_{t-1}-p|
\geq
 (|p-r_t|-1/t)_+\Big\}
,
\]
with~$(x)_+:=\max\{x,0\}$. Therefore, by Hoeffding's inequality,
\begin{equation}
\label{eq:Ht_hoeffding}
\mathbb P_p[H_t]
\le
2\exp\!\big(\!-2(t-1)(|p-r_t|-1/t)_+^2\big).
\end{equation}

We now distinguish two cases.

\smallskip
\noindent\emph{Case (a): $|p-r_t|\ge 4/n$.}
Since $n\ge 2(M_0+1)$ and $t\in T_n$, we have $t\ge n-M_0\ge n/2$, hence $1/t\le 2/n$.
Thus, in this case,
\[
|p-r_t|-\frac{1}{t}
\ge
|p-r_t|-\frac{2}{n}
\ge
\frac{|p-r_t|}{2},
\]
so that \eqref{eq:Ht_hoeffding} yields
\[
\mathbb P_p[H_t]
\le
2\exp\!\Big(\!-\frac{(t-1)|p-r_t|^2}{2}\Big).
\]
Since $t-1\ge n-M_0-1\ge n/2$ for $n\ge 2(M_0+1)$, we obtain
\begin{equation}
\label{eq:Ht_case_a}
\mathbb P_p[H_t]
\le
2\exp\!\Big(\!-\frac{n|p-r_t|^2}{4}\Big).
\end{equation}

\smallskip
\noindent\emph{Case (b): $|p-r_t|<4/n$.}
In this case, we use the trivial bound $\mathbb P_p[H_t]\le 1$ together with
\[
\exp\!\Big(\!-\frac{n|p-r_t|^2}{4}\Big)
\ge
\exp\!\Big(\!-\frac{4}{n}\Big)
\ge e^{-2}
,
\]
which gives
\begin{equation}
\label{eq:Ht_case_b}
\mathbb P_p[H_t]
\le
e^{2}\exp\!\Big(\!-\frac{n|p-r_t|^2}{4}\Big).
\end{equation}

Combining \eqref{eq:Ht_case_a} and \eqref{eq:Ht_case_b}, we conclude that for all $n\ge 2(M_0+1)$ and all $t\in T_n$,
\[
\mathbb P_p[H_t]
\le
C_2(p_0) e^{-c_2(p_0)n(p-r_t)^2}
,
\]
with~$C_2(p_0):=e^2$ and~$c_2(p_0):=\frac14$.
This establishes the result since the bound extends to the finitely many values $n\in\{2,\ldots,2M_0+1\}$ after possibly increasing~$C_2(p_0)$.
\end{proof}


\subsection{Proof of Theorem~\ref{thm:finite_oracle_ineq_p0}}
\label{secProofsUpperBound}

\begin{proof}[Proof of Theorem~\ref{thm:finite_oracle_ineq_p0}]
(i) Fix $p_0\in(0,1)$. Note that~$p_0\in[\frac{1}{M_0+1},\frac{1}{M_0})$.
Define the deterministic margin
\begin{equation}
\label{deltap0Defn}
\delta_0
:=
\Bigg\{ 
\begin{array}{ll}
\frac{1}{3M_0(M_0+1)}
& 
\textrm{if } p_0< \frac{1}{2}
\\[2mm]
\frac{1}{12}
& 
\textrm{if } p_0\geq \frac{1}{2}.
\end{array}
\end{equation}
For~$p_0< \frac{1}{2}$, the cardinality of~$\mathcal B_{p_0}$ is at least~$2$, and we have
$$
\delta_0
=
\frac13\min\bigl\{|q-q'|:\, q,q'\in\mathcal B_{p_0},\, q\neq q'\bigr\}
=
\frac13
\bigg( 
\frac{1}{M_0}
-
\frac{1}{M_0+1}
\bigg)
.
$$
For~$p_0\ge \frac12$, we have~$M_0=1$ and~$\mathcal B_{p_0}=\{\frac12\}$, so that~$\Delta_p=|p-\frac12|$ for all $p\in[p_0,1)$. The proof below still applies with the choice
$\delta_0=\frac{1}{12}$, and the uniqueness arguments involving the closest boundary point are immediate since
$\mathcal B_{p_0}$ is a singleton. Hence, it suffices to treat the case $p_0<\frac12$ in the remainder of the proof.

The proof decomposes into five steps. Throughout, we will assume that~$n \geq 2(M_0+3)$ (this is without any loss of generality, since the case with smaller values of~$n$ can be covered by absorbing constants, as we just did in the proof of Lemma~\ref{lem:Etails}).

\paragraph*{Step 1: the bound in~(\ref{eq:p0_pointwise_distance}) holds for~$p\in\mathcal B_{p_0}$}
Fix~$p\in\mathcal B_{p_0}$, so that~$p=1/(m+1)$ for some $m\in\{1,\ldots,M_0\}$ and~$\Delta_p=0$. Since~$n\geq M_0\geq m=m(p)$, the  \mbox{$p$-oracle} rule stops at the first success (if any) from~$s_n(p)=n-m+1$ onwards; see~(\ref{eq:sn_homo}). As shown in the proof of Proposition~\ref{prop:no_p_equality_n_ge_6}(iii), the win probability~$V_n(p)$ of the \mbox{$p$-oracle} rule is the same as the win probability of the rule stopping on the first success (if any) from~$\tilde{s}=n-m$ onwards.   

We now compare on $E_n^{\delta_0}(p)$ (see the definition in Lemma~\ref{lem:Etails}) the win probability of the plug-in strategy to that of the \mbox{$p$-oracle} rule. Recall first that the plug-in rule cannot stop before~$h_n:=\lceil \tfrac n2\rceil+1$. For~$t\geq h_n$, we have on $E_n^{\delta_0}(p)$ that
$$
\hat p_t 
\ge 
p-\delta_0
\ge
\frac{1}{M_0+1} - \delta_0
>
\frac{1}{M_0+2}
.
$$
Therefore, for any~$t\in \{h_n,h_n+1,\ldots,n-M_0-1\}$, 
we have
$$
\hat p_t
>
\frac{1}{M_0+2}
\geq
\frac{1}{n-t+1}
$$
on $E_n^{\delta_0}(p)$, so that the plug-in rule cannot stop before~$n-M_0$ on $E_n^{\delta_0}(p)$. 
Now, for~$t\in \{n-M_0,n-M_0+1,\ldots,n-1\}$, we have on~$E_n^{\delta_0}(p)$
\begin{equation}
\label{tot}
\frac{1}{m+2}
<
 p-\delta_0
\leq
\hat p_t
\leq 
 p+\delta_0
<
\frac{1}{m}
,
\end{equation}
so that
$$
\hat p_t
<
\frac{1}{n-t+1}
$$
always holds if~$t\geq n-m+1$, but will also hold at~$t=n-m$ if~$\hat{p}_t<1/(m+1)$, which may be the case under~(\ref{tot}). On~$E_n^{\delta_0}(p)$, we thus have that, depending on the value of~$\hat{p}_t$, the plug-in rule stops on the first success (if any) from~$n-m+1$ onwards or from~$n-m$ onwards, hence coincides with one of the two optimal \mbox{$p$-oracle} rules above. Consequently, 
$
\mathbb P_p[\textrm{plug-in wins}| E_n^{\delta_0}(p)]=V_n(p),
$
and it follows that
\begin{eqnarray*}
W_n(p)
\!\!&\!\! = \!\! &\!\!
\mathbb P_p[\text{plug-in wins}| E_n^{\delta_0}(p)]
\mathbb P_p[E_n^{\delta_0}(p)]
+
\mathbb P_p[\text{plug-in wins}| (E_n^{\delta_0}(p))^c]
\mathbb P_p[(E_n^{\delta_0}(p))^c]
\\[2mm]
\!\!&\!\! \geq \!\! &\!\!
V_n(p) \mathbb P_p[E_n^{\delta_0}(p)]
\\[2mm]
\!\!&\!\! \geq \!\! &\!\!
V_n(p)-\mathbb P_p[(E_n^{\delta_0}(p))^c]
.
\end{eqnarray*}
Therefore, Lemma~\ref{lem:Etails} shows that
$$
V_n(p)-W_n(p)\le C(p_0)e^{-c(p_0)n}
$$
for some positive constants~$C(p_0)$, $c(p_0)$. This establishes~(\ref{eq:p0_pointwise_distance}) for~$p\in\mathcal B_{p_0}$, so that we may restrict in the rest of the proof to the case~$p\notin\mathcal B_{p_0}$ (for which~$\Delta_p>0$).


\paragraph*{Step 2: quantify the loss on $E_n^{\delta_0}(p)$ for a general~$p$}
Fix $p\in[p_0,1)$ and let $m:=m(p)=\lceil \tfrac 1p\rceil-1$. Then, $m\le M_0$.
For $r\in\{0,1,\dots,n-1\}$, let
\begin{equation}
\label{eq:Vr_closed}
V_n^{(r)}(p)=(r+1)p(1-p)^r
\end{equation}
be the win probability of the deterministic threshold rule that stops at the first success (if any) in $\{n-r,\dots,n\}$. In particular, the \mbox{$p$-oracle} rule corresponds to $r=m-1$ (which provides the threshold time~$s_n(p)=n-m+1$), so that
$$
V_n(p)=V_n^{(m-1)}(p)=mp(1-p)^{m-1}
.
$$

Using \eqref{eq:Vr_closed} and the fact that~$x(1-x)^k\le \tfrac 1{k+1}$ for~$k\geq 0$ and~$x\in(0,1)$, we obtain
\begin{equation}
\label{eq:loss_plus_exact}
V_n^{(m-1)}(p)-V_n^{(m)}(p)
=
(m+1)p(1-p)^{m-1}
\!
\Bigl(p-\frac{1}{m+1}\Bigr)
\leq
2 \Bigl(p-\frac{1}{m+1}\Bigr)
\ \
 (m\ge 1)
\end{equation}
and
\begin{equation}
\label{eq:loss_minus_exact}
\hspace{-8mm}
V_n^{(m-1)}(p)-V_n^{(m-2)}(p)
=
mp(1-p)^{m-2}\!
\Bigl(\frac{1}{m}-p\Bigr)
\leq
2 \Bigl(\frac{1}{m}-p\Bigr)
\ \
 (m\ge 2)
.
\end{equation}

For $t\ge h_n$, define the indicators
\begin{equation}
\label{defI}
\mathbb I_t:=\mathbb I\bigg[\hat p_t < \frac{1}{n-t+1}\bigg]
.
\end{equation}
The plug-in rule stops at the first time $t\in \{h_n,\ldots,n-1\}$ at which~$X_t=1$ and $\mathbb I_t=1$ if there is such a~$t$, and otherwise stops at~$n$ if~$X_n=1$.
In particular, on any sample path for which there exists $s\in\{h_n,\dots,n\}$ such that
\[
\mathbb I_t=0 \ \text{for all }h_n\le t<s
\quad
\textrm{ and }
\quad
\mathbb I_t=1 \ \text{for all }t\ge s,
\]
the plug-in rule coincides with the deterministic threshold rule that stops at the first success (if any) from~$s$ onwards.
Using the same argument as in Step~1, the plug-in rule cannot stop before~$n-M_0$ on~$E_n^{\delta_0}(p)$. For~$t\in \{n-M_0,n-M_0+1,\ldots,n-1\}$, we have on~$E_n^{\delta_0}(p)$
\begin{equation}
\label{totgeneralp}
\frac{1}{m+2}
<
 p-\delta_0
\leq
\hat p_t
\leq 
 p+\delta_0
<
\bigg\{
\begin{array}{cl}
\frac{1}{m-1} & \textrm{if } m\geq 2 \\[1mm]
1 & \textrm{if } m=1
\end{array}
\end{equation}
(compare~(\ref{totgeneralp}) with~(\ref{tot}), that holds for boundary values of~$p$ only). Since~$t\mapsto 1/(n-t+1)$ is strictly increasing in~$t$, the same argument as in Step~1 allows us to conclude that, depending on the value of~$\hat{p}_t$, the plug-in rule stops on the first success (if any) (i) from~$n-m$, (ii) from~$s_n(p)=n-m+1$, or (for~$m\geq 2$:) (iii) from~$n-m+2$  onwards. Its deficit in terms of win probability compared to the optimal \mbox{$p$-oracle} rule is therefore~$V_n^{(m-1)}(p)-V_n^{(m)}(p)$ in case~(i), zero in case~(ii), and (for~$m\geq 2$:)~$V_n^{(m-1)}(p)-V_n^{(m-2)}(p)$ in case~(iii). When it is positive, this deficit can be thus controlled by~(\ref{eq:loss_plus_exact})--(\ref{eq:loss_minus_exact}). 
\vspace{2mm}


\paragraph*{Step 3: predictable switch time and conditioning}
Define the random switch time
$$
S
:=
\inf
\!
\big\{
\inf\{t\in T_n
=\{n-M_0,\dots,n-1\}
: \mathbb I_t=1\}
,n
\big\}
,
$$
where~$\mathbb I_t$ was defined in~(\ref{defI}) and $\inf \varnothing=+\infty$. On~$E_n^{\delta_0}(p)$, Step~2 ensures that~$\mathbb I_t=1$ for all~$t\geq S$, that the plug-in rule coincides pathwise with the deterministic threshold rule that stops on the first success (if any) from~$S$ onwards, and that the possible values of~$S$ on~$\{S\leq n-1\}$ are~$n-m$, $n-m+1$, and~(for~$m\geq 2$:) $n-m+2$.

For $t\ge 2$, note the elementary bound
\begin{equation}\label{eq:hatp_increment}
|\hat p_t-\hat p_{t-1}|
=
\bigg|\frac{S_{t-1}+X_t}{t}-\frac{S_{t-1}}{t-1}\bigg|
\le 
\frac{1}{t}
\end{equation}
and, with the events~$H_t$ introduced in Lemma~\ref{lem:Etails},  let
\begin{equation} 
\label{EpredDefn}
\Epred
:=
\bigcap_{t\in T_n}\big(\{S=t\}^c\cup H_t^c\big)
=
\{S=n\}\ \cup\ \bigcup_{t\in T_n}\big(\{S=t\}\cap H_t^c\big)
.
\end{equation}
On \(\Epred\), if \(S=t (\in T_n)\) then the sign of \(\hat p_t-1/(n-t+1)\) cannot change when revealing~\(X_t\) (by \eqref{eq:hatp_increment}), hence
\[
\mathbb I_t=\mathbb I\!\left[\hat p_{t-1}<\frac{1}{n-t+1}\right]
\qquad\text{on }\Epred\cap\{S=t\}
\]
(compare with~(\ref{defI})).
Therefore, 
$\Epred\cap\{S=t\}=\{S=t\}\cap H_t^{c}\in\mathcal F_{t-1}$ for all $t\in T_n$.
Since $S\in T_n\cup\{n\}$ by definition, we also have $\Epred\cap\{S=n\}=\{S=n\}\in\mathcal F_{n-1}$.
On~$\Epred$, the switch time $S$ is \emph{predictable} in the sense that for all $t\in T_n\cup\{n\}$, the event $\{S=t\}$ is $\mathcal F_{t-1}$-measurable relative to~$\Epred$.
Consequently, on $\Epred\cap\{S=t\}$ the variable $X_t$ is independent of $\mathcal F_{t-1}$ and is ${\rm Bernoulli}(p)$ (in other words, $X_S$ is independent of~$\mathcal F_{S-1}$ with $X_S\sim{\rm Bernoulli}(p)$).

Condition then on $\mathcal F_{S-1}$ and split into $X_S=1$ and $X_S=0$. On~$E_n^{\delta_0}(p)\cap \Epred$, if $X_S=1$, then the plug-in rule stops at~$S$ and wins if and only if $X_{S+1}=\cdots=X_n=0$ on~$\{S\leq n-1\}$ (whereas it then always wins on~$\{S= n\}$).
If $X_S=0$, then the rule stops at the first success (if any) in $\{S+1,\dots,n\}$ and wins if and only if there is exactly one success in~$\{S+1,\dots,n\}$ on~$\{S\leq n-1\}$ (whereas the plug-in rule then always loses on~$\{S=n\}$). 
Write $D_n:=\{\text{plug-in wins}\}$.
Since $X_{S+1},\dots,X_n$ are independent of $\mathcal F_S$ (hence of $\mathcal F_{S-1}$) and i.i.d.\
$\mathrm{Bernoulli}(p)$, we obtain that, on $E_n^{\delta_0}(p)\cap \Epred\cap \{S\leq n-1\}$,
\begin{eqnarray*}
\mathbb P_p[D_n| \mathcal F_{S-1}]
\!\!&\!\! = \!\! &\!\!
\mathbb P_p[X_{S+1}=\cdots=X_n=0]
\mathbb P_p[X_S=1| \mathcal F_{S-1}]
\nonumber
\\[1mm]
& &
\hspace{8mm}
+
\mathbb P_p\big[ \textstyle\sum_{t=S+1}^n X_t=1\big]
\mathbb P_p[X_S=0| \mathcal F_{S-1}]
\nonumber
\\[1mm]
\!\!&\!\! = \!\! &\!\!
p(1-p)^{n-S}
+(1-p)(n-S)p(1-p)^{n-S-1}
\nonumber
\\[1mm]
\!\!&\!\! = \!\! &\!\!
(n-S+1)p(1-p)^{n-S}
\nonumber
\\[1mm]
\!\!&\!\! = \!\! &\!\!
V_n^{(n-S)}(p)
,
\end{eqnarray*}
where the last equality uses \eqref{eq:Vr_closed}, whereas on~$E_n^{\delta_0}(p)\cap \Epred\cap \{S= n\}$, we have 
\begin{eqnarray*}
\lefteqn{
\hspace{-10mm}
\mathbb P_p[D_n| \mathcal F_{S-1}]
=
1
\times
\mathbb P_p[X_S=1| \mathcal F_{S-1}]
+
0
\times
\mathbb P_p[X_S=0| \mathcal F_{S-1}]
\nonumber
}
\\[1mm]
& &
\hspace{22mm}
=
p
=
(n-S+1)p(1-p)^{n-S}
=
V_n^{(n-S)}(p)
.
\end{eqnarray*}
Thus, we always have $\mathbb P_p[D_n| \mathcal F_{S-1}]=V_n^{(n-S)}(p)$ on~$E_n^{\delta_0}(p)\cap \Epred$.
\vspace{2mm}


\paragraph*{Step 4: comparing conditional and unconditional win probabilities}
Fix $t\in T_n$ and work on the event $E_n^{\delta_0}(p)\cap\{S=t\}$.
On $E_n^{\delta_0}(p)$, Step~2 implies that the plug-in rule coincides pathwise with the deterministic threshold rule
that stops on the first success (if any) from time~$S$ onwards. In particular, on $\{S=t\}$, it behaves from time~$t$
onwards like the deterministic threshold rule with parameter $r:=n-t$ (see the discussion around~\eqref{eq:Vr_closed}).

Let
$
\lambda_t:=(1-p)^{n-t}
$
and
$\mu_t:=(n-t)p(1-p)^{n-t-1}$.
If one were to reveal $X_t$ at time~$t-1$ and then apply the deterministic threshold rule from time $t$ onwards, then the conditional win probability would be equal to~$\lambda_t$ when $X_t=1$ and $\mu_t$ when~$X_t=0$.
Consequently, on~$\{S=t\}$, we have
$$
\mathbb P_p[D_n| \mathcal F_{S-1}]
=
q_t\lambda_t+(1-q_t)\mu_t,
\qquad
q_t:=\mathbb P_p[X_t=1| \mathcal F_{t-1},S=t],
$$
while the unconditional win probability of the deterministic threshold rule from time $t$ onwards is
\[
V_n^{(n-t)}(p)=p\lambda_t+(1-p)\mu_t.
\]
Subtracting the last two displays yields
\begin{equation}
\label{eq:cond_vs_uncond_threshold}
\mathbb P_p[D_n| \mathcal F_{S-1}]
-
V_n^{(n-t)}(p)
=
(q_t-p)\,(\lambda_t-\mu_t)
\qquad\text{on }\{S=t\}.
\end{equation}
Moreover, using the notation~$r_t:=1/(n-t+1)$ introduced in Lemma~\ref{lem:Etails}, we have
\[
\lambda_t-\mu_t
=
(1-(n-t+1)p)
(1-p)^{n-t-1}
=
(n-t+1)
(r_t-p)
(1-p)^{n-t-1}
.
\]
Using the elementary bound (note that~$t\in T_n$ implies~$n-t\leq M_0$)
\[
(n-t+1)(1-p)^{n-t-1}
\le
n-t+1
\le
M_0+1
\le
\frac{1}{p_0}+1
\le
\frac{2}{p_0},
\]
we thus obtain
\begin{equation}
\label{eq:atbt_bound}
|\lambda_t-\mu_t|
\le
\frac{2}{p_0}\,|p-r_t|.
\end{equation}
Since $|q_t-p|\le 1$, combining \eqref{eq:cond_vs_uncond_threshold} and \eqref{eq:atbt_bound} gives the bound
\begin{equation}
\label{eq:cond_threshold_bias_bound}
|
\mathbb P_p[D_n| \mathcal F_{S-1}]
-
V_n^{(n-t)}(p)
|
\le
\frac{2}{p_0}\,|p-r_t|
\qquad\text{on }\{S=t\}.
\end{equation}
\vspace{-2mm}


\paragraph*{Step 5: conclude}
Since $\mathbb P_p[D_n|\mathcal F_{S-1}]\geq 0$ almost surely, the tower property provides
\begin{equation}
\label{eq:restrict_event}
W_n(p)
=
\mathbb E_p
[\mathbb P_p[D_n|\mathcal F_{S-1}]
]
\ge
\mathbb E_p[
\mathbb I[E_n^{\delta_0}(p)] \mathbb P_p[D_n|\mathcal F_{S-1}]
]
.
\end{equation}
Using~$V_n(p)\leq 1$ and Lemma~\ref{lem:Etails}, this yields
\begin{eqnarray}
V_n(p)-W_n(p)
\!\!&\!\! \leq \!\! &\!\!
V_n(p)
\mathbb P_p[(E_n^{\delta_0}(p))^c]
+
\mathbb E_p
[
\mathbb I[E_n^{\delta_0}(p)]
(V_n(p)-\mathbb P_p[D_n|\mathcal F_{S-1}])
]
\nonumber
\\[2mm]
\!\!&\!\! \leq \!\! &\!\!
C(p_0)e^{-c(p_0)n}
+
F_n(p)
+
G_n(p)
,
\label{conclude1}
\end{eqnarray}
for some positive constants~$C(p_0),c(p_0)$, where we let
$$
F_n(p)
:=
\mathbb E_p
[
\mathbb I[E_n^{\delta_0}(p)\cap \Epred]
(V_n(p)-\mathbb P_p[D_n|\mathcal F_{S-1}])
]
,
$$
and
$$
G_n(p)
:=
\mathbb E_p
[
\mathbb I[E_n^{\delta_0}(p)\cap(\Epred)^c]
(V_n(p)-\mathbb P_p[D_n|\mathcal F_{S-1}])
]
.
$$
To conclude the proof of~\eqref{eq:p0_pointwise_distance}, we therefore need to upper-bound~$F_n(p)$ and~$G_n(p)$.
 \vspace{2mm}


\paragraph*{Upper-bound on $F_n(p)$}
From Step~3, 
$\mathbb P_p[D_n|\mathcal F_{S-1}]
=
V_n^{(n-S)}(p)$ on~$E_n^{\delta_0}(p)\cap \Epred$, so that we have
$$
F_n(p)
=
\mathbb E_p
[
\mathbb I[E_n^{\delta_0}(p)\cap \Epred]
(V_n(p)-V_n^{(n-S)}(p))
]
.
$$
Now, on~$E_n^{\delta_0}(p)$, Step~2 implies that~$n-S\in\{m-2,m-1,m\}$ (with the convention that \mbox{$m-2$} is absent when $m=1$), so that the resulting deficit with respect to the oracle win probability~$V_n(p)=V_n^{(m-1)}(p)$ is 
\begin{eqnarray*}
\lefteqn{
\hspace{-0mm}
V_n(p)-V_n^{(n-S)}(p)
=
(V_n^{(m-1)}(p)-V_n^{(m)}(p))
\mathbb I[S=n-m]
}
\\[2mm]
& & 
\hspace{33mm}
+
(V_n^{(m-1)}(p)-V_n^{(m-2)}(p))
\mathbb I[S=n-m+2]
\mathbb I[m\ge 2]
.
\end{eqnarray*}
Therefore, using~(\ref{eq:loss_plus_exact})--(\ref{eq:loss_minus_exact}), we obtain 
\begin{eqnarray*}
F_n(p)
\!\!&\!\! \leq \!\! &\!\!
2\Big(p-\frac{1}{m+1}\Big)
\mathbb E_p[\mathbb I[E_n^{\delta_0}(p)\cap \Epred] \mathbb I[S=n-m]] 
\\[2mm]
& & 
\hspace{10mm}
+
2\mathbb I[m\geq 2]
\Bigl(\frac{1}{m}-p\Bigr)
\mathbb E_p[\mathbb I[E_n^{\delta_0}(p)\cap \Epred] \mathbb I[S=n-m+2]] 
.
\end{eqnarray*}

On $E_n^{\delta_0}(p)$, the event $\{S=n-m\}$ can only happen if at time $t=n-m$ we already have~$\hat p_t<1/(m+1)$, and the event $\{S=n-m+2\}$ (when $m\ge 2$) can only happen if at time $t=n-m+1$ we still have
$\hat p_t\ge 1/m$. Therefore, by Hoeffding's inequality,
$$
\mathbb P_p[S=n-m]
\le
\mathbb P_p\Big[\hat p_{n-m}-p< -\Big(p-\frac{1}{m+1}\Big)\Big]
\le
\exp\!\big(\!\!-2(n-m)d_{p+}^2\big)
$$
and, for $m\ge 2$,
$$
\mathbb P_p[S=n-m+2]
\le
\mathbb P_p\Big[\hat p_{n-m+1}-p\ge \frac{1}{m}-p\Big]
\le
\exp\!\big(\!\!-2(n-m+1)d_{p-}^2\big)
,
$$
where we let
\[
d_{p+}:=p-\frac{1}{m+1}> 0
\quad
\textrm{ and }
\ \
(\textrm{for }m\ge 2\textrm{:})
\ \
d_{p-}:=\frac{1}{m}-p>0
.
\]
Therefore,
\begin{equation}
F_n(p)
\leq 
2d_{p+}
\exp\!\big(\!\!-2(n-m)d_{p+}^2\big)
\nonumber
+
2 \mathbb I[m\geq 2]
d_{p-}
\exp\!\big(\!\!-2(n-m+1)d_{p-}^2\big)
.
\label{tsppp2}
\end{equation}

If $m=1$, then the $d_{p-}$-term is absent and the $d_{p+}$-term is of the expected form since~$d_{p+}=\Delta_p$ and~$n-m\geq n-M_0\geq n/2$. 
If $m\ge 2$, then $\Delta_p=\min\{d_{p+},d_{p-}\}$.
If $\Delta_p=d_{p+}$, then, with the constant 
$
\delta_0
=
\delta_0(p_0)
=
\frac13\min\{|q-q'|:\, q,q'\in\mathcal B_{p_0},\, q\neq q'\}$ from~(\ref{deltap0Defn}), 
  we have
$$
\delta_0
\leq
\frac{1}{2}
\Big(\frac{1}{m}-\frac{1}{m+1}\Big)
\leq 
d_{p-}
\leq 
1
,
$$
so that the $d_{p-}$-term is bounded by $C(p_0)e^{-c(p_0)n}$ and can be absorbed into the exponential-in-$n$ remainder term. Similarly, if $\Delta_p=d_{p-}$, then $\delta_0\leq d_{p+}\leq 1$, so the $d_{p+}$-term is absorbed into $C(p_0)e^{-c(p_0)n}$. In all cases, we thus have
\begin{equation}
\label{conclude2}
F_n(p)
\le
C_1(p_0)
\Delta_p
e^{-c(p_0)n\Delta_p^2}
+
C_2(p_0)
e^{-c(p_0)n}
\end{equation}
after renaming constants.
\vspace{2mm}


\paragraph*{Upper-bound on $G_n(p)$}
From~(\ref{EpredDefn}), we obtain
\[
(\Epred)^c
=
\bigcup_{t\in T_n}\big(\{S=t\}\cap H_t\big)
.
\]
Let
$
\Delta
:=
|V_n(p)-\mathbb P_p[D_n|\mathcal F_{S-1}]|$. Note that $0\le \Delta\le 1$ almost surely.
Then,
\begin{eqnarray}
G_n(p)
\!\!&\!\! = \!\! &\!\!
\mathbb E_p[
\mathbb I[E_n^{\delta_0}(p)\cap(\Epred)^c]
(V_n(p)-\mathbb P_p[D_n|\mathcal F_{S-1}])
]
\nonumber
\\[2mm]
\!\!&\!\! \leq \!\! &\!\!
\mathbb E_p[
\mathbb I[E_n^{\delta_0}(p)\cap(\Epred)^c]
\Delta
]
\nonumber
\\[1mm]
\!\!&\!\! \leq \!\! &\!\!
\sum_{t\in T_n}
\mathbb E_p[
\mathbb I[E_n^{\delta_0}(p)\cap\{S=t\}\cap H_t]
\Delta
].
\label{eq:Gn_union_bound_EpredS}
\end{eqnarray}

Note that $(\Epred)^c\subseteq\{S\in T_n\}$, so that the case $\{S=n\}$ does not contribute to $G_n(p)$. On~$E_n^{\delta_0}(p)\cap\{S=t\}$ with $t\in T_n$, Step~2 implies that
$t\in \{n-m,\ n-m+1,\ n-m+2\}\cap T_n$ (with $n-m+2$ only relevant when $m\ge2$). In particular, $n-t\in\{m,\ m-1,\ m-2\}$, and by \eqref{eq:loss_plus_exact}--\eqref{eq:loss_minus_exact} we have the deterministic bound
\begin{equation}
\label{eq:oracle_vs_threshold_bound_rt}
0\le V_n(p)-V_n^{(n-t)}(p)\le 2|p-r_t|
\qquad\text{on }E_n^{\delta_0}(p)\cap\{S=t\}
\end{equation}
(when $n-t=m-2$, we used that $|p-\tfrac 1m|\le |p-\tfrac 1{m-1}|=|p-r_t|$ since $p\le \tfrac 1m$).

Putting \eqref{eq:cond_threshold_bias_bound} and \eqref{eq:oracle_vs_threshold_bound_rt} together, we obtain on
$E_n^{\delta_0}(p)\cap\{S=t\}$,
\[
0\le
\Delta
\le
\bigl(V_n(p)-V_n^{(n-t)}(p)\bigr)
+
\big|V_n^{(n-t)}(p)-\mathbb P_p[D_n| \mathcal F_{S-1}]\big|
\le
\Bigl(2+\frac{2}{p_0}\Bigr)|p-r_t|.
\]
Inserting this into \eqref{eq:Gn_union_bound_EpredS} yields
$$
G_n(p)
\le
\Bigl(2+\frac{2}{p_0}\Bigr)
\sum_{t\in T_n}
|p-r_t|
\mathbb P_p[E_n^{\delta_0}(p)\cap\{S=t\}\cap H_t]
.
$$
Since Lemma~\ref{lem:Etails} entails that, for every~$t\in T_n$,
\[
\mathbb P_p[E_n^{\delta_0}(p)\cap\{S=t\}\cap H_t]
\leq
\mathbb P_p[H_t]
\le 
C(p_0)
e^{-c(p_0)n(p-r_t)^2}
\]
for some positive constants~$C(p_0),c(p_0)$, we obtain
\begin{equation}
G_n(p)
\leq 
C(p_0) 
\Bigl(2+\frac{2}{p_0}\Bigr)
\sum_{t\in T_n}
|p-r_t|
e^{-c(p_0)n(p-r_t)^2}
.
\label{eq:Bn_weighted_union}
\end{equation}

Working again with the quantity~$
\delta_0$ from~(\ref{deltap0Defn}), we distinguish two cases.

\smallskip
\noindent\emph{Case (a): $\Delta_p\ge \delta_0$.}
Since $r_t\in\mathcal B_{p_0}$, we then have $|p-r_t|\ge \Delta_p\ge \delta_0$ for all $t\in T_n$. Using~$|T_n|\leq M_0$ and~$|p-r_t|\leq 1$, we obtain from~(\ref{eq:Bn_weighted_union}) that
$$
G_n(p)
\le
C(p_0)\,e^{-c(p_0)n}
,
$$
after renaming constants.

\smallskip
\noindent\emph{Case (b): $\Delta_p< \delta_0$.}
Then, the closest boundary point is unique: let $r^\star\in\mathcal B_{p_0}$ be such that $|p-r^\star|=\Delta_p$.
Since the map~$t\mapsto r_t=1/(n-t+1)$ is one-to-one from~$T_n$ to~$\mathcal B_{p_0}$, there is a unique $t^\star\in T_n$ such that $r_{t^\star}=r^\star$.
For $t=t^\star$, we have $|p-r_t|=\Delta_p$, and~(\ref{eq:Bn_weighted_union}) writes
$$
G_n(p)
\leq 
C(p_0) 
\Bigl(2+\frac{2}{p_0}\Bigr)
\Delta_p
e^{-c(p_0)n\Delta_p^2}
+
C(p_0) 
\Bigl(2+\frac{2}{p_0}\Bigr)
\sum_{t\in T_n\setminus\{t^\star\}}
|p-r_t|
e^{-c(p_0)n(p-r_t)^2}
.
$$
By definition of~$\delta_0$, we have~$\delta_0\leq |p-r_t|\leq 1$ for all $t\in T_n\setminus\{t^\star\}$, hence the same argument as in case~(a) yields
$
G_n(p)
\leq 
C_1(p_0) 
\Delta_p
e^{-c(p_0)n\Delta_p^2}
+
C_2(p_0) 
 e^{-c(p_0)n}
$
after renaming constants.

Combining the two cases, we have shown that there exist constants~$C_1(p_0)$, $C_2(p_0)$ and~$c(p_0)$ such that
\begin{equation}
\label{conclude3}
G_n(p)
\leq 
C_1(p_0) 
\Delta_p
e^{-c(p_0)n\Delta_p^2}
+
C_2(p_0) 
 e^{-c(p_0)n}
\end{equation}
for all~$n\geq 2(M_0+3)$ and all~$p\in[p_0,1)\setminus \mathcal B_{p_0}$. Combining~(\ref{conclude1}), (\ref{conclude2}) and~(\ref{conclude3}) establishes the result in~(\ref{eq:p0_pointwise_distance}) for all~$n\geq 2(M_0+3)$ and all~$p\in[p_0,1)\setminus \mathcal B_{p_0}$. Since Step~1 already showed the result for~$p\in \mathcal B_{p_0}$ and since the result extends to smaller values of~$n$ by absorbing constants, this concludes the proof of~(\ref{eq:p0_pointwise_distance}).


(ii) For~$p_0>\frac{1}{2}$, we have~$\inf_{p\in[p_0,1)} \Delta_p>0$, so that~(\ref{eq:p0_uniform_exp_term}) directly follows from~(\ref{eq:p0_pointwise_distance}). For~$p_0\leq \frac{1}{2}$, taking the supremum over $p\in[p_0,1)$ in~(\ref{eq:p0_pointwise_distance}) and using that for every $a>0$,
$$
\sup_{\Delta\ge 0}\ \Delta e^{-a n \Delta^2}
=
\frac{1}{\sqrt{2a en}},
$$
we obtain that
$$
\sup_{p\in[p_0,1)}\bigl(V_n(p)-W_n(p)\bigr)
\le
\frac{C_1(p_0)}{\sqrt{2c(p_0)en}}+C_2(p_0)e^{-c(p_0)n}
.
$$
Since  the exponential term can be absorbed into the $1/\sqrt{n}$ one by enlarging~$C_1(p_0)$, this proves~\eqref{eq:p0_uniform_two_term}.
\end{proof}


\subsection{Proof of Theorem~\ref{thmLowerBound}}
\label{secProofsLowerBound}

The proof of Theorem~\ref{thmLowerBound} requires the following preliminary result.

\begin{lem}
\label{LemmaLowerBound}
Let $\pi$ be a (possibly randomized) \mbox{$p$-blind} rule and denote the corresponding stopping time  as $\tau_n^\pi=\tau_n^\pi(X_1,\ldots,X_n,U)$, where the auxiliary random variable~$U$ is realizing the possible internal randomization. Then,
(i) for any $p\in(\frac{1}{3},\frac{1}{2})$,
$$
V_n(p)-W_n^\pi(p)\ge (1-2p)\mathbb P_p[X_{n-1}=1,\tau_n^\pi\geq n];
$$
(ii) for any $p\in(\frac{1}{2},1)$,
$$
V_n(p)-W_n^\pi(p)\ge (2p-1)\mathbb P_p[X_{n-1}=1,\tau_n^\pi=n-1]
.
$$
\end{lem}

\begin{proof}
(i) Fix $p\in(\tfrac13,\tfrac12)$ and let
$
B:=\{X_{n-1}=1,\tau_n^\pi\geq n\}$. Since $\{\tau_n^\pi\geq n\}=\{\tau_n^\pi\le n-1\}^c$, we have $B\in\mathcal F_{n-1}$, where~$\mathcal F_{t}:=\sigma(X_1,\ldots,X_{t},U)$. Consider then the \mbox{$p$-blind} rule $\tilde\pi$ associated with the stopping time
\[
\tilde\tau_n:=
\bigg\{
\begin{array}{ll}
n-1 & \text{if } B \text{ occurs} \\[1mm]
\tau_n^\pi & \text{otherwise.}
\end{array}
\]
Note that~$\tilde\tau_n$ is indeed a stopping time since 
$\{\tilde\tau_n\le t\}=\{\tau_n^\pi\le t\}\in\mathcal F_t$ for $t\le n-2$, and $\{\tilde\tau_n\le n-1\}=\{\tau_n^\pi\le n-1\}\cup B\in\mathcal F_{n-1}$.

Now, let~$\mathcal W^{\tilde\pi}:=\{X_{\tilde\tau_n}=1,\ X_{\tilde\tau_n+1}=\cdots=X_n=0\}$ be the win event of~$\tilde{\pi}$. On~$B$, we have $\mathcal W^\pi\subseteq \{X_n=1\}$ and~$\mathcal W^{\tilde\pi}=\{X_n=0\}$. Since $X_n$ is independent of $\mathcal F_{n-1}$, this yields
\begin{equation}
\label{tutuu1}
\mathbb P_p[\mathcal W^{\tilde\pi}|\mathcal F_{n-1}]
-\mathbb P_p[\mathcal W^\pi|\mathcal F_{n-1}]
\geq  (1-p)-p
= 1-2p
(\geq 0)
\quad\text{on }B.
\end{equation}
On $B^c$, we have~$\tilde\tau_n=\tau_n^\pi$, hence $\mathcal W^{\tilde\pi}=\mathcal W^\pi$, so that
\begin{equation}
\label{tutuu2}
\mathbb P_p[\mathcal W^{\tilde\pi}|\mathcal F_{n-1}]
-\mathbb P_p[\mathcal W^\pi|\mathcal F_{n-1}]
=0
\quad\text{on }B^c.
\end{equation}
From~(\ref{tutuu1})--(\ref{tutuu2}), we have
\[
W_n^{\tilde\pi}(p)-W_n^\pi(p)
=
\mathbb E_p[
(
\mathbb P_p[\mathcal W^{\tilde\pi}|\mathcal F_{n-1}]-\mathbb P_p[\mathcal W^\pi|\mathcal F_{n-1}]
)
\mathbb I[B]
]
\geq
(1-2p)\mathbb P_p[B].
\]
Since the optimality of the \mbox{$p$-oracle} rule implies that~$V_n(p)\ge W_n^{\tilde\pi}(p)$, we conclude that
$$
V_n(p)-W_n^\pi(p)\ge W_n^{\tilde\pi}(p)-W_n^\pi(p)
\geq
(1-2p)\mathbb P_p[B],
$$
which proves~(i).
\vspace{2mm}

(ii) Fix $p\in(\frac{1}{2},1)$ and denote the win event of~$\pi$ by
$
\mathcal W^\pi:=\{X_{\tau_n^\pi}=1,X_{{\tau_n^\pi}+1}=\cdots=X_n=0\}$.
Since $p>\frac{1}{2}$, the oracle win probability is~$V_n(p)=p$, so that
\begin{equation}
\label{LemMin1}
V_n(p)-W_n^\pi(p)
=
p-\mathbb P_p[\mathcal W^\pi]
=
\mathbb E_p[p-\mathbb P_p[\mathcal W^\pi|\mathcal F_{n-1}]].
\end{equation}
We claim that
\begin{equation}
\label{LemMin2}
\mathbb P_p[\mathcal W^\pi|\mathcal F_{n-1}]\le p
\quad
\mathbb P_p\textrm{-almost surely.}
\end{equation}
Indeed, conditional on $\mathcal F_{n-1}$, there are two cases.
If $\tau_n^\pi\le n-1$, then $\mathcal W^\pi\subseteq\{X_n=0\}$, hence on~$\{\tau_n^\pi\le n-1\}$, we have $\mathbb P_p[\mathcal W^\pi|\mathcal F_{n-1}]\le \mathbb P_p[X_n=0]=1-p\le p$. 
If $\tau_n^\pi\geq n$, then $\mathcal W^\pi\subseteq\{X_n=1\}$, so on~$\{\tau_n^\pi\geq n\}$,
we have $\mathbb P_p[\mathcal W^\pi|\mathcal F_{n-1}]\le \mathbb P_p[X_n=1]=p$. This shows~(\ref{LemMin2}).

Let then
$
A:=\{X_{n-1}=1,\tau_n^\pi=n-1\}\in\mathcal F_{n-1}.
$
On $A$, we have that~$\mathcal W^\pi$ occurs if and only if $X_n=0$. Thus,
$\mathbb P_p[\mathcal W^\pi|\mathcal F_{n-1}]=1-p$ on $A$. Using (\ref{LemMin1})--(\ref{LemMin2}), it follows that
\begin{eqnarray*}
V_n(p)-W_n^\pi(p)
&\ge&
\mathbb E_p[(p-\mathbb P_p[\mathcal W^\pi|\mathcal F_{n-1}])\mathbb I[A]]
\\[1mm]
&=&
(p-(1-p))\mathbb P_p[A]
\\[1mm]
&=&
(2p-1)\mathbb P_p[X_{n-1}=1,\tau_n^\pi=n-1]
,
\end{eqnarray*}
which proves (ii).
\end{proof}


\begin{proof}[Proof of Theorem~\ref{thmLowerBound}]
Fix an arbitrary (possibly randomized) \mbox{$p$-blind} rule~$\pi$.
For a fixed~$h>0$, let
$$
p_n^-:=\frac12-\frac{h}{\sqrt n}
\qquad
\textrm{and}
\qquad
p_n^+:=\frac12+\frac{h}{\sqrt n}
.
$$
For all $n$ large enough, we have $p_n^-\in(\max(p_0,\tfrac13),\tfrac12)$ and $p_n^+\in(\frac{1}{2},1)$, so that
Lemma~\ref{LemmaLowerBound}(i)--(ii) apply at~$p_n^-$ and $p_n^+$, respectively.

Let $\mathbb Q_n^\pm$ denote the joint law of $(X_1,\dots,X_{n-1},U)$ under $\mathbb P_{p_n^\pm}$, where~$U$ is the random variable that realizes the possible randomization of~$\pi$. Since $U$ is independent of the $X_t$'s and its distribution does not depend on~$p$, we have
$$
\mathrm{TV}(\mathbb Q_n^+,\mathbb Q_n^-)
=
\mathrm{TV}
(\mathbb P_{n-1}^+,\mathbb P_{n-1}^-
)
,
$$
where~$\mathrm{TV}$ denotes the total variation distance and~$\mathbb P_{n-1}^\pm$ stand for the joint law of $(X_1,\ldots,X_{n-1})$ under $\mathbb P_{p_n^\pm}$. Moreover, denoting as~$\mathrm{KL}(P\|Q)$
the Kullback--Leibler divergence between the probability measures $P,Q$ with $P\ll Q$,
we have
$$
\mathrm{KL}(\mathbb P_{n-1}^+\|\mathbb P_{n-1}^-)
=
(n-1)\mathrm{kl}(p_n^+\|p_n^-)
,
$$
 where
$$
\mathrm{kl}(u\|v)
:=
u\log\Big(\frac{u}{v}\Big)+(1-u)\log\Big(\frac{1-u}{1-v}\Big)
,
\qquad
u,v\in(0,1)
,
$$
is the KL divergence between the ${\rm Bernoulli}(u)$ and ${\rm Bernoulli}(v)$ distributions.

Now, there exists an absolute constant~$C>0$ such that, for any $u,v\in[\tfrac{1}{4},\tfrac{3}{4}]$, 
$$
\mathrm{kl}(u\|v)
\le 
C(u-v)^2.
$$
Indeed, for fixed $u$, the map $g_u(v):=\mathrm{kl}(u\|v)$ satisfies $g_u(u)=0$, $g_u'(u)=0$, and
$$
0
\leq
g_u''(v)=\frac{u}{v^2}+\frac{1-u}{(1-v)^2}
\leq
C
,
\qquad v\in [\tfrac{1}{4},\tfrac{3}{4}]
,
$$
for some absolute constant~$C$ (in the rest of the proof, the constant~$C$ may change from line to line). By Taylor's theorem with remainder, we thus have
$$
\mathrm{kl}(u\|v)=g_u(v)\le \frac12
(u-v)^2
\sup_{t\in\big[\tfrac{1}{4},\tfrac{3}{4}\big]} 
|g_u''(t)|
\le 
C
(u-v)^2.
$$

For $n$ large enough, we have $p_n^\pm\in[\tfrac14,\tfrac34]$, hence
$$
\mathrm{KL}(\mathbb P_{n-1}^+\|\mathbb P_{n-1}^-)
\le 
C(n-1)\frac{(2h)^2}{n}
\le 
C
h^2.
$$
By Pinsker's inequality, we conclude that
there exists an absolute constant $C_{\rm TV}>0$ such that
$$
\mathrm{TV}(\mathbb Q_n^+,\mathbb Q_n^-)
=
\mathrm{TV}
(\mathbb P_{n-1}^+,\mathbb P_{n-1}^-
)
\le 
\sqrt{\frac12\,\mathrm{KL}(\mathbb P_{n-1}^+\|\mathbb P_{n-1}^-)}
\le C_{\rm TV} h
$$
for all $n$ large enough.

Now, define the $\mathcal F_{n-1}$-measurable events
\[
A:=\{X_{n-1}=1,  \tau_n^\pi=n-1\}
\quad
\textrm{ and }
\quad
D:=\{\tau_n^\pi\le n-2\}.
\]
Since $A,D\in\sigma(X_1,\ldots,X_{n-1},U)$, we have
$\mathbb P_{p_n^\pm}[A]=\mathbb Q_n^\pm[A]$ and $\mathbb P_{p_n^\pm}[D]=\mathbb Q_n^\pm[D]$.
Hence, for all $n$ large enough,
\[
|\mathbb P_{p_n^+}[A]-\mathbb P_{p_n^-}[A]|
=
|\mathbb Q_n^+[A]-\mathbb Q_n^-[A]|
\le \mathrm{TV}(\mathbb Q_n^+,\mathbb Q_n^-)
\le C_{\rm TV}h,
\]
and similarly
\[
|\mathbb P_{p_n^+}[D]-\mathbb P_{p_n^-}[D]|
=
|\mathbb Q_n^+[D]-\mathbb Q_n^-[D]|
\le C_{\rm TV}h.
\]
We then treat two cases.

\medskip
\emph{Case (a):} $\mathbb P_{p_n^-}[D]\ge \tfrac14$.
Since $p_n^+>\tfrac12$, the $p_n^+$-oracle win probability is $V_n(p_n^+)=p_n^+$.
On~$D$, a necessary condition for~$\pi$ to win is that~$X_n=0$, hence
$\mathbb P_{p_n^+}[\mathcal W^\pi|\mathcal F_{n-1}]\le 1-p_n^+$ on~$D$. Since~$\mathbb P_{p_n^+}[\mathcal W^\pi|\mathcal F_{n-1}]\leq {p_n^+}$ $\mathbb P_{p_n^+}$-almost surely (this was proved in~(\ref{LemMin2})), this implies that
\begin{eqnarray*}
V_n(p_n^+)-W_n^\pi(p_n^+)
\!\!&\!\! = \!\! &\!\!
 \mathbb E_{p_n^+}[p_n^+-\mathbb P_{p_n^+}[\mathcal W^\pi|\mathcal F_{n-1}]] 
 \\[2mm]
\!\!&\!\! \geq \!\! &\!\!
 \mathbb E_{p_n^+}[ (p_n^+-\mathbb P_{p_n^+}[\mathcal W^\pi|\mathcal F_{n-1}] ) \mathbb I[D]] 
 \\[2mm]
\!\!&\!\! \geq \!\! &\!\!
(2p_n^+-1)\mathbb P_{p_n^+}[D]
.
\end{eqnarray*}
Since $\mathbb P_{p_n^+}[D]\ge \mathbb P_{p_n^-}[D]-C_{\rm TV}h \ge \tfrac14-C_{\rm TV}h$, this yields
\begin{equation}
\label{rrr1}
V_n(p_n^+)-W_n^\pi(p_n^+)
\ge \frac{2h}{\sqrt n}\Bigl(\frac14-C_{\rm TV}h\Bigr).
\end{equation}

\emph{Case (b):} $\mathbb P_{p_n^-}[D]< \tfrac14$.
Since $D=\{\tau_n^\pi\le n-2\}\in\sigma(X_1,\ldots,X_{n-2},U)$, the event $D$ is independent of $X_{n-1}$ under
$\mathbb P_{p_n^-}$. Hence,
\[
\mathbb P_{p_n^-}[X_{n-1}=1,\, D^c]
=\mathbb P_{p_n^-}[X_{n-1}=1]\mathbb P_{p_n^-}[D^c]
=p_n^-(1-\mathbb P_{p_n^-}[D]).
\]
Therefore,
$$
p_n^-(1-\mathbb P_{p_n^-}[D])
=
\mathbb P_{p_n^-}[X_{n-1}=1, \tau^\pi_n> n-2]
=
\mathbb P_{p_n^-}[A]+\mathbb P_{p_n^-}[X_{n-1}=1, \tau^\pi_n\geq n]
.
$$
Using $\mathbb P_{p_n^-}[A]\le \mathbb P_{p_n^+}[A]+C_{\rm TV}h$, we obtain
$$
p_n^-(1-\mathbb P_{p_n^-}[D])
\leq
\mathbb P_{p_n^-}[X_{n-1}=1, \tau^\pi_n\geq n]
+
\mathbb P_{p_n^+}[A]
+
C_{\rm TV}h
, 
$$
which, since $\mathbb P_{p_n^-}[D]< \tfrac14$, yields
$$
\mathbb P_{p_n^-}[X_{n-1}=1, \tau_n^\pi\geq n]
+
\mathbb P_{p_n^+}[A]
\ge 
\frac34 p_n^- - C_{\rm TV}h
.
$$
For all~$n$ large enough, we have $p_n^-\in(\tfrac13,\tfrac12)$ and $p_n^+\in(\frac{1}{2},1)$. Applying Lemma~\ref{LemmaLowerBound}(i)--(ii) (at~$p_n^-$ and~$p_n^+$, respectively) then provides
\begin{eqnarray}
\lefteqn{
\hspace{-8mm}
\max\{
V_n(p_n^-)-W_n^\pi(p_n^-)
, 
V_n(p_n^+)-W_n^\pi(p_n^+)
\} 
}
\nonumber
\\[2mm]
& &
\hspace{3mm}
\ge \frac12
\bigl(
(1-2p_n^-)\mathbb P_{p_n^-}[X_{n-1}=1,\tau_n^\pi\geq n]
+
(2p_n^+-1)\mathbb P_{p_n^+}[A]
\bigr) 
\nonumber
\\[2mm]
& &
\hspace{3mm}
= 
\frac{h}{\sqrt n}
\bigl(
\mathbb P_{p_n^-}[X_{n-1}=1,\tau_n^\pi\geq n]
+
\mathbb P_{p_n^+}[A]
\bigr) 
\nonumber
\\[2mm]
& &
\hspace{3mm}
\ge \frac{h}{\sqrt n}\Bigl( \frac34 p_n^- - C_{\rm TV}h\Bigr).
\label{rrr2}
\end{eqnarray}

We can now conclude on the basis of~(\ref{rrr1})--(\ref{rrr2}). Choose $h>0$ small enough to have $\tfrac14-C_{\rm TV}h> \tfrac18$. Since $p_n^-\to\tfrac12$,  we have
$\tfrac34 p_n^- - C_{\rm TV}h\ge \tfrac14$ for all $n$ large enough.
Hence, in both cases~(a)--(b), we have
\[
\sup_{p\in[p_0,1)}\bigl(V_n(p)-W_n^\pi(p)\bigr)\ge \frac{h}{4\sqrt n}
\]
for all $n$ large enough, which establishes the result.
\end{proof}


\section{Proofs for Section~3}
\label{appProofsConstants}

This appendix proves Theorems~\ref{thm:plugin_sharp}--\ref{thm:minimax_sharp}. We first record three elementary facts, used here and in Appendix~\ref{secProofsSampleSplit}.

\begin{lem}
\label{lem:gfacts}
Fix~$p\in(0,1)$ and let~$m=m(p)$; see~(\ref{eq:sn_homo}). Then,~$g(m)=\max_{\ell\ge1}g(\ell)$, and the identities~(\ref{eq:g_gap_minus}) and~(\ref{eq:g_gap_plus}) hold, the latter for~$m\ge2$; moreover, the right-hand sides of these identities are nonnegative.
\end{lem}

\begin{proof}
For~$\ell\ge1$, $g(\ell+1)/g(\ell)=\frac{\ell+1}{\ell}(1-p)\ge1$ if and only if~$\ell\le\frac1p-1$, so~$g$ is nondecreasing up to~$\lfloor 1/p\rfloor$ and nonincreasing afterwards; since~$m=\lceil 1/p\rceil-1$ equals~$\lfloor 1/p\rfloor$ when~$1/p\notin\mathbb N$ and equals~$\lfloor 1/p\rfloor-1$ otherwise (in which case~$g(m)=g(m+1)=g(\lfloor 1/p\rfloor)$), we get~$g(m)=\max_{\ell\ge1}g(\ell)$. Next, direct computations provide
$$
g(m)-g(m+1)
=p(1-p)^{m-1}\bigl\{(m+1)p-1\bigr\}
=K_-(p)\Delta_-
,
$$
$$
g(m)-g(m-1)
=p(1-p)^{m-2}\bigl\{1-mp\bigr\}
=K_+(p)\Delta_+
,
$$
and nonnegativity follows since~$m=\lceil 1/p\rceil-1$ is equivalent to~$\frac{1}{m+1}\le p<\frac1m$.
\end{proof}

\begin{lem}
\label{lem:cellsep}
Let~$p\in[p_0,1)$ and~$m=m(p)$, and set~$\Delta_-:=p-\frac{1}{m+1}$ and~$\Delta_+:=\frac1m-p$. Then $m\le M_0$,
$$
\Delta_-+\Delta_+=\frac{1}{m(m+1)}
,
\quad
\textrm{ and}
\quad
\max\{\Delta_-,\Delta_+\}\ \ge\ \frac{1}{2M_0(M_0+1)}=:c_1(p_0)>0
.
$$
Moreover, for any~$q\in(0,1)$ and any integer~$i\ge2$, $m(q)=m+i$ forces $|q-p|\ge c_0(p_0):=1/\{(M_0+1)(M_0+2)\}$, and so does~$m(q)=m-i$.
\end{lem}

\begin{proof}
Since~$p\ge p_0$, we have~$m=\lceil 1/p\rceil-1\le\lceil 1/p_0\rceil-1=M_0$. The displayed identity on~$\Delta_-+\Delta_+$ is immediate, and the lower bound on the maximum follows since~$\Delta_-,\Delta_+\geq 0$ sum to~$1/\{m(m+1)\} \ge 1/\{M_0(M_0+1)\}$. Finally, $m(q)\ge m+2$ forces~$q<\frac{1}{m+2}$, hence
$
|q-p|\ge \frac{1}{m+1}-\frac{1}{m+2}=\frac{1}{(m+1)(m+2)}\ge c_0(p_0)
$,
and similarly~$m(q)\le m-2$ forces~$q\ge\frac{1}{m-1}$, hence~$|q-p|\ge\frac{1}{m-1}-\frac1m\ge c_0(p_0)$.
\end{proof}

\begin{lem}
\label{lem:gammamax}
For~$\gamma_k$ defined in~(\ref{eq:gammak}), one has~$\gamma_k\le\frac12$ for every integer~$k\ge2$, with equality if and only if~$k=2$. Moreover,~$\sup_{u>0}u\Phi(-u)$ is attained at the unique positive root~$u_\star=0.7517\ldots$ of~$\Phi(-u)=u\varphi(u)$, so that~$C_\star=\frac12\sup_{u>0}u\Phi(-u)=0.08498\ldots$
\end{lem}

\begin{proof}
The first claim follows from the fact that~$(1-\frac1k)^{k-2}\le1$ for any~$k\geq 2$ (with equality if and only if~$k=2$) and the strict increase of~$x\mapsto x(1-x)$ increases on~$(0,\frac12]$. For the second claim, $h(u):=u\Phi(-u)$ satisfies~$h(0)=0$, $h>0$ on~$(0,\infty)$ and~$h(u)\to0$ as~$u\to\infty$, so its supremum is attained at a critical point, where~$h'(u)=\Phi(-u)-u\varphi(u)=0$. Uniqueness follows by considering~$r(u):=\Phi(-u)/(u\varphi(u))$, so that~$h'(u)=0$ if and only if~$r(u)=1$. Using~$(u\varphi(u))'=\varphi(u)(1-u^2)$,
$$
r'(u)=\frac{-u\varphi(u)-\Phi(-u)(1-u^2)}{u^2\varphi(u)}
.
$$
For~$u\in(0,1]$, the numerator is a sum of nonpositive terms, one strictly negative. For~$u>1$, the Mills ratio bound~$\Phi(-u)<\varphi(u)/u$ gives
$$
-u\varphi(u)+\Phi(-u)(u^2-1)<-u\varphi(u)+\frac{\varphi(u)(u^2-1)}{u}=-\frac{\varphi(u)}{u}<0
.
$$
Hence,~$r'<0$ on~$(0,\infty)$; as~$r(0+)=+\infty$ and~$r(\infty)=0$, the equation~$r(u)=1$ has a unique positive root~$u_\star$.
\end{proof}

\begin{proof}[Proof of Theorem~\ref{thm:plugin_sharp}]
Fix~$p\in[p_0,1)$, write~$m=m(p)\le M_0$, $\sigma_p:=\sqrt{p(1-p)}$, and recall~$t_-=n-m$, $t_+=n-m+1$.
\vspace{2mm}

\emph{Step 1: reduction to the two decision times~$t_\pm$.}
Let~$c_0(p_0):=\frac{1}{(M_0+1)(M_0+2)}$ and~$\delta_0:=c_0(p_0)/2$, and let~$E_n^{\delta_0}(p)$ be the event of Lemma~\ref{lem:Etails}, on which~$|\hat p_t-p|\le\delta_0$ for all~$t\ge\lceil n/2\rceil+1$; that lemma gives~$\mathbb P_p[(E_n^{\delta_0}(p))^c]\le Ce^{-cn}$ uniformly in~$p\in[p_0,1)$. We claim that, on~$E_n^{\delta_0}(p)$:

\emph{(a) the plug-in rule does not stop before~$t_-$.} By~(\ref{noearlystop}) it does not stop before~$\lceil n/2\rceil+1$; and for~$\lceil n/2\rceil+1\le t\le t_--1$ we have~$n-t+1\ge m+2$, so stopping would require~$\hat p_t<\frac{1}{m+2}\le p-c_0(p_0)$, contradicting~$|\hat p_t-p|\le\delta_0$.

\emph{(b) If it has not stopped earlier, then the plug-in rule stops on the first success (if any) in~$\{t_++1,\ldots,n\}$.} For such~$t$ we have~$n-t+1\le m-1$, so~$\frac{1}{n-t+1}\ge\frac{1}{m-1}\ge p+c_0(p_0)>\hat p_t$.

Consequently, on~$E_n^{\delta_0}(p)$ the plug-in and oracle rules can differ only through the decisions taken at~$t_-$ and~$t_+$ (recall indeed that the oracle rule stops on the first success, if any, from~$t_+=n-m+1$ onwards). Let then~$\mathcal E_-$ denote the event that the plug-in rule stops at~$t_-$, where the oracle does not, and~$\mathcal E_+$ the event that it fails to stop at~$t_+$, where the oracle does, and let
$$
D_\pm(p)
:=
\mathbb P_p\bigl[\textrm{the oracle wins},\,\mathcal E_\pm\bigr]
-
\mathbb P_p\bigl[\textrm{the plug-in rule wins},\,\mathcal E_\pm\bigr]
$$
be the contribution of the decision at~$t_\pm$ to the deficit. We claim that
\begin{equation}
\label{eq:sharp_two_terms}
V_n(p)-W_n(p)=D_-(p)+D_+(p)+O(e^{-cn})
,
\end{equation}
uniformly in~$p\in[p_0,1)$. To see this, write~$E:=E_n^{\delta_0}(p)$ and denote by~$\mathcal W^{\rm or}$ and~$\mathcal W^{\hat\pi}$ the events that the oracle and the plug-in rule win, respectively. Since~$\mathbb P_p[\mathcal W^{\rm or}\cap E^c]$ and~$\mathbb P_p[\mathcal W^{\hat\pi}\cap E^c]$ are both bounded by~$\mathbb P_p[E^c]\le Ce^{-cn}$,
$$
V_n(p)-W_n(p)
=
\mathbb P_p[\mathcal W^{\rm or}]-\mathbb P_p[\mathcal W^{\hat\pi}]
=
\mathbb P_p[\mathcal W^{\rm or}\cap E]-\mathbb P_p[\mathcal W^{\hat\pi}\cap E]+O(e^{-cn})
.
$$
On~$E$, the two rules stop at the same time off~$\mathcal E_-\cup\mathcal E_+$, so that~$\mathcal W^{\rm or}$ and~$\mathcal W^{\hat\pi}$ coincide there; as~$\mathcal E_-$ and~$\mathcal E_+$ are disjoint, this gives
$$
\mathbb P_p[\mathcal W^{\rm or}\cap E]-\mathbb P_p[\mathcal W^{\hat\pi}\cap E]
=
\sum_{\varepsilon\in\{-,+\}}
\Bigl(
\mathbb P_p[\mathcal W^{\rm or}\cap E\cap\mathcal E_\varepsilon]
-
\mathbb P_p[\mathcal W^{\hat\pi}\cap E\cap\mathcal E_\varepsilon]
\Bigr)
.
$$
Discarding the restriction to~$E$ in each of these four probabilities produces a further~$O(e^{-cn})$, and the resulting terms are~$D_-(p)$ and~$D_+(p)$; this establishes~(\ref{eq:sharp_two_terms}). Steps~2 and~3 below evaluate~$D_-$ and~$D_+$.
\vspace{2mm}

\emph{Step 2: the loss at~$t_-$.}
On~$E_n^{\delta_0}(p)$, the plug-in rule stops at~$t_-$ if and only if~$X_{t_-}=1$ and~$S_{t_-}\le b_-:=\lceil t_-/(m+1)\rceil-1$; writing~$S_{t_-}=S_{t_--1}+X_{t_-}$, this reads~$X_{t_-}=1$ and~$S_{t_--1}\le b_--1$, so that~$\mathcal E_-$ is~$\sigma(X_1,\ldots,X_{t_-})$-measurable and satisfies
$$
\mathbb P_p[\mathcal E_-]
=
p\, \mathbb P_p\bigl[S_{t_--1}\le b_--1\bigr]
,
$$
by independence of~$X_{t_-}$ and~$S_{t_--1}$. On~$\mathcal E_-$ the plug-in rule stops at~$t_-$ and hence wins if and only if~$X_{t_-+1}=\cdots=X_n=0$, an event of probability~$(1-p)^m$ independent of~$\mathcal E_-$; the oracle, which never stops at~$t_-$, wins if and only if exactly one success occurs in~$\{t_-+1,\ldots,n\}$, with probability~$g(m)$. If~$X_{t_-}=0$, or if~$S_{t_--1}>b_--1$, both rules stop at the first success in~$\{t_-+1,\ldots,n\}$ and their outcomes coincide. Hence
$$
D_-(p)
=
p\,\mathbb P_p\bigl[S_{t_--1}\le b_--1\bigr]
\bigl(g(m)-(1-p)^m\bigr)
=
K_-(p)\Delta_-\ \mathbb P_p\bigl[S_{t_--1}\le b_--1\bigr]
,
$$
where the last equality uses~$g(m)-(1-p)^m=(1-p)^{m-1}[mp-(1-p)]=(m+1)(1-p)^{m-1}\Delta_-$ and~$K_-(p)=(m+1)p(1-p)^{m-1}$.
\vspace{2mm}

\emph{Step 3: the loss at~$t_+$.}
Assume~$m\ge2$ (for~$m=1$ one has~$t_+=n$, where the terminal clause makes the plug-in rule stop on a success, so that~$D_+(p):=0$). On~$E_n^{\delta_0}(p)$, the plug-in rule stops at~$t_+$ if and only if~$X_{t_+}=1$ and~$S_{t_+-1}\le b_+-1$, with~$b_+:=\lceil t_+/m\rceil-1$, whereas the oracle stops at~$t_+$ as soon as~$X_{t_+}=1$. If~$X_{t_+}=0$ both continue and coincide. If~$X_{t_+}=1$ and~$S_{t_+-1}>b_+-1$---that is, on~$\mathcal E_+$---the oracle stops at~$t_+$ and wins with probability~$(1-p)^{m-1}$, while the plug-in rule continues and, by Step~1(b), stops at the first success in~$\{t_++1,\ldots,n\}$, winning with probability~$g(m-1)=(m-1)p(1-p)^{m-2}$. Therefore,
$$
D_+(p)
=
p\,\mathbb P_p\bigl[S_{t_+-1}>b_+-1\bigr]\bigl((1-p)^{m-1}-g(m-1)\bigr)
=
K_+(p)\Delta_+\ \mathbb P_p\bigl[S_{t_+-1}\ge b_+\bigr]
,
$$
since~$(1-p)^{m-1}-g(m-1)=(1-p)^{m-2}[1-mp]=m(1-p)^{m-2}\Delta_+$ and~$K_+(p)=mp(1-p)^{m-2}$.
\vspace{2mm}

\emph{Step 4: from~(\ref{eq:sharp_two_terms}) to the constant.}
Fix~$\varepsilon>0$ and let
$$
C_3(p_0):=\sup\{\max(K_-(p),K_+(p)):p\in[p_0,1)\}<\infty.
$$

\emph{(i) Large distances.} Let~$\Lambda\ge1$ and assume that~$\Delta_-\ge \Lambda/\sqrt n$. Write~$N:=t_--1$. Since~$b_-=\lceil t_-/(m+1)\rceil-1<t_-/(m+1)$, the event~$\{S_N\le b_--1\}$ entails
$$
S_N-Np
<
\frac{t_-}{m+1}-1-Np
=
-N\Delta_-+\Bigl(\frac{1}{m+1}-1\Bigr)
\le
-N\Delta_-
,
$$
where we used~$t_-=N+1$ and~$\frac{1}{m+1}\le\frac12$. Hoeffding's inequality  thus gives~$\mathbb P_p[S_N\le b_--1]\le e^{-2N\Delta_-^2}$. Moreover~$N=n-m-1\ge n/2$, because~$m\le M_0$ and~$n\ge2(M_0+1)$, so that, by Step~2,
$$
D_-(p)
=
K_-(p)\Delta_-\,\mathbb P_p[S_N\le b_--1]
\le
C_3(p_0)\,\Delta_-e^{-n\Delta_-^2}
.
$$
Finally, $x\mapsto xe^{-nx^2}$ is nonincreasing on~$[1/\sqrt{2n},\infty)$, an interval that contains~$\Lambda/\sqrt n$ since~$\Lambda\ge1$; evaluating at~$\Delta_-\ge \Lambda/\sqrt n$ therefore yields
$$
D_-(p)
\le
C_3(p_0)\,\frac{\Lambda}{\sqrt n}\,e^{-\Lambda^2}
,
$$
and similarly for~$D_+(p)$. Choose~$\Lambda=\Lambda(\varepsilon)$ so large that this is~$\le\varepsilon/\sqrt n$. This fixes~$\Lambda$ for the rest of the proof; the constants implied by the~$O(\cdot)$'s below may depend on it.

\emph{(ii) At most one small distance.} By Lemma~\ref{lem:cellsep}, $\Delta_-+\Delta_+\ge\frac{1}{M_0(M_0+1)}>2\Lambda/\sqrt n$ for~$n$ large, so at most one of~$\Delta_\pm$ is~$<\Lambda/\sqrt n$; if none is, (\ref{eq:sharp_two_terms}) and~(i) give~$V_n(p)-W_n(p)\le2\varepsilon/\sqrt n+O(e^{-cn})$.

\emph{(iii) One small distance.} Suppose~$\Delta_-<\Lambda/\sqrt n$, the case~$\Delta_+<\Lambda/\sqrt n$ being identical. Then~$D_+(p)\le\varepsilon/\sqrt n$ by~(i), and~$p$ lies within~$\Lambda/\sqrt n$ of~$q:=\frac{1}{m+1}\in[\frac{1}{M_0+1},\frac12]$, hence in a fixed compact~$J\subset(0,1)$ for~$n$ large, on which~$\sigma_p\ge\sigma_0>0$. Since~$\mathbb E|X_1-p|^3\le p(1-p)$, the Berry--Esseen theorem \citep[Chapter~5]{Petrov1995} gives~$A=A(p_0)$ with
$$
\biggl|\mathbb P_p\bigl[S_{t_--1}\le b_--1\bigr]-\Phi\biggl(\frac{b_--1-(t_--1)p}{\sigma_p\sqrt{t_--1}}\biggr)\biggr|\le\frac{A}{\sqrt n}
\qquad\textrm{for all }p\in J
.
$$
Moreover~$b_--1-(t_--1)p=\frac{t_-}{m+1}-t_-p+O(1)=-t_-\Delta_-+O(1)$, so that, $\Phi$ being Lipschitz and~$t_-=n(1+O(1/n))$, the above normal quantity equals~$\Phi(-\Delta_-\sqrt n/\sigma_p)+O(1/\sqrt n)$. Since~$\Delta_-\le \Lambda/\sqrt n$, multiplying by~$K_-(p)\Delta_-$ gives
$$
D_-(p)
\le
K_-(p)\sigma_p\,\frac{1}{\sqrt n}\,\frac{\Delta_-\sqrt n}{\sigma_p}\Phi\Bigl(-\frac{\Delta_-\sqrt n}{\sigma_p}\Bigr)
+
O\Bigl(\frac{1}{n}\Bigr)
\le
\frac{2C_\star K_-(p)\sigma_p}{\sqrt n}+O\Bigl(\frac 1n\Bigr)
,
$$
using~$\sup_{u\ge0}u\Phi(-u)=2C_\star$. Finally, $K_-$ and~$\sigma_\cdot$ being Lipschitz on~$J$ and~$|p-q|<\Lambda/\sqrt n$,
$$
K_-(p)\sigma_p=K_-(q)\sigma_q+O(1/\sqrt n)=\gamma_{m+1}+O(1/\sqrt n)\le\tfrac12+O(1/\sqrt n)
,
$$
where~$K_-(q)\sigma_q=(m+1)q(1-q)^{m-1}\sqrt{q(1-q)}=\gamma_{m+1}$ for~$q=\frac{1}{m+1}$, and Lemma~\ref{lem:gammamax} was used. Hence~$D_-(p)\le C_\star/\sqrt n+O(1/n)$, uniformly in~$p$.

Collecting~(ii) and~(iii) gives~$\limsup_n\sqrt n\sup_{p\in[p_0,1)}(V_n(p)-W_n(p))\le C_\star+2\varepsilon$, and~$\varepsilon\searrow0$ yields the upper bound in~(\ref{eq:plugin_sharp}).
\vspace{2mm}

\emph{Step 5: the matching lower bound and the least favourable sequences.}
Throughout this step we write~$h(u):=u\Phi(-u)$, as in the proof of Lemma~\ref{lem:gammamax}: $h$ is continuous on~$[0,\infty)$, vanishes at~$0$ and at infinity, and attains its maximum~$h(u_\star)=2C_\star$ at the single point~$u_\star$.

\emph{(a) Sufficiency.} Fix~$u>0$ and let~$(p_n)$ be any~$[p_0,1)$-valued sequence with~$\sqrt n\,|p_n-\frac12|\to\frac u2$; write~$\delta_n:=|p_n-\frac12|$, so that~$2\sqrt n\,\delta_n\to u$ and, for~$n$ large, $p_n\in(\max\{p_0,\frac13\},1)\setminus\{\frac12\}$. Both possible positions of~$p_n$ relative to~$\frac12$ produce the same critical time~$n-1$ and the same threshold~$b_n:=\lceil(n-1)/2\rceil-1$:
if~$p_n<\frac12$, then~$m(p_n)=2$ and~$t_+=n-1$, with~$\Delta_+=\delta_n$, $K_+(p_n)=2p_n$ and~$b_+=b_n$, and the plug-in rule errs at~$t_+$ by \emph{failing to stop}, on the event~$\mathcal A_n:=\{S_{n-2}\ge b_n\}$; if~$p_n>\frac12$, then~$m(p_n)=1$ and~$t_-=n-1$, with~$\Delta_-=\delta_n$, $K_-(p_n)=2p_n$ and~$b_-=b_n$, and the plug-in rule errs at~$t_-$ in the opposite direction, by \emph{stopping} although the oracle waits for the terminal time, on the event~$\mathcal A_n:=\{S_{n-2}\le b_n-1\}$. In both cases, Step~3 gives
$$
V_n(p_n)-W_n(p_n)
  \geq  2p_n\delta_n\,\mathbb P_{p_n}[\mathcal A_n]-Ce^{-cn}
.
$$
Slud's inequality \citep[Theorem~2.1]{Slud1977} applies to~$S_{n-2}\sim\mathrm{Bin}(n-2,p_n)$ in the first case, and, upon rewriting~$\mathcal A_n=\{n-2-S_{n-2}\ge n-1-b_n\}$, to~$n-2-S_{n-2}\sim\mathrm{Bin}(n-2,1-p_n)$ in the second; in either case the success parameter is~$\le\frac12$, the relevant threshold lies between the mean and~$n-2$ minus the mean for~$n$ large, and it exceeds that mean by~$(n-2)\delta_n+O(1)=\frac{u\sqrt n}{2}+O(1)$, while the corresponding standard deviation is~$\sqrt{(n-2)p_n(1-p_n)}=\frac{\sqrt n}{2}(1+o(1))$. Hence~$\mathbb P_{p_n}[\mathcal A_n]\ge1-\Phi(u+o(1))=\Phi(-u)+o(1)$, and~$2p_n\delta_n\sqrt n\to\frac u2$ yields
$$
\liminf_{n\to\infty}\sqrt n\bigl(V_n(p_n)-W_n(p_n)\bigr)
  \geq  \tfrac12 h(u)
.
$$
The two cases differ only in the direction of the misclassification, the coefficient~$K_\pm(p_n)=2p_n\to1$, the variance~$\sigma_{p_n}^2\to\frac14$ and the effective sample size~$n-2$ being the same; the barrier is thus two-sided.

Applying the above to~$p_n=\frac12-\frac{u}{2\sqrt n}$ and taking the supremum over~$u>0$ gives
$$
\liminf_{n\to\infty}\sqrt n\sup_{p\in[p_0,1)}\bigl(V_n(p)-W_n(p)\bigr)
  \geq  \tfrac12\sup_{u>0}h(u)
=
C_\star
,
$$
which, together with Step~4, proves~(\ref{eq:plugin_sharp}). Taking~$u=u_\star$ then shows that any sequence with~$\sqrt n\,|p_n-\frac12|\to\frac{u_\star}{2}$ satisfies~$\liminf_n\sqrt n(V_n(p_n)-W_n(p_n))\ge C_\star$; since~(\ref{eq:plugin_sharp}) bounds the corresponding~$\limsup$ by~$C_\star$, such a sequence is least favourable.
\vspace{2mm}

\emph{(b) Necessity.} Conversely, let~$(p_n)$ be an arbitrary~$[p_0,1)$-valued sequence, let~$\frac{1}{k_n}$ be a point of~$\mathcal B_{p_0}$ nearest to~$p_n$, and set~$u_n:=\sqrt n\,\Delta_{p_n}/\sigma_{p_n}$. Fix~$\varepsilon>0$ and let~$\Lambda=\Lambda(\varepsilon)$ be as in Step~4. If~$\Delta_{p_n}\ge\Lambda/\sqrt n$, then Step~4(i)--(ii) gives~$\sqrt n(V_n(p_n)-W_n(p_n))\le2\varepsilon+O(n^{-1/2})$; otherwise, the first inequality in the display of Step~4(iii), combined with~$K_\pm(p_n)\sigma_{p_n}=\gamma_{k_n}+O(n^{-1/2})$ and with the boundedness of~$h$, gives~$\sqrt n(V_n(p_n)-W_n(p_n))\le\gamma_{k_n}h(u_n)+\varepsilon+O(n^{-1/2})$. Since~$h\ge0$, both cases are covered by
$$
\sqrt n\bigl(V_n(p_n)-W_n(p_n)\bigr)
\ \le\
\gamma_{k_n}h(u_n)+2\varepsilon+O(n^{-1/2})
.
$$
Assume now that the left-hand side converges to~$C_\star$. Letting~$n\to\infty$ and then~$\varepsilon\searrow0$ yields~$\limsup_n\gamma_{k_n}h(u_n)\ge C_\star$, whereas~$\gamma_{k_n}h(u_n)\le\gamma_2h(u_\star)=C_\star$ for every~$n$ by Lemma~\ref{lem:gammamax}; therefore~$\gamma_{k_n}h(u_n)\to C_\star$. As~$k_n$ ranges over the finite set~$\{2,\ldots,M_0+1\}$ and~$\gamma_k<\gamma_2=\frac12$ for~$k\ge3$, this forces~$k_n=2$ for~$n$ large, hence~$h(u_n)\to h(u_\star)$ and, $h$ having a unique maximizer and vanishing at both ends of~$[0,\infty)$, $u_n\to u_\star$. (Both failure modes are thereby excluded: a sequence with~$u_n\to0$ leads the plug-in rule to misclassify~$m(p_n)$ with non-vanishing probability, but at a negligible cost, while one with~$u_n\to\infty$ entails a substantial cost that is incurred too rarely.) Finally, $k_n=2$ means that~$\Delta_{p_n}=|p_n-\frac12|$, which tends to~$0$, so that~$\sigma_{p_n}\to\frac12$ and~$\sqrt n\,|p_n-\frac12|=\sigma_{p_n}u_n\to\frac{u_\star}{2}$. This completes the proof of Theorem~\ref{thm:plugin_sharp}.
\end{proof}

We turn to Proposition~\ref{prop:local_at_1k}, which describes both the deficit of the plug-in rule and the local minimax risk at an arbitrary transition point~$\frac1k$. Its proof requires two lemmas. The first identifies the limiting local experiment there and evaluates the affinity that Le~Cam's two-point bound requires; recall that~$\mathrm{TV}$ denotes the total variation distance.

\begin{lem}
\label{lem:LAN}
Fix~$h>0$ and an integer~$k\ge2$, let
$$
p_n^\pm:=\frac1k\pm\frac{h\sigma_k}{\sqrt n},
\quad
\textrm{with } 
\
\sigma_k:=
\sqrt{\frac1k\Big(1-\frac1k\Big)}
,
$$
write~$N:=n-k$, and let~$\mathbb P^\pm_N$ denote the law of~$(X_1,\ldots,X_N)$ under~$\mathbb P_{p_n^\pm}$, with likelihood ratio~$L_N:=d\mathbb P^+_N/d\mathbb P^-_N$. Then, (i) one has
$$
\log L_N
\ \stackrel{d}{\longrightarrow}\
\mathcal N(-2h^2,4h^2)
\qquad
\textrm{under }\mathbb P^-_N
,
$$
so that, as~$n$ diverges to infinity, the local experiments~$(\mathbb P^-_N,\mathbb P^+_N)$ converge to the Gaussian shift experiment~$(\mathcal N(-h,1),\mathcal N(h,1))$; (ii) if~$Q_n^\pm$ denotes the law of~$(X_1,\ldots,X_N,U)$ under~$\mathbb P_{p_n^\pm}$, where~$U$ realizes the possible internal randomization of a rule (independent of the~$X_t$'s, with a \mbox{$p$-free} distribution), then
$$
1-\mathrm{TV}(Q_n^+,Q_n^-)
\to
2\Phi(-h)
,
$$
as~$n$ diverges to infinity.
\end{lem}

\begin{proof}
(i) Write~$S_N:=\sum_{t=1}^N X_t$, and abbreviate~$a:=\frac1k$, $b:=1-\frac1k$, and~$\varepsilon:=\frac{h\sigma_k}{\sqrt n}$. Then,
\begin{eqnarray*}
\log L_N
\!\!&\!\! = \!\! &\!\!
S_N\log\frac{a+\varepsilon}{a-\varepsilon}
+
(N-S_N)\log\frac{b-\varepsilon}{b+\varepsilon}
\\[2mm]
\!\!&\!\! = \!\! &\!\!
2\varepsilon
\Bigl( \frac{S_N}{a}-\frac{N-S_N}{b} \Bigr)
+
O(N\varepsilon^3)
=
\frac{2\varepsilon}{\sigma_k^2}\bigl(S_N-Na\bigr)
+
O\Big(\frac{1}{\sqrt{n}}\Big)
,
\end{eqnarray*}
where we used
$$
\frac1a+\frac1b=\frac{a+b}{ab}=\frac{1}{\sigma_k^2}
\quad
\textrm{ and }
\quad
N\varepsilon^3
=
O\Big(\frac{1}{\sqrt{n}}\Big)
.
$$
Under~$\mathbb P^-_N$, $S_N\sim{\rm Bin}(N,a-\varepsilon)$, so that~$\mathbb E[S_N-Na]=-N\varepsilon$ and~${\rm Var}[S_N-Na]=N\sigma_k^2(1+o(1))$. Since
$$
\frac{2\varepsilon}{\sigma_k^2}(-N\varepsilon)
=-\frac{2h^2N}{n}
\to
-2h^2 \quad
\textrm{ and }
\quad
\Bigl(\frac{2\varepsilon}{\sigma_k^2}\Bigr)^2N\sigma_k^2
=
\frac{4h^2N}{n}
\to
4h^2, 
$$
the Lindeberg--Lévy central limit theorem yields
$$
\log L_N
\ \stackrel{d}{\longrightarrow}\
G\sim\mathcal N(-2h^2,4h^2)
\qquad\textrm{under }\mathbb P^-_N
.
$$
The Gaussian shift experiment~$(\mathcal N(-h,1),\mathcal N(h,1))$ has log-likelihood ratio~$2hZ$, which is~$\mathcal N(-2h^2,4h^2)$-distributed when~$Z\sim\mathcal N(-h,1)$; this establishes~(i).
\vspace{2mm}

(ii) Without loss of generality, we assume that~$n$ is large enough to have~$p_n^\pm\in(0,1)$. Since~$U$ is independent of the~$X_t$'s and its law does not depend on~$p$, we have
$$
\mathrm{TV}(Q_n^+,Q_n^-)
=
\mathrm{TV}(\mathbb P^+_N,\mathbb P^-_N)
.
$$
Now, $\mathbb P^\pm_N$ are supported on the finite set~$\{0,1\}^N$ and, since~$p_n^\pm\in(0,1)$, give a positive mass to each of its points, so that~$L_N(\mathbf x)=\mathbb P^+_N[\{\mathbf x\}]/\mathbb P^-_N[\{\mathbf x\}]$ for any~$\mathbf x\in\{0,1\}^N$. Combining~$\mathrm{TV}(\mathbb P^+_N,\mathbb P^-_N)=\frac12\sum_{\mathbf x}|\mathbb P^+_N[\{\mathbf x\}]-\mathbb P^-_N[\{\mathbf x\}]|$ with $\min(r,s)=\frac12(r+s-|r-s|)$, and then factoring out~$\mathbb P^-_N[\{\mathbf x\}]$, we obtain
$$
1-\mathrm{TV}(\mathbb P^+_N,\mathbb P^-_N)
=
\sum_{\mathbf x\in\{0,1\}^N}\min\bigl(\mathbb P^+_N[\{\mathbf x\}],\mathbb P^-_N[\{\mathbf x\}]\bigr)
=
\mathbb E_{\mathbb P^-_N}[\min(1,L_N)]
.
$$
Since~$x\mapsto\min(1,e^x)$ is bounded and continuous, the weak convergence above yields
$$
1-\mathrm{TV}(\mathbb P^+_N,\mathbb P^-_N) 
\to
 \mathbb E[\min(1,e^{G})]
.
$$
Finally, with~$\mu:=-2h^2$ and~$\sigma:=2h$, we have~$\mu+\frac{\sigma^2}{2}=0$, so that
$$
\mathbb E[\min(1,e^{G})]
=
\mathbb P[G>0]+\mathbb E[e^{G}\mathbb I[G\le0]]
=
\Phi\Bigl(\frac{\mu}{\sigma}\Bigr)+e^{\mu+\frac{\sigma^2}{2}}\Phi\Bigl(\frac{-\mu-\sigma^2}{\sigma}\Bigr)
=
2\Phi(-h)
,
$$
which is the announced limit.
\end{proof}

The second lemma is the analogue, at the boundary~$\frac1k$, of Lemma~\ref{LemmaLowerBound}. Throughout, $t_k:=n-k+1$ denotes the first time of the oracle window~$\{n-k+1,\ldots,n\}$ associated with~$m(p)=k$.

\begin{lem}
\label{lem:switch_k}
Let~$\pi$ be a (possibly randomized) \mbox{$p$-blind} rule with stopping time~$\tau_n^\pi$, and let~$k\ge2$. Then, (i) for any~$p\in(\frac{1}{k+1},\frac1k)$,
$$
V_n(p)-W_n^\pi(p)
  \geq  k(1-p)^{k-2}\Bigl(\frac1k-p\Bigr)\,\mathbb P_p\bigl[X_{t_k}=1,\ \tau_n^\pi\ge t_k+1\bigr]
;
$$
(ii) for any~$p\in(\frac1k,\frac{1}{k-1})$,
$$
V_n(p)-W_n^\pi(p)
  \geq  k(1-p)^{k-2}\Bigl(p-\frac1k\Bigr)\,\mathbb P_p\bigl[X_{t_k}=1,\ \tau_n^\pi=t_k\bigr]
.
$$
\end{lem}

\begin{proof}
The argument is that of Lemma~\ref{LemmaLowerBound}, with~$n-1$ replaced by~$t_k$; we only record the two conditional comparisons, the switching construction and the measurability checks being identical.

(i) Here~$m(p)=k$, so the oracle stops at the first success in~$\{t_k,\ldots,n\}$. On~$B:=\{X_{t_k}=1,\ \tau_n^\pi\ge t_k+1\}\in\mathcal F_{t_k}$, the rule~$\tilde\pi$ that stops at~$t_k$ instead wins if and only if~$X_{t_k+1}=\cdots=X_n=0$, an event of probability~$(1-p)^{k-1}$; whereas~$\pi$, which can then only stop within~$\{t_k+1,\ldots,n\}$, wins with conditional probability at most~$\max_{1\le\ell\le k-1}g(\ell)=g(k-1)$, the maximum being at~$\ell=k-1$ because~$g$ increases up to~$\ell=m(p)=k$ by Lemma~\ref{lem:gfacts}. Hence, on~$B$, the conditional gain of~$\tilde\pi$ over~$\pi$ is at least
$$
(1-p)^{k-1}-(k-1)p(1-p)^{k-2}
=
(1-p)^{k-2}(1-kp)
=
k(1-p)^{k-2}\Bigl(\frac1k-p\Bigr)
\ (\ge0)
.
$$

(ii) Here~$m(p)=k-1$, so the oracle does \emph{not} stop at~$t_k$. On~$B:=\{X_{t_k}=1,\ \tau_n^\pi=t_k\}$, the rule that continues and applies the oracle prescription on~$\{t_k+1,\ldots,n\}$ wins with conditional probability~$g(k-1)$, against~$(1-p)^{k-1}$ for~$\pi$, a gain of
$$
(k-1)p(1-p)^{k-2}-(1-p)^{k-1}
=
(1-p)^{k-2}(kp-1)
=
k(1-p)^{k-2}\Bigl(p-\frac1k\Bigr)
\ (\ge0)
.
$$
In both cases the conditional gain vanishes off~$B$, and taking expectations gives the announced bounds.
\end{proof}

\begin{proof}[Proof of Proposition~\ref{prop:local_at_1k}]
Write~$p_n:=p_{n,k}(u)$ and~$\sigma_p:=\sqrt{p(1-p)}$.
\vspace{2mm}

\emph{(i) The plug-in rule.}
For~$u=0$ we have~$p_n=\frac1k$, so that~$\Delta_+=0$ while~$\Delta_-=\frac1k-\frac{1}{k+1}$ is bounded away from~$0$; by Step~4(i) in the proof of Theorem~\ref{thm:plugin_sharp}, $V_n(p_n)-W_n(p_n)=O(e^{-cn})$ and both sides of~(\ref{eq:local_plugin_risk}) vanish. Let then~$u\ne0$. For~$n$ large, $p_n$ lies in the cell~$[\frac{1}{k+1},\frac1k)$ if~$u<0$ and in~$[\frac1k,\frac{1}{k-1})$ if~$u>0$, so that
\begin{eqnarray*}
m(p_n)=k
&\ \textrm{ and }\ &
\Delta_+=\tfrac1k-p_n=\tfrac{|u|\sigma_k}{\sqrt n}
\qquad\ \ (u<0)
,
\\[2mm]
m(p_n)=k-1
&\ \textrm{ and }\ &
\Delta_-=p_n-\tfrac1k=\tfrac{|u|\sigma_k}{\sqrt n}
\qquad\ \ (u>0)
,
\end{eqnarray*}
while the distance to the other endpoint of the cell stays bounded away from~$0$. By Step~4(i) in that same proof, the corresponding term~$D_\mp(p_n)$ is~$O(e^{-cn})$, so that~(\ref{eq:sharp_two_terms}) gives
$$
V_n(p_n)-W_n(p_n)
=
K(p_n)\,\frac{|u|\sigma_k}{\sqrt n}\,\Pi_n
+
O(e^{-cn})
,
$$
where~$K:=K_+$ and~$\Pi_n:=\mathbb P_{p_n}[S_{t_+-1}\ge b_+]$ if~$u<0$, and~$K:=K_-$ and~$\Pi_n:=\mathbb P_{p_n}[S_{t_--1}\le b_--1]$ if~$u>0$. In either case~$K(p_n)\to(1-\frac1k)^{k-2}$: indeed~$K_+(\frac1k)=k\cdot\frac1k(1-\frac1k)^{k-2}$ with~$m=k$, and~$K_-(\frac1k)=k\cdot\frac1k(1-\frac1k)^{k-2}$ with~$m=k-1$, as already noted after~(\ref{eq:gammak}). Moreover the Berry--Esseen argument of Step~4(iii) there applies, since~$\sqrt n\Delta_\pm=|u|\sigma_k$ is bounded, and yields
$$
\Pi_n
=
\Phi\Bigl(-\frac{\sqrt n\,\Delta_\pm}{\sigma_{p_n}}\Bigr)+O(n^{-1/2})
=
\Phi(-|u|)+o(1)
,
$$
because~$\sqrt n\Delta_\pm/\sigma_{p_n}=|u|\sigma_k/\sigma_{p_n}\to|u|$. Multiplying by~$\sqrt n$ gives
$$
\sqrt n\bigl(V_n(p_n)-W_n(p_n)\bigr)
\to
\Bigl(1-\frac1k\Bigr)^{k-2}\sigma_k |u| \Phi(-|u|)
=
\gamma_k |u| \Phi(-|u|)
,
$$
by the definition~(\ref{eq:gammak}) of~$\gamma_k$. This is~(\ref{eq:local_plugin_risk}).
\vspace{2mm}

\emph{(ii) The local minimax bound.}
Fix~$h>0$, let~$c\ge h$, put~$p_n^\pm:=p_{n,k}(\pm h)$ and
$$
R_n:=\sqrt n\sup_{|u|\le c}\bigl(V_n(p_{n,k}(u))-W_n^{\pi_n}(p_{n,k}(u))\bigr)
.
$$
As in the proof of Theorem~\ref{thm:minimax_sharp}, we may assume that~$\sup_nR_n=:K_0<\infty$.
\vspace{2mm}

\emph{Step 1: the rule cannot stop early.} Let~$q_n:=\mathbb P_{p_n^-}[\tau_n\le n-k]$ and write~$p:=p_n^-$. By~(\ref{eq:Wpi_def}), a win requires~$X_{\tau_n}=1$, so that stopping at a time~$t$ with~$X_t=0$ can only decrease the win probability; since~$\{\tau_n=t,X_t=1\}$ is independent of~$(X_{t+1},\ldots,X_n)$,
$$
\mathbb P_p[\pi_n\textrm{ wins},\, \tau_n\le n-k]
=
\sum_{t=1}^{n-k}\mathbb P_p[\tau_n=t,X_t=1]
(1-p)^{n-t}
\ \le\
(1-p)^{k}q_n
,
$$
because~$(1-p)^{n-t}\le(1-p)^{k}$ for~$t\le n-k$. On~$\{\tau_n\ge t_k\}$, in turn, the rule can only stop within the~$k$ times~$\{t_k,\ldots,n\}$, so that its conditional win probability is at most the oracle value of the last-success problem on~$k$ i.i.d.\ trials, namely~$\max_{1\le\ell\le k}g(\ell)=g(k)=V_n(p)$, since~$m(p)=k$. Hence,
$$
W_n^{\pi_n}(p) \leq q_n(1-p)^{k}+(1-q_n)V_n(p)
,
$$
so that~$V_n(p)-W_n^{\pi_n}(p)\ge q_n
\{
V_n(p)-(1-p)^{k}
\}$. As
$$
V_n(p)-(1-p)^{k}
=
(1-p)^{k-1}\bigl(kp-(1-p)\bigr)
\to
\tfrac1k\bigl(1-\tfrac1k\bigr)^{k-1}>0
$$
as~$p\to\frac1k$, there are~$C(k)$ and~$n_0(k,h)$ such that~$q_n\le C(k)R_n/\sqrt n$ for all~$n\ge n_0(k,h)$; since~$\sup_nR_n<\infty$, this gives~$q_n\to0$.
\vspace{2mm}

\emph{Step 2: reduction to a test at~$t_k$.} Let~$\psi_p:=\mathbb P_p[\tau_n=t_k | X_{t_k}=1]$ denote the conditional probability that the rule stops at the critical time when a success occurs there. Since~$t_k=n-k+1$, the events~$\{\tau_n\le n-k\}$, $\{\tau_n=t_k\}$ and~$\{\tau_n\ge t_k+1\}$ partition the sample space, so that
$$
\mathbb P_p[X_{t_k}=1,\tau_n\ge t_k+1]
=
\mathbb P_p[X_{t_k}=1]
-
\mathbb P_p[X_{t_k}=1,\tau_n\le n-k]
-
\mathbb P_p[X_{t_k}=1,\tau_n=t_k]
.
$$
Here, $\mathbb P_p[X_{t_k}=1]=p$, while~$\{\tau_n\le n-k\}\in\sigma(X_1,\ldots,X_{n-k},U)$ is independent of~$X_{t_k}$, so that~$\mathbb P_p[X_{t_k}=1,\tau_n\le n-k]=p\mathbb P_p[\tau_n\le n-k]$, and~$\mathbb P_p[X_{t_k}=1,\tau_n=t_k]=p\psi_p$ by definition of~$\psi_p$. Therefore,
$$
\mathbb P_p[X_{t_k}=1,\tau_n\ge t_k+1]=p\bigl(1-\psi_p-\mathbb P_p[\tau_n\le n-k]\bigr)
,
\quad
\mathbb P_p[X_{t_k}=1,\tau_n=t_k]
=
p\psi_p
.
$$
Lemma~\ref{lem:switch_k}(i)--(ii), applied at~$p_n^-$ and~$p_n^+$ respectively, therefore gives
$$
V_n(p_n^-)-W_n^{\pi_n}(p_n^-)\ \ge\ A_n^-\bigl(1-\psi_{p_n^-}-q_n\bigr)
,
\qquad
V_n(p_n^+)-W_n^{\pi_n}(p_n^+)\ \ge\ A_n^+\,\psi_{p_n^+}
,
$$
with~$A_n^\pm:=k(1-p_n^\pm)^{k-2}|p_n^\pm-\frac1k|\,p_n^\pm$. Since~$|p_n^\pm-\frac1k|=h\sigma_k/\sqrt n$ and~$p_n^\pm\to\frac1k$,
$$
A_n^\pm
=
k\Bigl(1-\frac1k\Bigr)^{k-2}\frac{h\sigma_k}{\sqrt n}\cdot\frac1k\,(1+o(1))
=
\frac{\gamma_k h}{\sqrt n}\,(1+o(1))
,
$$
again by~(\ref{eq:gammak}). As each of the two deficits is at most~$R_n/\sqrt n$, adding them yields
$$
\frac{2R_n}{\sqrt n}
  \geq  \frac{\gamma_k h}{\sqrt n}\bigl(1+o(1)\bigr)\bigl\{
\psi_{p_n^+}+1-\psi_{p_n^-}-q_n
\bigr\}
.
$$

\emph{Step 3: Le~Cam's bound.} Exactly as in Step~3 of the proof of Theorem~\ref{thm:minimax_sharp}, conditionally on~$\{X_{t_k}=1\}$, the event~$\{\tau_n=t_k\}$ is measurable with respect to~$\sigma(X_1,\ldots,X_{n-k},U)$, so that~$\psi_p$ is the power of a test based on~$(X_1,\ldots,X_{n-k},U)$ and
$$
\psi_{p_n^+}+1-\psi_{p_n^-}\ \ge\ 1-\mathrm{TV}(Q_n^+,Q_n^-) \to 2\Phi(-h)
$$
by Lemma~\ref{lem:LAN}(ii), applied at the boundary~$\frac1k$ with~$N=n-k$. With~$q_n\to0$ from Step~1, this gives~$\liminf_nR_n\ge\frac{\gamma_kh}{2}\cdot2\Phi(-h)=\gamma_kh\Phi(-h)$.
\vspace{2mm}

\emph{Step 4: optimization.} Letting~$c\to\infty$ and taking the supremum over~$h>0$ yields~(\ref{eq:local_minimax_1k}), since~$\sup_{h>0}h\Phi(-h)=2C_\star$ by Lemma~\ref{lem:gammamax}. The bound is attained: for~$c\ge u_\star$, part~(i) gives
$$
\sup_{|u|\le c}\ \lim_{n\to\infty}\sqrt n\bigl(V_n(p_{n,k}(u))-W_n(p_{n,k}(u))\bigr)
=
\sup_{|u|\le c}\gamma_k|u|\Phi(-|u|)
=
2\gamma_kC_\star
$$
for the plug-in rule.
\end{proof}

Theorem~\ref{thm:minimax_sharp} is now a consequence of Theorem~\ref{thm:plugin_sharp} and of Proposition~\ref{prop:local_at_1k} at the boundary~$\frac12$.

\begin{proof}[Proof of Theorem~\ref{thm:minimax_sharp}]
Since~$\hat\pi$ is \mbox{$p$-blind}, Theorem~\ref{thm:plugin_sharp} entails that
$$
\limsup_{n\to\infty}
\ \sqrt n\,
\inf_\pi\sup_{p\in[p_0,1)}\bigl(V_n(p)-W_n^\pi(p)\bigr)
\leq
C_\star
,
$$
and it is therefore sufficient to establish the corresponding lower bound. To do so, let~$(\pi_n)$ be an arbitrary sequence of (possibly randomized) \mbox{$p$-blind} rules and fix~$c>0$. Since~$p_0<\frac12$, we have~$p_{n,2}(u)\in[p_0,1)$ for all~$|u|\le c$ and all~$n$ large, so that
$$
\sqrt n\sup_{p\in[p_0,1)}\bigl(V_n(p)-W_n^{\pi_n}(p)\bigr)
  \geq  \sqrt n\sup_{|u|\le c}\bigl(V_n(p_{n,2}(u))-W_n^{\pi_n}(p_{n,2}(u))\bigr)
.
$$
Taking~$\liminf_n$, then letting~$c\to\infty$, and applying Proposition~\ref{prop:local_at_1k}(ii) with~$k=2$, for which~$\gamma_2=\frac12$ by~(\ref{eq:gammak}), we obtain
$$
\liminf_{n\to\infty}\ \sqrt n\sup_{p\in[p_0,1)}\bigl(V_n(p)-W_n^{\pi_n}(p)\bigr)
  \geq  2\gamma_2C_\star
=
C_\star
,
$$
which, $(\pi_n)$ being arbitrary, provides the announced lower bound.
\end{proof}





\section{Proofs for Section~4}
\label{secProofsSampleSplit}
We first establish the closed-form win probability of~$\pi_a$ announced in Section~\ref{sec:splitting}, which is the explicit form of the compact representation~(\ref{eq:Wa_compact}).

\begin{thm}
\label{thm:exact_Wa}
Let~$a\in(0,1)$ and~$n\ge2$ be such that~$1\le t_a\le n-1$, and let the~$L_j$ be as in~(\ref{eq:La_def}). Then, for any~$p\in(0,1)$,
\begin{equation}
\label{eq:Wa_exact}
W_n^{\pi_a}(p)
=
\sum_{j=0}^{t_a}
\binom{t_a}{j}p^j(1-p)^{t_a-j}
L_jp(1-p)^{L_j-1}
.
\end{equation}
In particular, $p\mapsto W_n^{\pi_a}(p)$ is a polynomial and can be evaluated in~$O(n)$ arithmetic operations.
\end{thm}

\begin{proof}
Condition on~$S_{t_a}=j$, so that~$\hat p^{(a)}=j/t_a$ and~$m(\hat p^{(a)})=\lceil t_a/j\rceil-1=:\hat m_j$ for~$j\ge1$, with~$\hat m_0:=+\infty$. Since~$n-t+1$ is an integer and~$m<r$ is equivalent to~$m\le\lceil r\rceil-1$ for integer~$m$, the condition~$\hat p^{(a)}<1/(n-t+1)$ is equivalent to~$n-t+1\le \hat m_j$, that is, to~$t\ge n-\hat m_j+1$. Hence, on~$\{S_{t_a}=j\}$, the rule~$\pi_a$ stops at the first success (if any) in the terminal block
$$
A_j
=
\bigl\{\max\{t_a+1,\,n-\hat m_j+1\},\ldots,n\bigr\}
\quad
(\textrm{and } A_j=\{n\} \textrm{ if } \hat m_j=0)
,
$$
of cardinality~$|A_j|=L_j$, and it wins if and only if that success is the last one of the whole sequence, that is, if and only if~$A_j$ contains exactly one success. Since~$A_j$ is determined by~$S_{t_a}$ and since~$A_j\subseteq\{t_a+1,\ldots,n\}$, with $(X_t)_{t>t_a}$ independent of~$S_{t_a}$, we obtain
$$
\mathbb P_p[\pi_a \textrm{ wins}|S_{t_a}=j]
=
L_j p (1-p)^{L_j-1}
.
$$
The result follows by summing over~$j$ against the~${\rm Bin}(t_a,p)$ probability mass function of~$S_{t_a}$.
\end{proof}

Throughout this appendix, $m=m(p)=\lceil 1/p\rceil-1$, $M_0:=\lceil 1/p_0\rceil-1$, and~$g$, $K_\pm$, $\Delta_\pm$ are as in~(\ref{eq:gdef})--(\ref{eq:g_gap_plus}); we freely use Lemmas~\ref{lem:gfacts}--\ref{lem:gammamax} of Appendix~\ref{appProofsConstants}.

\begin{proof}[Proof of Theorem~\ref{thm:oracle_split}]
(i) Write~$t_a=\lfloor an\rfloor$ and~$\hat p=\hat p^{(a)}$, and let~$n_0=n_0(p_0,a)$ be large enough that~$n-t_a\ge M_0+1$ and~$t_a\ge1$ for~$n\ge n_0$; for~$n<n_0$, the bound~(\ref{eq:split_pointwise}) holds trivially by enlarging~$C_2$, since the left-hand side is at most~$1$. Let~$n\ge n_0$ and~$p\in[p_0,1)$, and set~$m=m(p)\le M_0$.

By~(\ref{eq:deficit_as_expectation}),
\begin{equation}
\label{eq:split_decomp}
V_n(p)-W_n^{\pi_a}(p)
=
\mathbb E_p[g(m)-g(L_{S_{t_a}})]
=
\sum_{\ell\ne m}
(g(m)-g(\ell))\,\mathbb P_p[L_{S_{t_a}}=\ell]
,
\end{equation}
where every summand is nonnegative, since~$g(m)=\max_{\ell\ge1}g(\ell)$ by Lemma~\ref{lem:gfacts}. It is therefore enough to bound the summands separately.
Recall from~(\ref{eq:La_def}) that~$L_{S_{t_a}}=\max\{1,\min\{n-t_a,m(\hat p)\}\}$, with the convention~$m(0)=+\infty$. Since~$m\le M_0\le n-t_a-1$ and~$m\ge1$, we have~$L_{S_{t_a}}=m$ on the event~$\{m(\hat p)=m\}$, so that only the following three groups of terms contribute to~(\ref{eq:split_decomp}).

\emph{(a) The term~$\ell=m+1$.} It requires~$m(\hat p)=m+1$, hence~$\hat p<\frac{1}{m+1}$, that is, $\hat p-p<-\Delta_-$. By Lemma~\ref{lem:gfacts} and Hoeffding's inequality applied to the~$t_a$ i.i.d.\ Bernoulli variables~$X_1,\ldots,X_{t_a}$,
\begin{eqnarray*}
\bigl(g(m)-g(m+1)\bigr)\mathbb P_p[L_{S_{t_a}}=m+1]
\!\!&\!\! \leq \!\! &\!\!
(m+1)p(1-p)^{m-1}\Delta_-\,e^{-2t_a\Delta_-^2}
\\[2mm]
\!\!&\!\! \leq \!\! &\!\!
C(p_0)\,\Delta_-\,e^{-2t_a\Delta_-^2}
,
\end{eqnarray*}
where we used~$(m+1)p\le(M_0+1)$.

\emph{(b) The term~$\ell=m-1$ (only if~$m\ge2$).} It requires~$\hat p\ge\frac1m$, that is, $\hat p-p\ge\Delta_+$, and Lemma~\ref{lem:gfacts} together with Hoeffding's inequality gives, in the same way,
$$
\bigl(g(m)-g(m-1)\bigr)\mathbb P_p[L_{S_{t_a}}=m-1]
\le
C(p_0)\,\Delta_+\,e^{-2t_a\Delta_+^2}
.
$$

\emph{(c) The remaining terms.} These require either~$|m(\hat p)-m|\ge2$, or~$m(\hat p)>n-t_a$ (so that the clipping at~$n-t_a$ is active), or~$m(\hat p)=0$ (i.e.\ $\hat p=1$). By Lemma~\ref{lem:cellsep} the first case forces~$|\hat p-p|\ge c_0(p_0)$; the second forces~$\hat p<\frac{1}{n-t_a}\le\frac{p_0}{2}$ for~$n$ large, hence~$|\hat p-p|\ge\frac{p_0}{2}$; as for the third, it gives~$L_{S_{t_a}}=1$, which is a term of~(c) only when~$1\notin\{m-1,m,m+1\}$, that is, only when~$m\ge3$, hence only when~$p<\frac13$; there, $\mathbb P_p[\hat p=1]=p^{t_a}<3^{-t_a}$, which is again exponentially small. Since~$g(m)-g(\ell)\le1$ throughout, Hoeffding's inequality bounds the total contribution of these terms by~$C_2(p_0,a)e^{-c(p_0,a)n}$, using~$t_a\ge an-1$.

It remains to combine (a)--(c). By Lemma~\ref{lem:cellsep}, $\max\{\Delta_-,\Delta_+\}\ge c_1(p_0)$; the corresponding term among (a)--(b) is therefore bounded by~$e^{-2t_ac_1(p_0)^2}\le C_2 e^{-cn}$ and can be absorbed into~(c). The other term involves~$\min\{\Delta_-,\Delta_+\}$. Finally, since~$\mathcal B_{p_0}=\{\frac1k:k=2,\ldots,M_0+1\}$ contains~$\frac{1}{m+1}$ for~$m\le M_0$, and contains~$\frac1m$ whenever~$m\ge2$, we have~$\Delta_p=\min\{\Delta_-,\Delta_+\}$ when~$m\ge2$, whereas~$\Delta_p=\Delta_-=p-\frac12$ when~$m=1$ (the value~$\frac{1}{m}=1$ not belonging to~$\mathcal B_{p_0}$); in the latter case the term~(b) is absent. In all cases,
$$
V_n(p)-W_n^{\pi_a}(p)
\le
C_1(p_0)\,\Delta_p\,e^{-2t_a\Delta_p^2}
+
C_2(p_0,a)\,e^{-c(p_0,a)n}
,
$$
which is~(\ref{eq:split_pointwise}).

(ii) Since~$\sup_{x\ge0}xe^{-2t_ax^2}=1/(2\sqrt{e t_a})$, taking the supremum over~$p\in[p_0,1)$ in~(\ref{eq:split_pointwise}) yields
$$
\sup_{p\in[p_0,1)}\bigl(V_n(p)-W_n^{\pi_a}(p)\bigr)
\le
\frac{C_1(p_0)}{2\sqrt{e\,t_a}}+C_2(p_0,a)e^{-c(p_0,a)n}
.
$$
For~$n$ large enough, $t_a=\lfloor an\rfloor\ge an/2$, and the exponential term is absorbed into the first one by enlarging the constant; this establishes~(\ref{eq:split_uniform}).
\end{proof}

\begin{proof}[Proof of Theorem~\ref{thm:split_sharp}]
Throughout, we write $t:=\lfloor an\rfloor$, $\hat p:=\hat p^{(a)}=S_t/t$, $m=m(p)$, $\sigma_p:=\sqrt{p(1-p)}$, and
$$
K_-(p):=(m+1)p(1-p)^{m-1}
,
\qquad
K_+(p):=mp(1-p)^{m-2}
\ \ (m\ge2)
,
$$
so that, by Lemma~\ref{lem:gfacts}, $g(m)-g(m+1)=K_-(p)\Delta_-$ and~$g(m)-g(m-1)=K_+(p)\Delta_+$, with~$\Delta_-=p-\frac{1}{m+1}$ and~$\Delta_+=\frac1m-p$. Since~$\sqrt{t}=\sqrt{an}(1+o(1))$, it is equivalent to prove~(\ref{eq:split_sharp}), as well as the characterization of the least favourable sequences, with~$\sqrt{an}$ replaced by~$\sqrt t$.
\vspace{2mm}

\emph{Part 1: the lower bound.}
Throughout Parts~1 and~3, we write~$h(u):=u\Phi(-u)$, which, by Lemma~\ref{lem:gammamax}, is continuous on~$[0,\infty)$, vanishes at~$0$ and at infinity, and attains its maximum~$h(u_\star)=2C_\star$ at the single point~$u_\star$.

Fix~$u>0$ and let~$(p_n)$ be any~$[p_0,1)$-valued sequence with~$\sqrt t\,|p_n-\frac12|\to\frac u2$; write~$\delta_n:=|p_n-\frac12|$, so that~$2\sqrt t\,\delta_n\to u$ and, for~$n$ large, $p_n\in(\max\{p_0,\frac13\},1)\setminus\{\frac12\}$. Both possible positions of~$p_n$ relative to~$\frac12$ involve the same coefficient~$2p_n$ and the same threshold~$b_n:=\lceil t/2\rceil$. Indeed, by Lemma~\ref{lem:gfacts}: if~$p_n<\frac12$, then~$m(p_n)=2$ and~$g(2)-g(1)=p_n(1-2p_n)=2p_n\delta_n$, while~$\pi_a$ uses~$L_{S_t}=1$ on~$\mathcal A_n:=\{S_t\ge b_n\}$, since~$L_j=1$ if and only if~$j/t\ge\frac12$; if~$p_n>\frac12$, then~$m(p_n)=1$ and~$g(1)-g(2)=p_n(2p_n-1)=2p_n\delta_n$, while~$\pi_a$ uses~$L_{S_t}=2$ on~$\mathcal A_n:=\{\lceil t/3\rceil\le S_t\le b_n-1\}$, since~$L_j=2$ if and only if~$\frac13\le j/t<\frac12$ (recall that~$n-t\ge M_0+1\ge2$ for~$n$ large). Since all terms in~(\ref{eq:split_decomp}) are nonnegative, retaining the relevant one gives, in both cases,
$$
V_n(p_n)-W_n^{\pi_a}(p_n)
  \geq  2p_n\delta_n\,\mathbb P_{p_n}[\mathcal A_n]
.
$$
Moreover, $\mathbb P_{p_n}[S_t<\lceil t/3\rceil]\le e^{-ct}$ by Hoeffding's inequality, so that~$\mathbb P_{p_n}[\mathcal A_n]=\mathbb P_{p_n}[S_t\le b_n-1]-O(e^{-ct})$ in the second case. Slud's inequality \citep[Theorem~2.1]{Slud1977} applies to~$S_t\sim{\rm Bin}(t,p_n)$ in the first case and, upon rewriting~$\{S_t\le b_n-1\}=\{t-S_t\ge t-b_n+1\}$, to~$t-S_t\sim{\rm Bin}(t,1-p_n)$ in the second; in either case the success parameter is~$\le\frac12$, the relevant threshold lies between the mean and~$t$ minus the mean for~$n$ large, and it exceeds that mean by~$t\delta_n+O(1)=\frac{u\sqrt t}{2}+O(1)$, while the corresponding standard deviation is~$\sqrt{tp_n(1-p_n)}=\frac{\sqrt t}{2}(1+o(1))$. Hence~$\mathbb P_{p_n}[\mathcal A_n]\ge1-\Phi(u+o(1))=\Phi(-u)+o(1)$ and, since~$2p_n\delta_n\sqrt t\to\frac u2$,
$$
\liminf_{n\to\infty}\sqrt t\,\bigl(V_n(p_n)-W_n^{\pi_a}(p_n)\bigr)
  \geq  \tfrac12 h(u)
.
$$
The two cases differ only in the direction of the misclassification of~$m(p_n)$ by the frozen estimate; the barrier is thus two-sided. Applying the above to~$p_n=\frac12-\frac{u}{2\sqrt t}$ and taking the supremum over~$u>0$ gives
$$
\liminf_n\sqrt t\sup_{p\in[p_0,1)}(V_n(p)-W_n^{\pi_a}(p))\ge\frac12\sup_{u>0}h(u)=C_\star.
$$
\vspace{0mm}

\emph{Part 2: the upper bound.}
Fix~$\varepsilon>0$. Exactly as in steps~(a)--(c) of the proof of Theorem~\ref{thm:oracle_split}, for all~$n$ large and all~$p\in[p_0,1)$,
\begin{equation}
\label{eq:sharp_reduction}
V_n(p)-W_n^{\pi_a}(p)
\ \le\
T_-(p)+T_+(p)+C_2e^{-cn}
,
\end{equation}
where~$T_-(p):=K_-(p)\Delta_-\,\mathbb P_p[\hat p-p<-\Delta_-]$ and~$T_+(p):=K_+(p)\Delta_+\,\mathbb P_p[\hat p-p\ge\Delta_+]$, with~$T_+:=0$ when~$m=1$. Let~$C_3(p_0):=\sup\{\max(K_-(p),K_+(p)):p\in[p_0,1)\}<\infty$, which is finite since~$m\le M_0$.

\emph{(i) Contributions from distances~$\ge \Lambda/\sqrt t$.} Let~$\Lambda\ge1$. If~$\Delta_-\ge \Lambda/\sqrt t$, Hoeffding's inequality gives $T_-(p)\le C_3(p_0)\Delta_-e^{-2t\Delta_-^2}$; as~$x\mapsto xe^{-2tx^2}$ is nonincreasing on~$[\frac{1}{2\sqrt t},\infty)\ni \Lambda/\sqrt t$, we get
$$
T_-(p) \leq \frac{C_3(p_0)\,\Lambda\,e^{-2\Lambda^2}}{\sqrt t}
,
$$
and similarly for~$T_+$. Choose once and for all~$\Lambda=\Lambda(\varepsilon)$ so large that~$C_3(p_0)Me^{-2\Lambda^2}
\linebreak
\le\varepsilon$; the constants implied by the~$O(\cdot)$'s below may depend on~$\Lambda$.

\emph{(ii) At most one small distance.} By Lemma~\ref{lem:cellsep}, $\Delta_-+\Delta_+\ge\frac{1}{M_0(M_0+1)}$, which exceeds~$2\Lambda/\sqrt t$ for~$n$ large; hence, at most one of~$\Delta_-,\Delta_+$ is~$<\Lambda/\sqrt t$. If none is, then~(\ref{eq:sharp_reduction}) and~(i) give $V_n(p)-W_n^{\pi_a}(p)\le2\varepsilon/\sqrt t+C_2e^{-cn}$.

\emph{(iii) One small distance.} Assume~$\Delta_-<\Lambda/\sqrt t$, the case~$\Delta_+<\Lambda/\sqrt t$ being identical upon exchanging the roles of~$K_-,\Delta_-$ and~$K_+,\Delta_+$. By~(i), $T_+(p)\le\varepsilon/\sqrt t$. Moreover, $p$ lies within~$\Lambda/\sqrt t$ of~$q:=\frac{1}{m+1}\in[\frac{1}{M_0+1},\frac12]$, so that, for~$n$ large, $p$ belongs to the compact interval~$J:=[\frac{1}{2(M_0+1)},\frac34]\subset(0,1)$, on which~$\sigma_p\ge\sigma_0>0$. Since the~$X_i$ are i.i.d.\ Bernoulli$(p)$ with~$\mathbb E|X_1-p|^3\le p(1-p)$, the Berry--Esseen theorem provides~$A=A(p_0)$ with
$$
\sup_{x\in\mathbb R}
\biggl|\mathbb P_p\biggl[\frac{S_t-tp}{\sigma_p\sqrt t}\le x\biggr]-\Phi(x)\biggr|
\ \le\
\frac{A}{\sqrt t}
\qquad\text{for all }p\in J
.
$$
Consequently, $\mathbb P_p[\hat p-p<-\Delta_-]\le\Phi(-\Delta_-\sqrt t/\sigma_p)+A/\sqrt t$, whence, using $\sup_{u\ge0}u\Phi(-u)=2C_\star$ with~$u=\Delta_-\sqrt t/\sigma_p$ and~$\Delta_-\le \Lambda/\sqrt t$,
$$
T_-(p)
\ \le\
\frac{K_-(p)\sigma_p}{\sqrt t}
\cdot
\frac{\Delta_-\sqrt t}{\sigma_p}\Phi\Bigl(-\frac{\Delta_-\sqrt t}{\sigma_p}\Bigr)
+
K_-(p)\Delta_-\frac{A}{\sqrt t}
\ \le\
\frac{2C_\star K_-(p)\sigma_p}{\sqrt t}
+
\frac{C_3(p_0)MA}{t}
.
$$
Finally, $K_-$ and~$\sigma_{\cdot}$ are Lipschitz on~$J$ and~$|p-q|<\Lambda/\sqrt t$, so that
$$
K_-(p)\sigma_p
=
K_-(q)\sigma_q+O(1/\sqrt t)
=
\gamma_{m+1}+O(1/\sqrt t)
\le
\frac12+O(1/\sqrt t)
,
$$
where we used~$K_-(q)\sigma_q=(m+1)q(1-q)^{m-1}\sqrt{q(1-q)}=\gamma_{m+1}$ for~$q=\frac{1}{m+1}$, together with Lemma~\ref{lem:gammamax}. Therefore,
$$
T_-(p) \leq \frac{C_\star}{\sqrt t}+O\Bigl(\frac{1}{t}\Bigr)
,
$$
the~$O(\cdot)$ being uniform in~$p$.

Collecting~(ii) and~(iii), we obtain, for all~$n$ large,
$$
\sup_{p\in[p_0,1)}\bigl(V_n(p)-W_n^{\pi_a}(p)\bigr)
\ \le\
\frac{C_\star+2\varepsilon}{\sqrt t}+O\Bigl(\frac Mt\Bigr)+C_2e^{-cn}
,
$$
so that~$\limsup_n\sqrt t\sup_{p\in[p_0,1)}(V_n(p)-W_n^{\pi_a}(p))\le C_\star+2\varepsilon$. Letting~$\varepsilon\searrow0$ and combining with Part~1 establishes~(\ref{eq:split_sharp}). Together with Part~1 applied with~$u=u_\star$, this also shows that any sequence with~$\sqrt t\,|p_n-\frac12|\to\frac{u_\star}{2}$ is least favourable.
\vspace{2mm}

\emph{Part 3: the least favourable sequences.}
Conversely, let~$(p_n)$ be an arbitrary~$[p_0,1)$-valued sequence, let~$\frac{1}{k_n}$ be a point of~$\mathcal B_{p_0}$ nearest to~$p_n$, and set~$u_n:=\sqrt t\,\Delta_{p_n}/\sigma_{p_n}$. Fix~$\varepsilon>0$ and let~$\Lambda=\Lambda(\varepsilon)$ be as in Part~2. If~$\Delta_{p_n}\ge\Lambda/\sqrt t$, then Part~2(i)--(ii) gives~$\sqrt t(V_n(p_n)-W_n^{\pi_a}(p_n))\le2\varepsilon+O(t^{-1/2})$; otherwise, the first inequality in the display of Part~2(iii), combined with~$K_\pm(p_n)\sigma_{p_n}=\gamma_{k_n}+O(t^{-1/2})$ and with the boundedness of~$h$, gives~$\sqrt t(V_n(p_n)-W_n^{\pi_a}(p_n))\le\gamma_{k_n}h(u_n)+\varepsilon+O(t^{-1/2})$. Since~$h\ge0$, both cases are covered by
$$
\sqrt t\bigl(V_n(p_n)-W_n^{\pi_a}(p_n)\bigr)
\ \le\
\gamma_{k_n}h(u_n)+2\varepsilon+O(t^{-1/2})
.
$$
Assume now that the left-hand side converges to~$C_\star$. Exactly as in Step~5(b) of the proof of Theorem~\ref{thm:plugin_sharp}, letting~$n\to\infty$ and then~$\varepsilon\searrow0$ forces~$\gamma_{k_n}h(u_n)\to C_\star$, hence~$k_n=2$ for~$n$ large and~$u_n\to u_\star$, so that~$\Delta_{p_n}=|p_n-\frac12|\to0$, $\sigma_{p_n}\to\frac12$ and~$\sqrt t\,|p_n-\frac12|=\sigma_{p_n}u_n\to\frac{u_\star}{2}$. This completes the proof of Theorem~\ref{thm:split_sharp}.
\end{proof}


\section{Proofs for Section~5}
\label{secProofsSection6}

Throughout this appendix, we write~$\mathbb P$ and~$\mathbb E$ rather than~$\mathbb P_{p_n}$ and~$\mathbb E_{p_n}$ to keep the notation light.
We start with the proof of the result in~(\ref{eq:Vn-rate}).

\begin{prop}
\label{prop:sparseVn}
Let~$(p_n)$ be a sequence in~$(0,1)$ such that~$p_n\to 0$ and~$np_n\to\infty$. Then, there exists a positive constant~$C$ such that, for all~$n$ with~$n\geq m_n:=\lceil \frac{1}{p_n}\rceil-1$ and $p_n\leq \frac12$, we have
$$
\Big|V_n(p_n)-\frac{1}{e}\Big|
\le
Cp_n
,
$$
where~$V_n(p_n)=m_np_n(1-p_n)^{m_n-1}$.
\end{prop}

\begin{proof}
Fix $n$ such that $n\ge m_n$ and $p_n\le \frac12$. Write $p:=p_n\in(0,\tfrac12]$ and
$m:=m_n=\lceil \tfrac1p\rceil-1$. Since $n\ge m$, the win probability of the \mbox{$p$-oracle} rule is
$V_n(p)=mp(1-p)^{m-1}$.
Using $m=\lceil \tfrac1p\rceil-1$, we obtain
\[
m<\frac1p\le m+1,
\qquad\text{hence}\ \
mp<1\le (m+1)p
.
\]
This yields
\begin{equation}\label{eq:mp-close-1}
0\le 1-mp\le p,
\qquad\text{so that}\ \
|mp-1|\le p.
\end{equation}

\smallskip
We first control $(1-p)^{m-1}$ around $e^{-1}$. For $p\in(0,\tfrac12]$, define
\[
g(p):=\frac{\log(1-p)}{p}.
\]
Using $\log(1-p)=-\sum_{k=1}^\infty\frac{p^k}{k}$, we obtain
\[
0\le -(\log(1-p)+p)
=
\sum_{k=2}^\infty
\frac{p^k}{k}
\le 
\sum_{k=2}^\infty
p^k=\frac{p^2}{1-p}\le 2p^2,
\]
hence
\begin{equation}\label{eq:g-close-1}
|g(p)+1|=\frac{|\log(1-p)+p|}{p}\le 2p,\qquad p\in(0,\tfrac12].
\end{equation}
Also, since $-p/(1-p)\le \log(1-p)\le -p$ for $p\in(0,1)$, we have $-2\le g(p)\le -1$ for $p\in(0,\tfrac12]$.

Now set $\alpha:=p(m-1)=mp-p$. By \eqref{eq:mp-close-1}, $mp\in[1-p,1)$, hence
$\alpha\in[1-2p,1-p)$ and therefore
\begin{equation}\label{eq:alpha-close-1}
|\alpha-1|\le 2p.
\end{equation}
Since $\log(1-p)=pg(p)$, we can write
\[
(1-p)^{m-1}=\exp\!\big((m-1)\log(1-p)\big)=\exp(\alpha g(p)).
\]
Combining \eqref{eq:g-close-1}--\eqref{eq:alpha-close-1} with $|g(p)|\le 2$ for $p\le \tfrac12$, we get
\begin{equation}\label{eq:exp-arg}
|\alpha g(p)+1|
\le |\alpha-1|\,|g(p)|+|g(p)+1|
\le 4p+2p
\le 6p.
\end{equation}
Therefore, using $|e^x-1|\le e^{|x|}|x|$ (this follows from the mean value theorem),
\begin{align}
|(1-p)^{m-1}-e^{-1}|
&=e^{-1}|e^{\alpha g(p)+1}-1|
\le e^{-1}e^{|\alpha g(p)+1|}\,|\alpha g(p)+1|
\nonumber\\
&\le e^{-1}e^{6p}\,6p
\le 6e^{2}p,
\label{eq:pow-rate}
\end{align}
where we used $p\le \tfrac12$ so that $e^{6p}\le e^3$.

Finally,
\[
|V_n(p)-e^{-1}|
\le |mp-1|\,(1-p)^{m-1}+|(1-p)^{m-1}-e^{-1}|
\le p+6e^{2}p
\le Cp
\]
for $p\in(0,\tfrac12]$, with $C:=1+6e^{2}$. Returning to $p=p_n$ yields the claim.
\end{proof}


We first state and prove the uniform law of large numbers announced in Section~\ref{sec:sparse}.

\begin{lem}
\label{lem:uniform-lln-horizon}
Let~$(p_n)$ be a sequence in~$(0,1)$ such that~$p_n\to 0$ and~$np_n\to\infty$.
Consider
\begin{equation}
\label{eq:In-def}
I_n:=\Bigl\{t\in\{1,\dots,n\}:\ t\ge t_n \Bigr\},
\end{equation}
where~$t_n\in\{1,\ldots,n\}$ for any~$n$ and~$n/t_n=O(1)$. 
Then, there exist positive constants~$C,c$ such that
\begin{equation}
\label{eq:uniform-lln}
\mathbb P
\bigg[
\sup_{t\in I_n}\Bigl|\frac{\hat p_t}{p_n}-1\Bigr| > \varepsilon
\bigg]
\le
C
e^{-c\varepsilon^2 np_n}
\end{equation}
for all~$\varepsilon\in(0,1)$ and all~$n$ large enough.
\end{lem}

We turn to its proof, which is needed to establish Theorem~\ref{thm:plugin-eff-sparse}.

\begin{proof}[Proof of Lemma~\ref{lem:uniform-lln-horizon}]
Fix $\varepsilon\in(0,1)$. Define the martingale
$$
M_t:=S_t-tp_n=\sum_{i=1}^t (X_i-p_n),\qquad t=0,1,\dots,n,
$$
with respect to $\mathcal F_t=\sigma(X_1,\dots,X_t)$.
Its increments satisfy $|M_t-M_{t-1}|=|X_t-p_n|\le 1$ a.s., and its predictable quadratic variation is
\begin{equation}
\label{Vn}
Q_t:=\sum_{i=1}^t \mathbb E\bigl[(X_i-p_n)^2| \mathcal F_{i-1}\bigr]
= tp_n(1-p_n)
.
\end{equation}
Freedman's inequality\footnote{We use the convenient maximal form stated as Theorem~1.1 in \cite{Tropp2011FreedmanMatrix}.} yields, for any $a>0$,
$$
\mathbb P\bigg[\max_{1\le t\le n} M_t \ge a\bigg]
\le
\exp\Bigl(-\frac{a^2}{2(Q_n+a/3)}\Bigr),
$$
Since the same argument shows that, for any $a>0$, 
$$
\mathbb P\bigg[\max_{1\le t\le n} (-M_t) \ge a\bigg]
\le
\exp\Bigl(-\frac{a^2}{2(Q_n+a/3)}\Bigr),
$$
we obtain that, still for any $a>0$
\begin{equation}
\label{maxbound}
\mathbb P\bigg[\max_{1\le t\le n} |M_t| \ge a\bigg]
\le
2\exp\Bigl(-\frac{a^2}{2(Q_n+a/3)}\Bigr)
.
\end{equation}

On the event $\{\sup_{t\in I_n}|\hat p_t-p_n|>\varepsilon p_n\}$, there exists $t\in I_n$ such that
$$
|S_t-tp_n|
=
t|\hat p_t-p_n|
>
\varepsilon t p_n
\ge
\varepsilon t_n p_n
.
$$
In other words,
$$
\bigg\{
\sup_{t\in I_n}|\hat p_t-p_n|>\varepsilon p_n
\bigg\}
\subseteq
\bigg\{
\max_{1\le t\le n}|M_t|>\varepsilon t_n p_n
\bigg\}
.
$$
Combining this inclusion with the maximal deviation bound in~(\ref{maxbound}) and using the fact that~$Q_n\leq np_n$ (see~(\ref{Vn})) yields
$$
\mathbb P
\bigg[
\sup_{t\in I_n}\Bigl|\frac{\hat p_t}{p_n}-1\Bigr| > \varepsilon
\bigg]
\le
2\exp
\bigg(
\!
-\frac{\varepsilon^2 t_n^2 p_n^2}{2\bigl(n p_n+\varepsilon t_n p_n/3\bigr)}
\bigg)
\le
2\exp
\bigg(
\!
-\frac{\varepsilon^2 np_n}{2\bigl(n^2/t_n^2 +n/(3t_n)\bigr)}
\bigg)
.
$$
Since~$n/t_n=O(1)$, the result follows.
\end{proof}


We can now prove Theorem~\ref{thm:plugin-eff-sparse}.

\begin{proof}[Proof of Theorem~\ref{thm:plugin-eff-sparse}]
Fix a sequence~$(p_n)$ as in the statement of the theorem and let
\begin{equation}
\label{choiceeps}
\varepsilon_n
:=
\min\bigg\{\frac12,\ \sqrt{\frac{K\log(np_n)}{np_n}}\bigg\}
,
\end{equation}
where the positive constant~$K$ will be chosen later. Note that~$\varepsilon_n\in(0,\tfrac 12]$ 
and $\varepsilon_n\to 0$.
Define the second-half and near-horizon windows
$$
I_n^-:=\biggl\{t\in\{1,\dots,n\}:\ t\ge h_n
:=
\Bigl\lceil \frac{n}{2}\Bigr\rceil+1
\biggr\}
$$
and
$$
I_n^+:=\biggl\{t\in\{1,\dots,n\}:\ t\ge
t_n:= n-\biggl\lfloor \frac{2}{(1-\varepsilon_n)p_n}\biggr\rfloor\biggr\},
$$
along with the corresponding events
$$
E_n^-
:=
\bigg\{
\sup_{t\in I_n^-}\Bigl|\frac{\hat p_t}{p_n}-1\Bigr|\le \varepsilon_n
\bigg\}
\quad
\textrm{ and }
\quad
E_n^+
:=
\bigg\{
\sup_{t\in I_n^+}\Bigl|\frac{\hat p_t}{p_n}-1\Bigr|\le \varepsilon_n
\bigg\}.
$$
Further consider the deterministic times
$$
t^-_n
:=
n+1-\bigg\lceil \frac{1}{(1-\varepsilon_n)p_n}\bigg\rceil
\quad
\textrm{ and }
\quad
t^+_n
:=
n+2-\bigg\lfloor \frac{1}{(1+\varepsilon_n)p_n}\bigg\rfloor
.
$$
Note that~$h_n<t_n<t_n^-<t_n^+<n$ for all~$n$ large enough. 
\vspace{2mm}

We first show that 
\begin{equation}
\label{eq:tau-lower}
E_n^-
\subseteq 
\{\hat\tau_n\ge t_n^-\}
\quad
\textrm{for all }n \textrm{ large enough}
\end{equation}
and that
\begin{equation}
\label{eq:tau-upper}
E_n^+\cap A_n
\subseteq 
\Bigl\{
\hat{\tau}_n
=
\inf\bigl\{t\in\{t^+_n,\dots,n\}: X_t=1\bigr\}
\Bigr\}
,
\end{equation}
where we let~$A_n:=\{\hat\tau_n \ge t_n^+\}$.
In other words, on~$E_n^-$, the plug-in rule stops no earlier than~$t_n^-$ for all~$n$ large enough, whereas, on~$E_n^+$, if it has not stopped yet at~$t^+_n$, then it will stop at the first success in
$\{t^+_n,\dots,n\}$ (if any). 
\vspace{3mm}

\emph{Proof of~(\ref{eq:tau-lower}).}
It follows from~(\ref{noearlystop}) that  
\begin{equation}
\label{eq:no-stop-before-half}
\hat\tau_n\in I_n^- \cup \{+\infty\}
\quad
\textrm{  almost surely.} 
  \end{equation}
Fix then~$t\in I_n^-$ with $t<t_n^-$. Then,
$$
n-t+1>n-t_n^-+1
=
\biggl\lceil\frac{1}{(1-\varepsilon_n)p_n}\biggr\rceil
\ge \frac{1}{(1-\varepsilon_n)p_n},
$$
so that, on $E_n^-$,
$$
\frac{1}{n-t+1}<(1-\varepsilon_n)p_n\le \hat p_t.
$$
This shows that, on~$E_n^-$, the stopping condition cannot hold at any $t\in I_n^-$ with $t<t_n^-$.
Together with \eqref{eq:no-stop-before-half}, this establishes~(\ref{eq:tau-lower}).
\vspace{3mm}


\emph{Proof of~(\ref{eq:tau-upper}).}
For any~$t\geq t^+_n$, we have 
$$
n-t+1
\leq
n-t_n^+ +1
=
\bigg\lfloor \frac{1}{(1+\varepsilon_n)p_n} \bigg\rfloor -1
<
\frac{1}{(1+\varepsilon_n)p_n}
.
$$
Therefore, on~$E_n^+$, such~$t$ provide
$$
\hat p_{t}
\leq
(1+\varepsilon_n)p_n 
< \frac{1}{n-t+1}
.
$$
This shows that, on~$E_n^+\cap A_n$, the rule will stop at the first success in $\{t^+_n,\dots,n\}$ (if any), which establishes~(\ref{eq:tau-upper}).
\vspace{3mm}


We can now proceed with the proof (in the rest of the proof, $C$ and~$c$ are absolute constants that may change from line to line). 
With~$D_n:=\{\text{plug-in wins}\}$, write 
\begin{eqnarray*}
W_n(p_n)
&=&
\mathbb P[D_n|E_n^+\cap A_n ]
\mathbb P[ E_n^+\cap A_n ]
+
\mathbb P[D_n|(E_n^+\cap A_n)^c ]
\mathbb P[ (E_n^+\cap A_n)^c ]
\nonumber 
\\[2mm]
&=&
\mathbb P[D_n|E_n^+\cap A_n ]
+
(
\mathbb P[D_n|(E_n^+\cap A_n)^c ]
-
\mathbb P[D_n|E_n^+\cap A_n ]
)
\mathbb P[ (E_n^+\cap A_n)^c ]
.
\end{eqnarray*}
Since conditional win probabilities are in $[0,1]$, this yields the deterministic bound
\begin{equation}
\label{eq:Wn-decomp-bound}
|W_n(p_n)-\mathbb P[D_n|E_n^+\cap A_n]|
\le
\mathbb P[(E_n^+\cap A_n)^c]
\le
\mathbb P[(E_n^+)^c]+\mathbb P[A_n^c]
.
\end{equation}

From \eqref{eq:tau-upper}, on $E_n^+\cap A_n$ the plug-in rule wins if and only if there is exactly one success in the
residual block $\{t_n^+,\dots,n\}$. Denoting as $m_n^+:=n-t_n^++1$ the length of this block and letting
$N_n\sim \mathrm{Bin}(m_n^+,p_n)$, we thus have
\begin{equation}
\label{eq:success-given-An}
\mathbb P[D_n|E_n^+\cap A_n ]
=
\mathbb P[ N_n=1 ]
=
m_n^+p_n(1-p_n)^{m_n^+-1}
.
\end{equation}

We now compare the right-hand side of \eqref{eq:success-given-An} to $1/e$.
Set $\alpha_n:=m_n^+p_n$. Since
$$
m_n^+=n-t_n^+ + 1 
=
\biggl\lfloor \frac{1}{(1+\varepsilon_n)p_n}\bigg\rfloor-1,
$$
we have, for all $n$ large enough,
$$
\frac{1}{1+\varepsilon_n}-2p_n
\le
\alpha_n
\le
\frac{1}{1+\varepsilon_n}-p_n
,
$$
and therefore
\begin{equation}
\label{eq:alpha-close}
\bigg|\alpha_n-\frac{1}{1+\varepsilon_n}\bigg|
\le
2p_n
.
\end{equation}
In particular, $\alpha_n\to 1$. Using the bound $-p_n/(1-p_n)\le \log(1-p_n)\le -p_n$ for $p_n\in(0,1)$, we obtain
$$
\exp\bigg(\!\!-\frac{(m_n^+-1)p_n}{1-p_n}\bigg)
\le
(1-p_n)^{m_n^+-1}
\le
\exp(-(m_n^+-1)p_n)
.
$$
%
%
Since~$m_n^+p_n^2=\alpha_np_n\to 0$ and~$e^x-1\leq 2x$ for~$x\in[0,1]$, we thus have
$$
|(1-p_n)^{m_n^+-1}-e^{-(m_n^+-1)p_n}|
\le
e^{-(m_n^+-1)p_n}\bigg(\exp\bigg(\frac{(m_n^+-1)p_n^2}{1-p_n}\bigg)-1\bigg)
\le
2 p_n
$$
for all~$n$ large enough. Moreover, since $\alpha_n=(m_n^+-1)p_n+p_n$, we similarly have
$$
|e^{-(m_n^+-1)p_n}-e^{-\alpha_n}|
=
e^{-\alpha_n}\big|e^{p_n}-1\big|
\le
2 p_n
.
$$
Thus, for all~$n$ large enough, we have
\begin{equation}
\label{eq:pow-approx}
|(1-p_n)^{m_n^+-1}-e^{-\alpha_n}|
\le
4 p_n
.
\end{equation}

Next, the map $x\mapsto xe^{-x}$ is Lipschitz on $[\frac{1}{4},2]$, and for all~$n$ large enough we have
$\alpha_n\in[\frac{1}{4},2]$. Thus, letting
$$
f(\varepsilon):=\frac{1}{1+\varepsilon}\exp\Big(\!-\frac{1}{1+\varepsilon}\Big),
$$
we have for all~$n$ large enough
$$
|\alpha_n e^{-\alpha_n}-f(\varepsilon_n)|
\le
C\bigg|\alpha_n-\frac{1}{1+\varepsilon_n}\bigg|
\le
C p_n
,
$$
where we used~\eqref{eq:alpha-close}. Since the fact that~$f$ is~$C^1$ on~$[0,\tfrac 12]$ yields
$
|f(\varepsilon_n)-e^{-1}|
\le
C\varepsilon_n
$,
we thus have
$$
|\alpha_n e^{-\alpha_n}-e^{-1}|
\le
|\alpha_n e^{-\alpha_n}-f(\varepsilon_n)|
+
|f(\varepsilon_n)-e^{-1}|
\leq
Cp_n+C\varepsilon_n
$$
for all~$n$ large enough. Using \eqref{eq:pow-approx} and $\alpha_n\le 2$, it follows that, for all~$n$ large enough,
\begin{eqnarray}
\bigg|\mathbb P[D_n|E_n^+\cap A_n]-\frac{1}{e}\bigg|
\!\!&\!\! \leq \!\! &\!\!
\Big|
m_n^+p_n(1-p_n)^{m_n^+-1}-\frac{1}{e}
\Big|
\nonumber
\\[2mm]
\!\!&\!\! \leq \!\! &\!\!
|\alpha_n((1-p_n)^{m_n^+-1}-e^{-\alpha_n})|
+
|(\alpha_n e^{-\alpha_n}-e^{-1})|
\nonumber
\\[2mm]
\!\!&\!\! \leq \!\! &\!\!
Cp_n+C\varepsilon_n
.
\label{eq:mainterm-rate}
\end{eqnarray}

It remains to control $\mathbb P[(E_n^+)^c]+\mathbb P[A_n^c]$ in \eqref{eq:Wn-decomp-bound}.
By Lemma~\ref{lem:uniform-lln-horizon} applied to $I_n^+$ (note that $n/t_n=O(1)$),
we have for all $n$ large enough
$$
\mathbb P[(E_n^+)^c]
\le
C
e^{-c\varepsilon_n^2 np_n}
.
$$

It remains to control~$\mathbb P[A_n^c]$. From~(\ref{eq:tau-lower}), we have, for all $n$ large enough, that
\begin{eqnarray*}
\lefteqn{
\hspace{-10mm}
E_n^-\cap A_n^c
=
E_n^-\cap \{\hat\tau_n < t_n^+\}
\subseteq
E_n^-\cap \{t_n^-\leq  \hat\tau_n \leq t_n^+-1\}
}
\\[2mm]
& & 
\hspace{1mm}
\subseteq
\big\{
\exists\,t\in\{t_n^-,\dots,t_n^+-1\}\text{ with }X_t=1
\big\}
\subseteq
\Big\{
{\textstyle{\sum_{t=t_n^-}^{t_n^+-1} X_t \ge 1}}
\Big\}.
\end{eqnarray*}
By Markov's inequality,
$$
\mathbb P[E_n^-\cap A_n^c]
\le
\mathbb P\big[
{\textstyle{\sum_{t=t_n^-}^{t_n^+-1} X_t \ge 1}}
\big]
\le
\mathbb E\big[
{\textstyle{\sum_{t=t_n^-}^{t_n^+-1} X_t}}
\big]
=
(t_n^+ - t_n^-)\,p_n.
$$
Consequently,
$$
\mathbb P[A_n^c]
\le
(t_n^+ - t_n^-)p_n
+
\mathbb P[(E_n^-)^c]
.
$$
Since
$$
t_n^+ - t_n^-
=
\bigg\lceil \frac{1}{(1-\varepsilon_n)p_n}\bigg\rceil
-\bigg\lfloor \frac{1}{(1+\varepsilon_n)p_n}\bigg\rfloor + 1
=
\frac{2\varepsilon_n}{(1-\varepsilon_n^2)p_n} + O(1)
,
$$
we have
$$
(t_n^+ - t_n^-)p_n
\le
C \varepsilon_n
+
C p_n
.
$$
Lemma~\ref{lem:uniform-lln-horizon} applied to $I_n^-$ yields
$$
\mathbb P[(E_n^-)^c]
\le
C e^{-c\varepsilon_n^2 np_n}
.
$$
Therefore,
$$
\mathbb P[A_n^c]
\le
C \varepsilon_n
+
C p_n
+
C e^{-c\varepsilon_n^2 np_n}
.
$$
Plugging these bounds into \eqref{eq:Wn-decomp-bound} and combining with \eqref{eq:mainterm-rate} yields, for all $n$
large enough,
\begin{equation}
\label{eq:Wn-rate-final}
\bigg|W_n(p_n)-\frac{1}{e}\bigg|
\le
C\varepsilon_n + Cp_n
+
C e^{-c\varepsilon_n^2 np_n}
.
\end{equation}

Finally, if one picks~$K$ in~(\ref{choiceeps}) so large that~$cK\geq 1$, we have for all~$n$ large enough
$$
e^{-c\varepsilon_n^2 np_n}
\le
\frac{1}{np_n}
\le
C\,\sqrt{\frac{\log(np_n)}{np_n}}
.
$$
Hence, \eqref{eq:Wn-rate-final} gives
$$
\bigg|W_n(p_n)-\frac{1}{e}\bigg|
\le
C_1\sqrt{\frac{\log(np_n)}{np_n}}+C_2p_n
,
$$
for all $n$ large enough.

Moreover, since $n\ge m_n$ and~$p_n\leq 1/2$ for $n$ large enough in this regime, we have (see~(\ref{eq:Vn-rate}))
$$
\bigg|V_n(p_n)-\frac{1}{e}\bigg|
\le
C p_n
,
$$
which finally yields
$$
0
\leq
V_n(p_n)-W_n(p_n)
\le
\bigg|V_n(p_n)-\frac{1}{e}\bigg|+\bigg|W_n(p_n)-\frac{1}{e}\bigg|
\le
C_1\sqrt{\frac{\log(np_n)}{np_n}}+C_2p_n
,
$$
after renaming constants. This completes the proof.
\end{proof}


\begin{proof}[Proof of Theorem~\ref{thm:uniform-etan}]
A necessary condition for a rule to win is that there is at least one success in~$\{1,\dots,n\}$. Thus, 
$
\max(
V_n(p),W_n(p)
)
\le \mathbb P_p[
\,
{\textstyle{\sum_{i=1}^n X_i\ge 1}}
]
\le 
\mathbb E_p
[
\,
{\textstyle{\sum_{i=1}^n X_i}}
]
=
np
$, by Markov's inequality.
It follows directly that
$$
\sup_{p\in(0,\tilde{p}_n]} 
\big(V_n(p)-W_n(p)\big) 
\le
\sup_{p\in(0,\tilde{p}_n]} W_n(p)+\sup_{p\in(0,\tilde{p}_n]} V_n(p)
\le 2n\tilde{p}_n\to 0
.
$$
Therefore, it is sufficient to show that
\begin{equation}
\label{eq:uniform-etan-11B}
\lim_{n\to\infty}
\sup_{p\in[p_n,1)} 
\big(V_n(p)-W_n(p)\big) 
=
 0.
\end{equation}

Assume \emph{ad absurdum} that~\eqref{eq:uniform-etan-11B} fails. Then, there exist
$\varepsilon>0$, a subsequence $(n_k)$ and numbers $r_k\in[p_{n_k},1)$ such that
\begin{equation}
\label{subs}
V_{n_k}(r_k)-W_{n_k}(r_k)
\ge \varepsilon
\qquad\text{for all }k.
\end{equation}
By compactness of $[0,1]$, up to extracting a further subsequence we may assume that $(r_k)$ converges in~$[0,1]$. Denote the limit as~$p_\star$. We consider two cases.
\vspace{3mm}

\emph{Case (a): $p_\star\in(0,1]$.}
Let $p_0:=p_\star/2\in(0,1)$. Then, $r_k\in[p_0,1)$ for all large $k$, so that Theorem~\ref{thm:finite_oracle_ineq_p0} entails that $$
V_{n_k}(r_k)-W_{n_k}(r_k)
\leq
\sup_{p\in[p_0,1)} 
\big(V_{n_k}(p)-W_{n_k}(p)\big)
\to 0
$$
as~$k$ diverges to infinity. This contradicts~(\ref{subs}).
\vspace{3mm}

\emph{Case (b): $p_\star=0$.}
Then $r_k\to 0$ and, since $r_k\ge p_{n_k}$ for any~$k$, we have~$n_k r_k \ge n_k p_{n_k}\to\infty$. Therefore, applying Theorem~\ref{thm:plugin-eff-sparse} along the subsequence~$n=n_k$ with success probability~$r_k$, yields
$
V_{n_k}(r_k)-W_{n_k}(r_k)
\to 0,
$
which again contradicts~(\ref{subs}).
\vspace{3mm}

Since both cases lead to a contradiction, \eqref{eq:uniform-etan-11B} holds, and the result is proved.
\end{proof}


\begin{proof}[Proof of Theorem~\ref{thm:no_uniform_abs}]
Let $(\pi_n)$ be an arbitrary sequence of  (possibly randomized) rules. We realize the possible internal randomization by an auxiliary variable $U$ independent of the $X_t$'s,
and we write $\tau_n=\tau_n(X_1,\dots,X_n;U)$ for the corresponding (possibly randomized) stopping time. \emph{Ad absurdum}, assume that~(\ref{eq:assumption_gap}) holds.

For $t\in\{1,\dots,n\}$, define
$$
\beta_{n,t}
:=
\mathbb P[\tau_n=t|B_{n,t}]
,
\quad
\textrm{ with }
B_{n,t}
:=
\{X_1=\cdots=X_{t-1}=0,\ X_t=1\}
,
$$
where the probability is over the internal randomization~$U$.
Fix~$c\in(0,1)$ and let~$p_n:=c/n$.
Since $p_n<1/n$, the oracle threshold index equals $s_n(p_n)=1$, so that the oracle stops at the first success (if any) and wins if and only if~$S_n=1$. Hence,
\begin{equation}
\label{eq:V_pcn}
V_n(p_n)
=
\mathbb P_{p_n}[S_n=1]=n p_n(1-p_n)^{n-1}
.
\end{equation}
Now, for all $t\in\{1,\dots,n\}$,
$$
A_{n,t}:=\{X_1=\cdots=X_{t-1}=0,\ X_t=1,\ X_{t+1}=\cdots=X_n=0\}
$$
satisfies $\mathbb P_{p_n}[A_{n,t}]=p_n(1-p_n)^{n-1}$, and on $A_{n,t}$ the rule~$\pi_n$ wins if and only if it stops at time~$t$. Because $\{\tau_n=t\}$ is measurable with respect to~$\sigma(X_1,\dots,X_t,U)$ and $\{X_{t+1}=\cdots=X_n=0\}$ is independent
of $\sigma(X_1,\dots,X_t,U)$ under $\mathbb P_{p_n}$, we have
$$
\mathbb P_{p_n}[\tau_n=t, A_{n,t}]
=
\mathbb P_{p_n}[\tau_n=t|B_{n,t}]\,\mathbb P_{p_n}[A_{n,t}]
=\beta_{n,t}\,p_n(1-p_n)^{n-1}.
$$
Therefore,
$$
\mathbb P_{p_n}[\text{$\pi_n$ wins and }S_n=1]
=\sum_{t=1}^n \mathbb P_{p_n}[\tau_n=t, A_{n,t}]
=p_n(1-p_n)^{n-1}\sum_{t=1}^n\beta_{n,t}
.
$$
Also, $\{\pi_n\text{ wins and }S_n\ge2\}\subseteq\{S_n\ge2\}$, so
\begin{equation}\label{eq:W_upper_pcn}
W_n^{\pi_n}(p_n)
\le
p_n(1-p_n)^{n-1}\sum_{t=1}^n\beta_{n,t}+\mathbb P_{p_n}[S_n\ge2].
\end{equation}
Combining \eqref{eq:V_pcn}--\eqref{eq:W_upper_pcn} gives
\begin{eqnarray}
V_n(p_n)-W_n^{\pi_n}(p_n)
\!\!&\!\! \geq \!\! &\!\!
p_n(1-p_n)^{n-1}\gamma_n-\mathbb P_{p_n}[S_n\ge2]
\nonumber
\\[2mm]
\!\!&\!\! = \!\! &\!\!
p_n(1-p_n)^{n-1}\gamma_n
-
\{1-(1-p_n)^n-np_n(1-p_n)^{n-1}\}
,
\label{eq:gap_lower_pcn}
\end{eqnarray}
where we let
$$
\gamma_n
:=
\sum_{t=1}^n(1-\beta_{n,t})\in[0,n]
.
$$
Note that since~$p_n=c/n$, 
\begin{equation}\label{eq:binom_limits_updated}
n p_n(1-p_n)^{n-1}\to c e^{-c}
\quad
\textrm{ and }
\quad
\mathbb P_{p_n}[S_n\ge2]\to 1-(1+c)e^{-c}.
\end{equation}

Assume for a moment that 
$$
\eta
:=
\limsup_{n\to\infty} \frac{\gamma_n}{n}
>
0
.
$$
Then, there exists a subsequence $(n_k)$
 such that $\gamma_{n_k}/n_k\to\eta>0$. Along this subsequence, \eqref{eq:gap_lower_pcn}--\eqref{eq:binom_limits_updated} yield
\begin{eqnarray}
\liminf_{k\to\infty}
\big(
V_{n_k}(p_{n_k})-W_{n_k}^{\pi_{n_k}}(p_{n_k})
\big)
\!\!&\!\! \geq \!\! &\!\!
\eta c e^{-c}-\bigl(1-(1+c)e^{-c}\bigr)
\nonumber
\\[2mm]
\!\!&\!\! = \!\! &\!\!
(1+(1+\eta)c)e^{-c}-1
.
\label{eq:liminf_gap_pcn}
\end{eqnarray}
Fix~$c\in(0,1)$ such that the right-hand side is strictly positive
(since~$\eta>0$, such a~$c$ exists).
%
Then \eqref{eq:liminf_gap_pcn} implies
$$
\liminf_{k\to\infty}
\sup_{p\in(0,1)}
\big(V_n(p)-W_n^{\pi_n}(p)\big)
 \ge
\liminf_{k\to\infty}
\big(V_{n_k}(p_{n_k})-W_{n_k}^{\pi_{n_k}}(p_{n_k})
\big)
 > 0
 .
$$
Since this contradicts~\eqref{eq:assumption_gap}, we must have
\begin{equation}\label{eq:dn_over_n_to_0_updated}
\frac{\gamma_n}{n}\to 0
.
\end{equation}

Let now $p_n:=d/n$ with~$d>1$, and define the event
$$
F_n:=\{\text{$\pi_n$ does not stop at the first success (if any)}\}.
$$
On $B_{n,t}$, the rule~$\pi_n$ stops at time $t$ with conditional probability $\beta_{n,t}$, so that
$$
\mathbb P_{p_n}[F_n\cap B_{n,t}]
=(1-\beta_{n,t})\,\mathbb P_{p_n}[B_{n,t}].
$$
Summing over $t$ and using $\mathbb P_{p_n}[B_{n,t}]=p_n(1-p_n)^{t-1}\le p_n$ gives
\begin{equation}
\label{eq:prob_Fn}
\mathbb P_{p_n}[F_n]\le p_n\sum_{t=1}^n(1-\beta_{n,t})=p_n \gamma_n=\frac{d}{n}\gamma_n.
\end{equation}
If $S_n\ge2$ and the rule~$\pi_n$ wins, then it must have skipped the first success, so that
$\{\pi_n\text{ wins and }S_n\ge2\}\subseteq F_n$. Therefore, \eqref{eq:prob_Fn} yields
\begin{equation}\label{eq:W_le_Sn1_plus_F_updated}
W_n^{\pi_n}(p_n)
\le \mathbb P_{p_n}[\pi_n\text{ wins and }S_n=1]
+
\mathbb P_{p_n}[F_n]
\leq
\mathbb P_{p_n}[S_n=1]
+
\frac{d}{n}\gamma_n
.
\end{equation}
Since~$p_n=d/n$, we have
$$
\mathbb P_{p_n}[S_n=1]
=
np_n(1-p_n)^{n-1}\to d e^{-d}
,
$$
so that \eqref{eq:dn_over_n_to_0_updated} yields
\begin{equation}
\label{eq:limsup_W_d_over_n}
\limsup_{n\to\infty} 
\,
W_n^{\pi_n}(d/n)\le d e^{-d}.
\end{equation}

Now, for $n\ge m(p)=\lceil 1/p\rceil-1$, the win probability of the \mbox{$p$-oracle} rule is
$V_n(p)=m(p)p(1-p)^{m(p)-1}$.
For~$p_n=d/n$, letting $m_n:=m(p_n)=\lceil n/d\rceil-1$, we have
\begin{equation}\label{eq:V_d_over_n_limit_updated}
V_n(p_n)
=
m_np_n(1-p_n)^{m_n-1}\to e^{-1}.
\end{equation}
Combining \eqref{eq:limsup_W_d_over_n}--\eqref{eq:V_d_over_n_limit_updated} gives
$$
\liminf_{n\to\infty}
\big(V_n(p_n)-W_n^{\pi_n}(p_n)\big)
 \ge
  e^{-1}-d e^{-d}
,
$$
 hence
$$
\liminf_{n\to\infty}
\sup_{p\in(0,1)}
\big(V_n(p)-W_n^{\pi_n}(p)\big)
 \ge
\liminf_{n\to\infty}
\big(V_n(d/n)-W_n^{\pi_n}(d/n)\big)
 \ge
  e^{-1}-d e^{-d}
  .
$$
Since~$d>1$ implies that~$e^{-1}-d e^{-d}>0$, this contradicts~\eqref{eq:assumption_gap}. \end{proof}


\end{appendix}


\subsection*{Acknowledgments}
		Davy Paindaveine is also affiliated at the Toulouse School of Economics, Université Toulouse 1 Capitole.
%

\subsection*{Funding}

Davy Paindaveine was supported by the ``Projet de Recherche'' T.0230.24 from the FNRS (Fonds National pour la Recherche Scientifique), Communauté Fran\-çaise de Belgique.
\bibliographystyle{chicago} 
\bibliography{Paper}       

\end{document}